\documentclass[a4paper]{amsart} 
\usepackage{amssymb} 

\usepackage[utf8]{inputenc}
\usepackage[T1]{fontenc} 

\usepackage{microtype}
\usepackage{mathtools}
\mathtoolsset{centercolon} 

\usepackage[shortlabels]{enumitem}

\usepackage{booktabs}

\usepackage{tikz}
\usetikzlibrary{decorations.markings}

\usepackage{pgfplots}
\pgfplotsset{compat=1.13} 

\usepackage[hidelinks]{hyperref} 
\hypersetup{final} 

\makeatletter
\@namedef{subjclassname@2020}{\textup{2020} Mathematics Subject Classification}
\makeatother

\newtheorem{theorem}{Theorem}[section]
\newtheorem{lemma}[theorem]{Lemma}
\newtheorem{corollary}[theorem]{Corollary} 
\newtheorem{proposition}[theorem]{Proposition} 
\newtheorem{conjecture}[theorem]{Conjecture} 

\theoremstyle{definition}
\newtheorem{definition}[theorem]{Definition}

\theoremstyle{remark}
\newtheorem{remark}[theorem]{Remark}

\numberwithin{equation}{section}

\newcommand*{\LongArxivOrShortJournalVersion}[2]{#1} 

\newcommand*{\RR}{\mathbb{R}}
\newcommand*{\CC}{\mathbb{C}}
\newcommand*{\NN}{\mathbb{N}}
\newcommand*{\NNzero}{\mathbb{N}\cup\{0\}}
\newcommand*{\NNone}{\mathbb{N}}

\newcommand*{\ind}[1]{\mathbf{1}_{#1}}
\newcommand*{\indbr}[1]{\ind{\{#1\}}}

\newcommand*{\BA}{\mathcal{B}}

\newcommand*{\weakconst}[1]{w(#1)} 
\newcommand*{\weakconstrestricted}[2]{w(#1|#2)} 
\newcommand*{\weakconstrestrictedbig}[2]{w\bigl(#1\big|#2\bigr)} 

\newcommand*{\fclass}[1]{\mathcal{E}_{#1}} 
\newcommand*{\fclassspec}[1]{\mathcal{F}_{#1}}
\newcommand*{\fclassStar}[1]{\mathcal{E}_{#1}^*} 
\newcommand*{\fclassspecStar}[1]{\mathcal{F}_{#1}^*}

\newcommand*{\fZero}{g}
\newcommand*{\fOne}{h}
\newcommand*{\fStar}{f_{*}}
\newcommand*{\fZeroStar}{g_{*}}
\newcommand*{\fOneStar}{h_{*}}

\newcommand*{\aStar}{a_*}
\newcommand*{\bStar}{b_*}
\newcommand*{\cStar}{c_*}
\newcommand*{\dStar}{d_*}

\newcommand*{\dZero}{B}
\newcommand*{\dOne}{D}
\newcommand*{\dOneStar}{D_*}
\newcommand*{\dZeroStar}{B_*}

\newcommand*{\bGain}{\widehat{b}}
\newcommand*{\dGain}{\widehat{d}}
\newcommand*{\bGainStar}{\widehat{b}_*}
\newcommand*{\dGainStar}{\widehat{d}_*}

\newcommand*{\tZeroStar}{t_{0*}}

\newcommand*{\specialbd}[1]{\Omega_{#1}}
\newcommand*{\specialbdStar}[1]{\Omega_{#1}^*}

\newcommand*{\XXXspecialbd}{W}
\newcommand*{\XXXspecialbdStar}{W_*}
\newcommand*{\XXXspecialbdDenominator}{V}
\newcommand*{\XXXspecialbdDenominatorStar}{V_*}

\newcommand*{\bMin}{b_{\min}}
\newcommand*{\bMax}{b_{\max}}
\newcommand*{\dMin}{d_{\min}}
\newcommand*{\dMax}{d_{\max}}
\newcommand*{\dOpt}{d_{\mathrm{opt}}}
\newcommand*{\bMinStar}{b_{*\!\min}}
\newcommand*{\bMaxStar}{b_{*\!\max}}
\newcommand*{\dMinStar}{d_{*\!\min}}
\newcommand*{\dMaxStar}{d_{*\!\max}}
\newcommand*{\dOptStar}{d_{*\!\operatorname{opt}}}

\newcommand*{\bSh}{b_{\operatorname{sp}}}
\newcommand*{\bShStar}{b_{*\!\operatorname{sp}}}

\newcommand*{\bMaxTilde}{\widetilde{b}_{\max}}
\newcommand*{\bMinTildeStar}{\widetilde{b}_{*\!\min}}

\newcommand*{\bStarWvsWStar}{\mathfrak{b}_*} 

\newcommand*{\xStar}{x}
\newcommand*{\yStar}{y}
\newcommand*{\zStar}{z}
\newcommand*{\hStar}{h}
\newcommand*{\xGain}{\hat{x}}
\newcommand*{\yGain}{\hat{y}}

\newcommand*{\XXXasymptoticDenominator}{V}
\newcommand*{\XXXasymptoticNominator}{U}

\newcommand*{\xOptInfty}{x_\infty}

\newcommand{\zCaseTwo}{3/2}
\newcommand{\zCaseThree}{1/2}

\tikzset{%
	MyLineHelp/.style={%
		black!50!white, loosely dotted}}

\tikzset{%
	MyLineCases/.style={%
		black, densely dashed}}

\tikzset{%
	MyLinePush/.style={%
		black, densely dotted}}

\tikzset{%
	MyLineOpt/.style={%
		very thick}}

\tikzset{->-/.style={decoration={
			markings,
			mark=at position .5 with 
			{\arrow{>}}},postaction={decorate}}}

\tikzset{-<-/.style={decoration={
			markings,
			mark=at position .5 with {\arrow{<}}},postaction={decorate}}}

\tikzset{
	MyArrowsR/.style ={
		decoration={
			markings,
			mark=at position #1 with {\arrow{>}}
		},
		postaction={decorate}
	}
}

\tikzset{
	MyArrowsL/.style ={
		decoration={
			markings,
			mark=at position #1 with {\arrow{<}}
			},
		postaction={decorate}
	}
}

\tikzset{
	MyArrowsRR/.style n args={2}{
		decoration={
			markings,
			mark=at position #1 with {\arrow{>}},
			mark=at position #2 with {\arrow{>}}
		},
		postaction={decorate}
	}
}

\tikzset{
	MyArrowsRL/.style n args={2}{
		decoration={
			markings,
			mark=at position #1 with {\arrow{>}},
			mark=at position #2 with {\arrow{<}}
		},
		postaction={decorate}
	}
}

\tikzset{
	MyArrowsLR/.style n args={2}{
		decoration={
			markings,
			mark=at position #1 with {\arrow{<}},
			mark=at position #2 with {\arrow{>}}
		},
		postaction={decorate}
	}
}

\tikzset{
	MyArrowsLL/.style n args={2}{
		decoration={
		markings,
		mark=at position #1 with {\arrow{<}},
		mark=at position #2 with {\arrow{<}}
		},
		postaction={decorate}
	}
}

\tikzset{
	MyArrowsRRR/.style n args={3}{
		decoration={
			markings,
			mark=at position #1 with {\arrow{>}},
			mark=at position #2 with {\arrow{>}},
			mark=at position #3 with {\arrow{>}}
		},
		postaction={decorate}
	}
}

\tikzset{
	MyArrowsLRR/.style n args={3}{
		decoration={
			markings,
			mark=at position #1 with {\arrow{<}},
			mark=at position #2 with {\arrow{>}},
			mark=at position #3 with {\arrow{>}}
		},
		postaction={decorate}
	}
}

\tikzset{
	MyArrowsRLL/.style n args={3}{
		decoration={
			markings,
			mark=at position #1 with {\arrow{>}},
			mark=at position #2 with {\arrow{<}},
			mark=at position #3 with {\arrow{<}}
		},
		postaction={decorate}
	}
}

\begin{document}
\title[]{Lower bounds for the weak-type constants\\
of the operators $\Lambda_m$}

\author[M. Strzelecki]{Micha\l{} Strzelecki}
\address{University of Warsaw, Institute of Mathematics, Banacha 2, 02--097 Warsaw, Poland.}
\email{michalst@mimuw.edu.pl}

\date{May 9, 2025.}


\begin{abstract}	
The operators $\Lambda_m$ ($m\in\mathbb{N}\cup \{0\}$)
arise when one studies
the  action of the Beurling--Ahlfors transform
on certain radial function subspaces.
It is known that the weak-type $(1,1)$
constant of $\Lambda_0$ is equal to $1/\ln(2)\approx 1.44$.
We construct examples showing that
the weak-type $(1,1)$ constant of $\Lambda_1$ is larger than $1.38$ 
and that 
the weak-type $(1,1)$ constant of $\Lambda_m$ does not tend to $1$ when $m\to\infty$. 
This disproves a~conjecture of Gill [Mich.\ Math.\ J.\ 59 (2010), No.\ 2, 353--363].
We also prove a~companion result for the adjoint operators.
This is the arXiv version of the paper -- it includes some additional discussion in the appendices.
\end{abstract}

\maketitle

\tableofcontents

\section{Introduction}

\subsection{Setting and notation}

We are interested in the value of the weak-type $(1,1)$ constant
of some operators acting on $L^1([0,\infty))$
which arise when one studies the  action of the Beurling--Ahlfors transform
on certain radial function subspaces.
Below we briefly introduce the most important definitions.

Recall that for an operator $T$ acting on $L^1([0,\infty))$ its weak-type $(1,1)$ constant
is the smallest number $\weakconst{T}$ such that the inequality
\begin{equation}
\label{eq:weak-type-def}
\bigl|\{t\in[0,\infty) : |T f(t)|\geq 1\} \bigr| \leq \weakconst{T} \|f\|_{L^1([0, \infty))}
\end{equation}
holds for all functions $f \in {L^1([0, \infty))}$. 
Here $|\cdot|$ stands for the Lebesgue measure.
For a class of functions $\mathcal{F}\subset L^1([0,\infty])$
it will be also convenient to denote
\begin{equation}
\label{eq:weak-type-def-restricted}
\weakconstrestricted{T}{\mathcal{F}} 
\coloneqq
\sup\Bigl\{ \frac{
\bigl|\{t\in[0,\infty) : |T f(t)|\geq 1\} \bigr|}{\|f\|_{L^1([0, \infty))}} :
f\in \mathcal{F},\ 
\|f\|_{L^1([0,\infty])}\neq 0 \Bigr\}.
\end{equation}
Clearly,
\[
\weakconst{T}
 = \weakconstrestrictedbig{T}{L^1([0,\infty])} \geq \weakconstrestricted{T}{\mathcal{F}}.
\]

The Beurling--Ahlfors transform, 
below denoted by $\BA$, is a singular operator 
which plays an important role in the study of quasiconformal mappings
 and partial differential equations
(we refer to, e.g., \cite{MR1294669, MR719167} for more details).
We remark that $\BA$ is bounded on $L^p(\CC)$ for $1<p<\infty$
and of weak-type $(1,1)$, which means that
\begin{equation*}
\weakconst{\BA} 
\coloneqq \sup\Bigl\{ \frac{
	\bigl|\{z\in \CC : |\BA F(z)|\geq 1\} \bigr| }{\|F\|_{L^1(\CC)}} :
F\in L^1(\CC),\ 
\|F\|_{L^1(\CC)}\neq 0 \Bigr\} < \infty,
\end{equation*}
where in this context $|\cdot|$ denotes  the two-dimensional Lebesgue measure on the complex plane~$\CC$.

Let us also recall the following important property:
if $F\in L^1(\CC)$ is of the form 
$F(re^{i\theta}) = f(r^2) e^{-im\theta}$
for some $f\colon[0,\infty)\to\CC$ and $m\in\NNzero$, then
\begin{equation*}
\BA F(z) = - \frac{\bar{z}^2}{|z|^2} \Lambda_m f(|z|^2),
\quad z\in \CC,
\end{equation*}
where  
\begin{align*}
\Lambda_m f(t) \coloneqq \frac{1+m}{t^{1+m/2}} \int_0^t f(s) s^{m/2} ds - f(t),
\quad t\geq 0
\end{align*}
(see \cite[Theorem~2.1]{MR2595549}; this is a consequence of the fact that $\BA \partial_{\bar{z}} = \partial_{z}$, cf.\ \cite{MR0856530}). 
By performing a polar change of coordinates, one can check that 
$\| F \|_{L^1(\CC)} =\pi \| f\|_{L^1([0,\infty))}$ 
and 
\[
\bigl|\{z\in \CC : |\BA F(z)|\geq 1\} \bigr| 
= \pi \bigl|\{t\in[0,\infty) : |\Lambda_m(f)(t)|\geq 1\} \bigr|.
\]
Consequently, the weak-type $(1,1)$ constant of the full Beurling--Ahlfors acting on the whole space $L^1(\CC)$
satisfies
\begin{equation*}
\weakconst{\BA} \geq \sup \bigl\{\weakconst{\Lambda_m} : m\in\NNzero\bigr\}.
\end{equation*}

\subsection{Results}

Ba\~nuelos and Janakiraman~\cite{MR2595549} proved that 
\begin{equation}
\label{eq:BJ-ln2}
\weakconst{\Lambda_0}=1/\ln(2) \approx 1.44
\end{equation}
and conjectured  that also $\weakconst{\BA} = 1/\ln(2)$
(see \cite[Conjecture~2]{MR2595549}).
Their proof was based on a delicate, multi-step optimization procedure.

Gill \cite{MR2677626}  postulated that 
it is not possible
to obtain a better lower bound for $\weakconst{\BA}$ 
by considering the operators $\Lambda_m$ with $m\geq 1$. 
He specifically conjectured that
\begin{equation}
\label{eq:conjecture-Gill}
\weakconst{\Lambda_m} \overset{\text{?}}{=} \frac{m 2^{2/(2+m)}}{(2+m)(2-2^{2/(2+m)})}
\end{equation} 
for $m\in\NNone$ 
(see \cite[Conjecture~1]{MR2677626}; cf.\ also \eqref{eq:reason-for-Gill's-conjecture} below).
The right-hand side of~\eqref{eq:conjecture-Gill} decreases with $m$ 
and tends to $1$ as $m\to\infty$.
On the other hand,
its limit as $m\to 0^+$ is equal to $1/\ln(2)$, 
which agrees with~\eqref{eq:BJ-ln2}. 
For $m=1$ one gets the value 
 $2^{2/3}(3(2-2^{2/3}))^{-1} \approx 1.28$.

Our first main result states that conjecture \eqref{eq:conjecture-Gill} does not hold.

\begin{theorem}
\label{thm:main-lower-bounds}
For every $m\in\NN$ 
there exists a class of functions $\fclassspec{m}\subset L^1([0,\infty])$ 
such that
\begin{equation*}
\weakconstrestricted{\Lambda_1}{\fclassspec{1}}\geq 1.38
\quad \text{and} \quad
\liminf_{m\to \infty} \weakconstrestricted{\Lambda_m}{\fclassspec{m}} \geq 1.37.
\end{equation*}
\end{theorem}

Numerical lower bounds for the weak-type constant of $\Lambda_m$ for $m\in\{1,2,3,4\}$
can be found in Table~\ref{table:small-m} on page~\pageref{table:small-m};
for a slightly more precise asymptotic estimate see Remark~\ref{rem:exact-asymptotic-bound}.
With some additional technical work it is also possible to prove that, e.g.,  $\weakconstrestricted{\Lambda_m}{\fclassspec{m}} \geq 1.34$ for every $m\in\NN$ (see \LongArxivOrShortJournalVersion{Appendix~\ref{APP:crunching}}{Appendix~B in the arXiv version of this paper}).
This lower bound of $1.34$  
is somewhat arbitrary --  
in particular, we believe that $\weakconst{\Lambda_m}$
is equal to $\weakconstrestricted{\Lambda_m}{\fclassspec{m}}$
and that it decreases with $m$
(see Conjecture~\ref{conj:mine} below).

One can also consider the adjoint operator $\Lambda_m^*$, $m\in\NNone$,
which
is defined by
\begin{equation}
\label{eq:def-Lambda_m^star}
\Lambda_m^* f(t) = (1+m) t^{m/2}  \int_t^\infty\frac{ f(s) ds}{s^{1+m/2}} -  f(t) , \quad t\geq 0.
\end{equation}
Since the Beurling--Ahlfors transform is a self-adjoint operator we have the lower bound
\begin{equation*}
\weakconst{\BA} \geq \sup \{\weakconst{\Lambda_m^*} : m\in\NNzero\}.
\end{equation*}

Gill~\cite{MR2677626} proved that
\[
\weakconst{\Lambda_0^*} = \weakconst{\Lambda_0} = 1/\ln(2).
\]
His proof 
was based on a clever reformulation of the problem for $\Lambda_0^*$ 
in terms of $\Lambda_0$ 
and we stress that there is no obvious reason for the values of the weak-type constants 
of $\Lambda_m$ and $\Lambda_m^*$ to be equal.
Before presenting 
our estimates of the weak-type constant 
of $\Lambda_m^*$ for $m\geq 1$, 
let us mention some other related results.

Ba\~{n}uelos and Osękowski~\cite{MR3018958}
gave a probabilistic proof of the equality~\eqref{eq:BJ-ln2}
and even found 
the value of the weak-type $(p,p)$ constant of $\Lambda_0$ for $1\leq p \leq 2$.
In a different article~\cite{MR3189475},
 the same authors
gave an alternative, probabilistic proof of Gill's result.
This proof was based on inequalities for pure-jump martingales 
and surprisingly required completely different methods than the proof
presented in their earlier paper~\cite{MR3018958}.

For the sake of completeness, we remark that the exact values of the $L^p$-norms
of the operators $\Lambda_m$ ($1<p<\infty$, $m\in\NNzero$) 
are also known, see \cite{MR3558516} and the citations therein.
Also, weak-type estimates for the  operator $\Lambda_0$
restricted to the classes of decreasing 
and positive decreasing functions have been recently studied in \cite{MR3868629}.

We continue with our second main result.
As in the articles \cite{MR2677626,MR3189475},
the estimates for $\Lambda_m^*$ do not follow immediately
from their counterparts for $\Lambda_m$
(e.g., by some kind of duality).

\begin{theorem}
	\label{thm:main-lower-bounds-star}
	For every $m\in\NN$ 
	there exists a class of functions $\fclassspecStar{m}\subset L^1([0,\infty])$ 
	such that
	\begin{equation*}
	\weakconstrestricted{\Lambda_1^*}{\fclassspecStar{1}}\geq 1.38
	\quad \text{and} \quad
	\liminf_{m\to \infty} \weakconstrestricted{\Lambda_m^*}{\fclassspecStar{m}} \geq 1.37.
	\end{equation*}
\end{theorem}

\begin{remark}
	\label{rem:W-vs-W-star}
	The classes of functions we construct below satisfy, 
	for every $m\in\NNone$,
	\begin{equation*}
	\weakconstrestricted{\Lambda_m^*}{\fclassspecStar{m}} 	
	= 		\weakconstrestricted{\Lambda_m}{\fclassspec{m}}.	
	\end{equation*}
\end{remark}

\subsection{Organization of the article}

The proof of Theorem~\ref{thm:main-lower-bounds} is presented in Section~\ref{SEC:Lambda} 
and consists of two steps:
\begin{itemize}
\item 
In Section~\ref{sec:initial} 
we introduce the classes of functions  $\fclass{m}\subset L^1([0,\infty))$,
 $m\in\NN$.
They yield the initial lower bound
 $\weakconst{\Lambda_m} \geq\weakconstrestricted{\Lambda_m}{\fclass{m}}$,
but unfortunately the right-hand side is somewhat too complicated to evaluate directly.
 \item
In Section~\ref{sec:guesstimates} we make some educated guesses
 and introduce smaller classes of functions $\fclassspec{m} \subset \fclass{m}$, $m\in\NN$.
Due to some additional constraints  the 
quantity $\weakconstrestricted{\Lambda_m}{\fclassspec{m}}$ is easier to handle,
which allows us to prove the lower bounds from Theorem~\ref{thm:main-lower-bounds}. 
 \end{itemize}

In Section~\ref{sec:proof-thm-2} we present the proof of Theorem~\ref{thm:main-lower-bounds-star}
and a sketch of the proof of Remark~\ref{rem:W-vs-W-star}.
Note that  the statement of Remark~\ref{rem:W-vs-W-star} is quite strong
and in order to prove it a lot of precise technical calculations and observations are needed (\LongArxivOrShortJournalVersion{the full proof is therefore postponed to Appendix~\ref{APP:equality}}{the full proof can be found in Appendix~A of the arXiv version of this article}).

The author believes that
\[
\weakconst{\Lambda_m} 
\overset{?}{=} \weakconstrestricted{\Lambda_m}{\fclassspec{m}}
\overset{?}{=} \weakconst{\Lambda_m^*}
\]
and that this quantity is decreasing in $m$.
\LongArxivOrShortJournalVersion{%
We discuss this conjecture and the motivation behind it
in Section~\ref{sec:conclusions} 
and in the Appendices~\ref{APP:equality}, \ref{APP:crunching}, \ref{APP:interlude},
and  \ref{APP:push}.}{%
We discuss this conjecture and the motivation behind it
in Section~\ref{sec:conclusions} 
and in the appendices 
of the arXiv version of this paper, available on \url{https://arxiv.org/abs/2411.09340v2};
those appendices contain a complete and detailed proof of Remark~\ref{rem:W-vs-W-star}, the proof of the lower bound  $\weakconstrestricted{\Lambda_m}{\fclassspec{m}} \geq 1.34$ for every $m\in\NN$, and some additional asymptotic considerations.}

\subsection{Acknowledgements}

I would like to thank Adam Osękowski for several discussions about the operators $\Lambda_m$ and related topics.

\section{Proof of Theorem~\ref{thm:main-lower-bounds}}

\label{SEC:Lambda}

\subsection{Initial calculations}

\label{sec:initial}

Let us recall the eigenfunctions of $\Lambda_m$.

\begin{lemma}
\label{lem:weak-type-eigenfunctions}
For $m\in\NNzero$ and $\alpha> -1-m/2$, 
the function $t\mapsto t^\alpha$, $t>0$,
is an eigenfunction of $\Lambda_m$ with eigenvalue $(m/2 - \alpha)/(1+\alpha +m/2)$.
\end{lemma}

\begin{proof}
This follows from a direct calculation (note that the function in question is locally integrable after multiplying by the weight $t^{m/2}$).
\end{proof}

In particular, 
for $m\in\NNone$ the constant function $(2+m)/m$ is mapped by $\Lambda_m$ to the constant function equal to $1$,
whereas the function $t\mapsto t^{m/2}$ vanishes under $\Lambda_m$.
By considering the optimal sum of those two functions
one obtains 
\begin{equation}
\label{eq:reason-for-Gill's-conjecture}
\weakconstrestrictedbig{\Lambda_m}{ \bigl\{ t\mapsto \bigl(\tfrac{2+m}{m} - t^{m/2}\bigr)\indbr{t\in[0,b]}  :  b > 0\bigr\} }
=\frac{m 2^{2/(2+m)}}{(2+m)(2-2^{2/(2+m)})} 
\end{equation} 
and arrives at conjecture~\eqref{eq:conjecture-Gill}.

It turns out that one can stitch together the functions $\pm (2+m)/m$ and $t\mapsto t^{m/2}$
differently and produce an example yielding a greater lower bound for the weak-type constant of $\Lambda_m$.
Let us introduce the following class of functions.

\begin{definition}
\label{def:weak-type-extremals-2}
For $m\in\NNone$ and $0<a<b\leq c<d$ 
define  the coefficients 
\begin{align*}
\dZero  &= \dZero(a, m) \coloneqq -\frac{2(1+m)}{ma^{m/2}},\\
\dOne &=  \dOne(a,b,c, m) \coloneqq  \frac{2(1+m)}{m c^{m/2}} +  \frac{2(1+m)}{m c^{m/2}} \Bigl(\frac{b}{c}\Bigr)^{1+m/2}+ \dZero\Bigl(\frac{b}{c}\Bigr)^{1+m}.
\end{align*}
We denote by $\fclass{m}$
the class of all functions $f\colon [0,\infty)\to \RR$
which are of the form
\begin{equation}
\label{eq:weak-type-extremals-2}
f(t)=
\bigl(\frac{2+m}{m} +\dZero t^{m/2}\bigr) \indbr{t\in (a,b]}
+ \bigl(-\frac{2+m}{m} + \dOne t^{m/2}\bigr) \indbr{t\in (c,d]}
\end{equation}
for some $0<a<b\leq c<d$ (where $\dZero  = \dZero(a, m)$,
$\dOne =  \dOne(a,b,c, m)$ are as above).
\end{definition}

We remark that one could also consider $a=0$ (which  is excluded in the formulation): in this case $\dZero$ can take any value and only the constraint on $\dOne$ remains. 

Note that below we usually suppress the dependence on $a, b, c, d, m$
and just write $\dZero$, $\dOne$, and $f$.
If it will be clear that, 
say, $m$, $a$, and $c$ are fixed throughout a proof or section,
then we will abuse the notation slightly and write $\dOne(b)$ 
(instead of $\dOne(a,b,c,m)$ or just $\dOne$)
to stress which parameters are of interest at the given moment.
We shall follow a similar convention for other objects as well.

The next lemma explains the motivation behind Definition~\ref{def:weak-type-extremals-2}.

\begin{lemma}
\label{lem:weak-type-extremals-2}
Suppose that $m\in\NNone$, $0<a<b\leq c<d$, 
and let $f\in\fclass{m}$ be the corresponding  function
of the form~\eqref{eq:weak-type-extremals-2}. 	
Then $\Lambda_m f(t)=1$ for $t\in(a,b)$ and $\Lambda_m f(t)=-1$ for $t\in(c,d)$.
\end{lemma}

\begin{proof}
Fix  $m\in\NNone$ and $0<a<b\leq c<d$.
For $t>0$ denote
\begin{align*}
\fZero(t) &=  \frac{2+m}{m} +\dZero t^{m/2},\\
\fOne(t)  &=  -\frac{2+m}{m} + \dOne t^{m/2},
\end{align*}
where $\dZero$, $\dOne$ are as in Definition~\ref{def:weak-type-extremals-2}.
By Lemma~\ref{lem:weak-type-eigenfunctions} we have $\Lambda_m \fZero(t) = 1 = -\Lambda_m \fOne(t)$ for all $t>0$.

We first claim that $\int_0^a f(s) s^{m/2} ds =\int_0^a \fZero(s) s^{m/2} ds$.
Indeed,  we clearly have $\int_0^a f(s) s^{m/2}ds= 0$.
By Lemma~\ref{lem:weak-type-eigenfunctions} and due to our choice of $\dZero$ we also have
\begin{align*}
\int_0^a \fZero(s) s^{m/2} ds
= \frac{a^{1+m/2}}{1+m} \Bigl(\Lambda_m \fZero(a) + \fZero(a) \Bigr)
&= \frac{a^{1+m/2}}{1+m} \Bigl( 1 + \frac{2+m}{m} + \dZero a^{m/2} \Bigr)\\
&= \frac{2 a^{1+m/2}}{m} +\dZero \frac{a^{1+m}}{1+m} = 0.
\end{align*} 	
Since $f=\fZero$ on $(a,b)$, this first  claim immediately implies that $\int_0^t f(s) s^{m/2} ds =\int_0^t \fZero(s) s^{m/2} ds$ for $t\in [a,b]$
and thus $\Lambda_m f(t) = \Lambda_m \fZero(t) = 1$ for $t\in(a,b)$.

We also claim  that $\int_0^c f(s) s^{m/2} ds= \int_0^c \fOne(s) s^{m/2} ds$.
Indeed, using the preceding claim in the second equality,
we see that
\begin{align*}
\int_0^c f(s) s^{m/2} ds
=
\int_0^b f(s) s^{m/2} ds
&=
\int_0^b \fZero(s) s^{m/2} ds\\
&=
\frac{b^{1+m/2}}{1+m} \Bigl(\Lambda_m \fZero(b) + \fZero(b) \Bigr)\\
&=\frac{b^{1+m/2}}{1+m} \Bigl( \frac{2(1+m)}{m} + \dZero b^{m/2} \Bigr),
\end{align*}
while
\begin{align*}
\int_0^c \fOne(s) s^{m/2} ds
=
\frac{c^{1+m/2}}{1+m} \Bigl(\Lambda_m \fOne(c) + \fOne(c) \Bigr)
=\frac{c^{1+m/2}}{1+m} \Bigl( -\frac{2(1+m)}{m} + \dOne c^{m/2} \Bigr).
\end{align*}
Thus, the claim follows by our choice of $\dOne$. 
Consequently, since $f=\fOne$ on $(c,d)$, we have $ \Lambda_m f(t) = \Lambda_m \fOne(t) = -1$ for $t\in(c,d)$. 
This ends the proof of the lemma.
\end{proof}

\subsection{Guesstimates}

\label{sec:guesstimates}

Ideally,
we would like to find 
\[
\weakconstrestricted{\Lambda_m}{\fclass{m}}
=
\sup\Bigl\{ \frac{
	\bigl|\{t\in[0,\infty) : |\Lambda_m f(t)|\geq 1\} \bigr|}{\|f\|_{L^1([0, \infty))}} :
f\in\fclass{m}\Bigr\}.
\]
for all $m\in\NNone$.
Due to several parameters this task does not seem to be straightforward.
In particular, 
for a function $f\in\fclass{m}$ of the form~\eqref{eq:weak-type-extremals-2} 
it may happen that 
$|\Lambda_m f (t)| \geq 1$ 
also for some $t\notin (a,b) \cup (c,d)$.
An explicit, precise formula 
for $\weakconstrestricted{\Lambda_m}{\fclass{m}}$
can be found (see \LongArxivOrShortJournalVersion{Appendix~\ref{APP:interlude}}{Appendix~C in the arXiv version of this paper}),
but the quantities which appear in its formulation
 are somewhat cumbersome to work with.
Therefore,
we shall restrict our attention to a special case where:
\begin{itemize} 
\item $c=b$,
\item the function $f$ changes sign in the interval $(b,d)$ (in particular, $\dOne >0$, $\lim_{t\to b^+} f(t) < 0 < \lim_{t\to d^-} f(t)$),
\item $\lim_{t\to d^-} f(t) <2$.
\end{itemize}
(The motivation for these choices comes from the considerations of  \LongArxivOrShortJournalVersion{Appendix~\ref{APP:interlude}}{Appendix~C in the arXiv version of this paper}).
By scaling invariance, we can also assume that $a=1$ in~\eqref{eq:weak-type-extremals-2}.
A self-contained, precise definition reads as follows.

\begin{definition}
	\label{def:f-class-spec}
	Let $\fclassspec{m}\subset \fclass{m}$ denote the class of functions
	which 
	are of the form
	\[
	f(t) = \bigl(\frac{2+m}{m} -\frac{2(1+m)}{m} t^{m/2}\bigr) \indbr{t\in (1,b]}
	+  \bigl( -\frac{2+m}{m} + \dOne t^{m/2}\bigr)\indbr{t\in(b,d]},
	\]
	for some $1  < b < d$ 	
	which satisfy
	\begin{equation}
	\label{eq:constraint-for-special-b-d}
	-\frac{2+m}{m} + \dOne(b,m) b^{m/2} < 0 <  -\frac{2+m}{m} + \dOne(b,m) d^{m/2} < 2,
	\end{equation}
	where	
	\[
	\dOne = \dOne(b,m) \coloneqq \frac{2(1+m)}{m} \bigl(2 b^{-m/2}  - 1\bigr) > 0.
	\]
\end{definition}

In order to state the next result and to make other statements and calculations cleaner, 
we need the following notation.
For $m\in\NN$ denote by $\specialbd{m}$
the set of $(b,d)\in\RR^2$ which satisfy all the constraints of Definition ~\ref{def:f-class-spec} (in particular, $1<b<d$;
for the Reader's convenience
this set 
and some of the claims which follow
are depicted in Figure~\ref{fig:m-fixed-optimal-curve} on page~\pageref{fig:m-fixed-optimal-curve}).
The constraints 
$-(2+m)/m + \dOne(b,m) b^{m/2} <0$ and $\dOne(b,m) >0$
imply that $b\in (\bMin, \bMax)$, where
\begin{equation}
\label{eq:def-bMin-bMax}
\bMin = \bMin(m)\coloneqq \Bigl(\frac{2+3m}{2+2m}\Bigr)^{2/m},
\qquad
\bMax = \bMax(m) \coloneqq 2^{2/m}.
\end{equation}
Note that 
$-(2+m)/m + \dOne(\bMin,m) \bMin^{m/2} = 0$.

Additionally, the constraint \eqref{eq:constraint-for-special-b-d} yields
that for fixed $b\in (\bMin, \bMax)$ we must have
$d\in (\dMin, \dMax)$,
where
\begin{align}
\label{eq:def-dMin}
\dMin ^{m/2} &= \dMin(b,m)^{m/2} 
\coloneqq \frac{2+m}{ 2(1+m) (2b^{-m/2}  -1)},\\
\label{eq:def-dMax}
\dMax^{m/2} &= \dMax(b,m)^{m/2} \coloneqq \frac{2+3m}{ 2(1+m) (2b^{-m/2}  -1)}.
\end{align}
Clearly, the functions
$b\mapsto \dMin(b,m)$,
$b\mapsto \dMax(b,m)$
are increasing (in the above specified interval, cf.~Figure~\ref{fig:m-fixed-optimal-curve})
and the definitions also make sense for $b= \bMin$.

Finally, for $m\in\NN$ and $(b,d)\in\specialbd{m}$ 
we denote
\begin{multline}
\label{eq:a=1-b=c-would-like-to-optimize-technical}
\XXXspecialbd(b,d,m)\\
\coloneqq \frac{d-1}{ \frac{m}{2+m} -\frac{2+m}{m}d - \frac{2m}{2+m}b +
\frac{4(1+m)}{m(2+m)}(2b^{-m/2}  -1)d^{1+m/2} 
+2t_0(b,m)},
\end{multline}
where  $t_0 = t_0(b,m)$ is the unique number in $[b,d]$ satisfying
$f(t_0) = -(2+m)/m + \dOne(b,m) t_0 ^{m/2} = 0$,
i.e.,
\begin{equation}
\label{eq:def-t_0-explicit-formula}
t_0 =  t_0(b,m) = \Bigl(\frac{2+m}{2(1+m)}\Bigr)^{2/m} \bigl(2 b^{-m/2} -1\bigr)^{-2/m} = \dMin(b,m).
\end{equation}
The definition of $t_0$ makes sense also for $b=\bMin$.
Note also that it follows from the proof of Lemma~\ref{lem:lower-bound-fixed-m-special-b-c-d-alternate-expression-with-inequality}
below that the denominator on the right-hand side of~\eqref{eq:a=1-b=c-would-like-to-optimize-technical} is positive and that we can consider $(b,d)$ on the boundary of $\specialbd{m}$ as well.

We have the following lower bound
(note that actually equality  holds -- see \LongArxivOrShortJournalVersion{Appendix~\ref{APP:equality}}{Appendix~A in the arXiv version of this paper}).

\begin{lemma}
	\label{lem:lower-bound-fixed-m-special-b-c-d-alternate-expression-with-inequality}
	For $m\in\NNone$,
	\begin{equation*}
	\weakconstrestricted{\Lambda_m}{\fclassspec{m}} 
	  \geq 
	\sup\bigl\{
	\XXXspecialbd(b,d,m) : (b,d)\in\specialbd{m} \bigr\}.
	\end{equation*}
\end{lemma}

\begin{proof}
	Consider $1 < b <d$ satisfying the constraints of  Definition~\ref{def:f-class-spec}
	and the corresponding function  $f\in\fclassspec{m}$. 
	By Lemma~\ref{lem:weak-type-extremals-2},
		\[
			\bigl|\{t\in[0,\infty) : |\Lambda_m f(t)|\geq 1\} \bigr| \geq d-1.
	    \]
		Let us calculate $\|f\|_{L^1([0, \infty))} = \int_1^b |f(t)| dt +\int_c^d |f(t)| dt$.
		Observe that $\lim_{t\to 1^+} f(t) = -1 $, and consequently $f$ is negative and decreases on $(1,b]$.
		Hence,
		\begin{align*}
		\int_1^b |f(t)| dt &= \int_1^b  - \frac{2+m}{m} +\frac{2(1+m)}{m} t^{m/2} dt  \\
		&= - \frac{2+m}{m}(b-1) + \frac{4(1+m)}{m(2+m)} (b^{1+m/2} - 1)\\
		&= \frac{m}{2+m} - \frac{2+m}{m}b + \frac{4(1+m)}{m(2+m)} b^{1+m/2}.
		\end{align*}
	Using the definition of $t_0 = t_0(b,m)$ 
	(twice, in the first and last equality) we can write
	\begin{align*}
	\MoveEqLeft[2]
	\int_b^d |f(t)| dt
	= \int_{b}^{t_0} \frac{2+m}{m} - \dOne t^{m/2}  dt 
	+  \int_{t_0}^d -\frac{2+m}{m} + \dOne t^{m/2}  dt \\
	&= -\frac{2+m}{m}(d-2t_0 + b) + \frac{2\dOne}{2+m} (d^{1+m/2} - 2 t_0^{1+m/2} + b^{1+m/2}) \\
	&=  -\frac{2+m}{m}(d+ b) + \frac{2\dOne}{2+m} (d^{1+m/2} + b^{1+m/2}) +2t_0 \cdot \bigl( \frac{2+m}{m} -  \frac{2\dOne}{2+m} t_0^{m/2} \bigr)  \\
	&=  -\frac{2+m}{m}(d+ b) + \frac{2\dOne}{2+m}(d^{1+m/2} + b^{1+m/2})+ 2t_0.
	\end{align*}
	Recalling the definition of $\dOne = \dOne(b,m)$ yields
	\begin{multline*}
	\frac{2\dOne}{2+m} (d^{1+m/2} + b^{1+m/2}) 
	= \frac{4(1+m)}{m(2+m)}(2b^{-m/2}  -1)(d^{1+m/2} + b^{1+m/2})\\
	=\frac{4(1+m)}{m(2+m)}(2b^{-m/2}  -1)d^{1+m/2} 
	+\frac{8(1+m)}{m(2+m)} b  - \frac{4(1+m)}{m(2+m)} b^{1+m/2}.
	\end{multline*}
	Gathering everything and simplifying the expression for $\int_1^b |f(t)| dt+ \int_{b}^d |f(t)| dt$ 
	yields
	\[
	\weakconstrestricted{\Lambda_m}{\fclassspec{m}} 
	\geq \frac{\bigl|\{t\in[0,\infty) : |\Lambda_m f(t)|\geq 1\} \bigr|}{\|f\|_{L^1([0, \infty))} } \geq 
	\XXXspecialbd(b,d,m).
	\]
	The assertion follows by taking the supremum over $(b,d)\in \specialbd{m}$.
\end{proof}

\begin{corollary}
	\label{cor:lower-bound-asymptotic-special-b-c-d-alternate-expression}
	We have 
	\begin{equation*}
	\liminf_{m\to \infty} \weakconstrestricted{\Lambda_m}{\fclassspec{m}}
	\geq
	\sup\Bigl\{ 
	\frac{x + y}{2e^x - x -4  - y + 2(2-e^x) (e^y + 1) -2\ln(2(2-e^x))}
	\Bigr\},
	\end{equation*}
	where the supremum is taken over
	$x\in [\ln(3/2), \ln(2))$
	and $e^{-y} \leq 2(2-e^x) \leq  3e^{-y}$.
\end{corollary}

\begin{proof}
	Fix $x\in [\ln(3/2), \ln(2))$, $y>0$ such that $e^{-y} < 2(2-e^x) <  3e^{-y}$.
	We shall first show that
	\begin{equation}
	\label{eq:ladida}
	\liminf_{m\to \infty} \weakconstrestricted{\Lambda_m}{\fclassspec{m}}
	\geq 	\frac{x + y }{2e^x - x -2+\int_0^y |-1 + 2(2-e^x) e^s | ds}.
	\end{equation}
	To this end set
	\begin{equation*}
	b = b(m) = e^{2x/m},\quad
	d = d(m) = e^{2(x+y)/m}
	\end{equation*}
	and let $\dOne = \dOne(b,m)$ be as in Definition~\ref{def:f-class-spec}.
	We have $D(b,m) > 0$ (since $x<\ln(2)$)
	and
	$\dOne(b,m)\to 2(2e^{-x} -1)$ for $m\to\infty$.
	Our constraints for $x$, $y$ imply that
		\begin{equation*}
		 -1 + 2(2e^{-x} - 1) e^x < 0 <  -1 + 2(2e^{-x} - 1) e^{x+y} < 2,
		\end{equation*}	
		so the constraint~\eqref{eq:constraint-for-special-b-d} is satisfied for $m$ large enough.
	Thus we can consider a~function $f\in\fclassspec{m}$ of the form given in Definition~\ref{def:f-class-spec}
	for $1 < b < d$ as above.

		From Lemma~\ref{lem:weak-type-extremals-2} we know that 
		\begin{equation*}
		m\bigl|\{ t\geq 0 : |\Lambda_m f (t)| \geq 1 \} \bigr| 
		\geq m(d(m) - 1) \xrightarrow[m\to\infty]{} 2x + 2y. 
		\end{equation*}
	    Moreover,
	    \begin{align*}
		m \int_0^\infty |f(t)| dt 
		&=m\int_{1}^{\exp(2x/m)} \Bigl| \frac{2+m}{m} -  \frac{2(1+m)}{m}  t^{m/2}\Bigr| dt\\
		&\qquad  
		+ m\int_{\exp(2x/m)}^{\exp(2(x+y)/m)} \Bigl| -\frac{2+m}{m} + \dOne t^{m/2}\Bigr| dt\nonumber\\
		&= 2 \int_{0}^{x} \Bigl| \frac{2+m}{m} - \frac{2(1+m)}{m} e^{s} \Bigr| e^{2s/m} ds \\
		&\qquad	+  2\int_{x}^{x+y} \Bigl| -\frac{2+m}{m} + \dOne e^{s} \Bigr| e^{2s/m} ds,
		\end{align*}
		where in the last step we substituted $t = e^{2s/m}$.
		For fixed $x, y$ we pass to the limit with $m\to\infty$.
		The integrals tend to
		\begin{align*}
		&2\int_{0}^{x} \Bigl| 1 - 2 e^{s}\Bigr| ds    
		=  2\int_{0}^{x} 2 e^{s} -1 ds  
		= 2(2e^{x} -x - 2),\\
		&2 \int_{x }^{x+y} \bigl| -1 +  2(2 e^{-x } -1)  e^{s}\bigr| ds
		= 2\int_{0 }^{y} \bigl| -1 +  2(2-e^x) e^{s}\bigr| ds,
		\end{align*}
		respectively; the inequality~\eqref{eq:ladida} follows.

	Denote $z = 2(2-e^x)$. Then $z\in(0,1]$, $-1+z \leq 0 < -1+ze^y$, and hence
	 \begin{align*}
	\int_0^y | -1 + z e^s | ds
	 & = \int_{0}^{-\ln(z)} 1 -  z  e^{s} ds
	 + \int_{ -\ln(z)}^{y} -1 +  z e^{s} ds\\
	 & = - 2\ln(z)  - y  + z (e^y - 2e^{-\ln(z)} + 1)\\
	 & = - 2\ln(2(2-e^x)) - y - 2   + 2(2-e^x)(e^y + 1).
	 \end{align*} 
	 The assertion is obtained by taking the supremum.	
\end{proof}

\begin{table}[b] 
	\centering
	\caption{
		The values of $b$, $d$ given below are close to the optimal ones 
		which attain the supremum from
		Lemma~\ref{lem:lower-bound-fixed-m-special-b-c-d-alternate-expression-with-inequality}
		(and were found using a computer algebra system).
		The number $t_0=t_0(b,m)$ is calculated using the 
		value of $b$ given in the table.
		The value in the fifth column is our lower bound for $\weakconst{\Lambda_m}$.
		For comparison, the last column contains the value of the previous lower bound from \cite[Conjecture~1]{MR2677626} (i.e., the right-hand side of~\eqref{eq:conjecture-Gill}).
	}
	\begin{tabular}{llllll}
		\toprule
		$m$ & $b$  & $d$ & $t_0(b,m)$ & $\XXXspecialbd(b,d,m)$ & Gill's bound \\ 
		\midrule
		1  &  2.157  & 6.623  & 4.29782...  & 1.383...  & 1.282...\\
		2 & 1.566 & 3.284  & 2.40552...  & 1.375... & 1.207... \\
		3 & 1.374  & 2.400  & 1.88345... & 1.373... & 1.163...\\
		4 & 1.279 & 2.003   & 1.64172...  & 1.371...   & 1.134... \\ 	  	  	 
		\bottomrule
	\end{tabular}
	\label{table:small-m}
\end{table}

\begin{proof}[Proof of Theorem~\ref{thm:main-lower-bounds}]
For, e.g., $m\in\{1,2,3,4\}$ we can numerically maximize the function $\XXXspecialbd(b,d,m)$
which appears in Lemma~\ref{lem:lower-bound-fixed-m-special-b-c-d-alternate-expression-with-inequality}.
The resulting near optimal values of $b$ and $d$ can be found in Table~\ref{table:small-m}.

In order to prove the asymptotic lower bound for  $\liminf_{m\to \infty} \weakconst{\Lambda_m}$ we use Corollary~\ref{cor:lower-bound-asymptotic-special-b-c-d-alternate-expression}.
Take  $x = 0.548$, $y= 1.164$. 
Then $x\in [\ln(3/2), \ln(2))$, $e^{-y} \leq 2(2-e^x) \leq 3e^{-y}$.
It follows that 
\begin{align*}
\liminf_{m\to \infty} \weakconstrestricted{\Lambda_m}{\fclassspec{m}}
&\geq \frac{x + y}{2e^x - x -4  - y + 2(2-e^x) (e^y + 1) -2\ln(2(2-e^x))}\\
&\geq 1.37.\qedhere
\end{align*}
 \end{proof}

\begin{remark}
	\label{rem:exact-asymptotic-bound}
	We have 
	\[
	\liminf_{m\to \infty} \weakconstrestricted{\Lambda_m}{\fclassspec{m}}
	\geq (\exp(\xOptInfty) - 1)^{-1},
	\]
	where $\xOptInfty\approx 0.54807758$  
	is the unique root of
	\[
	e^x (1-2x)  - (2-e^x)\ln(2(2-e^x)) = 0
	\]
	in the interval $[\ln(3/2), \ln(2))$.
\end{remark}

\begin{proof}[Proof of Remark~\ref{rem:exact-asymptotic-bound}]
	Take $x\in [\ln(3/2), \ln(2))$
	and $y$ defined by
	$e^x = 2(2-e^x) e^y$.
	Then $y = x - \ln(2(2-e^x))$ and $2(2-e^x) e^y \in [3/2, 2)$, so  Corollary~\ref{cor:lower-bound-asymptotic-special-b-c-d-alternate-expression} yields
	\begin{equation}
	\label{eq:exact-asymptotic-bound-to-maximize}
	\liminf_{m\to \infty} \weakconstrestricted{\Lambda_m}{\fclassspec{m}}
	\geq 	\frac{2x - \ln(2(2-e^x)) }{e^x - 2x  -\ln(2(2-e^x))}
	\end{equation}
	(note that our choice of $y$ is in fact optimal, see \LongArxivOrShortJournalVersion{Appendix~\ref{APP:interlude}}{Appendix~C in the arXiv version of this paper}).
	The derivative of the right-hand side is of the same sign as 
	\begin{align*}
	\MoveEqLeft[4]
	\Bigl(2-\frac{-e^x}{2-e^x} \Bigr)\bigl(e^x - 2x  -\ln(2(2-e^x))\bigr)
	- \bigl(2x  -\ln(2(2-e^x))\bigr)\Bigl(e^x - 2 -\frac{-e^x}{2-e^x} \Bigr)\\
	&= \frac{4-e^x}{2-e^x}\Bigl( e^x - 2x  -\ln(2(2-e^x)) - \bigl(2x  -\ln(2(2-e^x))\bigr)(e^x -1) \Bigr)\\
	&= \frac{4-e^x}{2-e^x}\Bigl( e^x (1-2x)  - (2-e^x)\ln(2(2-e^x)) \Bigr).
	\end{align*}
	Denote $h(x) = e^x (1-2x)  - (2-e^x)\ln(2(2-e^x))$ for $x\in[\ln(3/2), \ln(2))$.
	If  $x\in[\ln(3/2), \ln(2))$,
	then the factor $(4-e^x)/(2-e^x)$ is positive 
	and
	\[
	h'(x) = e^x\bigl(-2x + \ln(2(2-e^x))\bigr) < 0.
	\]
	Since
	\[
	h(\ln(3/2)) = 3/2 - 3 \ln(3/2) > 0 >  2-4\ln(2) =  \lim_{x\to \ln(2)^-} h(x),
	\]
	we conclude that the equation $h(x) = 0$ has exactly one root $x=\xOptInfty$ in the interval $[\ln(3/2),\ln(2))$;
	the right-hand side of \eqref{eq:exact-asymptotic-bound-to-maximize} is maximized at $x=\xOptInfty$.
	Moreover,
	\begin{align*}
	\frac{2\xOptInfty - \ln(2(2-e^{\xOptInfty})) }{e^{\xOptInfty} - 2\xOptInfty  -\ln(2(2-e^{\xOptInfty}))}
	&=
	\frac{2\xOptInfty - \frac{e^{\xOptInfty}(1-2\xOptInfty)}{2-e^{\xOptInfty}} }{e^{\xOptInfty} - 2\xOptInfty  -  \frac{e^{\xOptInfty}(1-2\xOptInfty)}{2-e^{\xOptInfty}}  }\\
	&=
	\frac{4\xOptInfty - e^{\xOptInfty}}{4\xOptInfty e^{\xOptInfty} - 4 \xOptInfty  - e^{2\xOptInfty} + e^{\xOptInfty} }
	=
	\frac{1}{e^{\xOptInfty} -1}.\qedhere
	\end{align*}
\end{proof}

\section{Proof of Theorem~\ref{thm:main-lower-bounds-star} and sketch of proof of Remark~\ref{rem:W-vs-W-star}}

\label{sec:proof-thm-2}

\subsection{Adjoint: initial calculations}

Recall that $\Lambda_m^*$ is given by~\eqref{eq:def-Lambda_m^star}.
The course of events in this and the next subsection  is for the most part very similar 
to the one in Section~\ref{SEC:Lambda}.
Where possible we try to denote corresponding objects or quantities in the same way as above
with an asterisk added somewhere in the sub- or superscript. 

\begin{lemma}
	\label{lem:weak-type-eigenfunctions-star}
	For $m\in\NNzero$ and $\alpha < m/2$, the function $t\mapsto t^\alpha$, $t>0$,
	is an eigenfunction of $\Lambda_m^*$ with eigenvalue $(1+\alpha +m/2)/(m/2 - \alpha)$.	
	In particular, for $m\in \NNone$, $\Lambda_m^*$ maps the constant function $m/(2+m)$ to  the constant function equal to $1$,
	whereas the function $t\mapsto t^{-1-m/2}$ vanishes under $\Lambda_m^*$.
\end{lemma}

\begin{proof} This follows from a direct calculation.
\end{proof}

\begin{definition}
	\label{def:weak-type-extremals-2-star}
	For $m\in\NNone$ 
	and $0< \dStar{} < \cStar{} \leq \bStar{} < \aStar{}$
	define the coefficients
	\begin{align*}
	\dZeroStar{} &= 
	\dZeroStar(\aStar, m)
	\coloneqq -\frac{2(1+m) \aStar^{1+m/2}}{2+m},\\
	\dOneStar{} &= \dOneStar(\aStar, \bStar, \cStar, m) \coloneqq  \frac{2(1+m)\cStar^{1+m/2}}{2+m}  \Bigl( 1 + \Bigl(\frac{\cStar}{\bStar}\Bigr)^{m/2}\Bigr) + \dZeroStar{} \Bigl(\frac{\cStar}{\bStar}\Bigr)^{1+m}.
	\end{align*}
	We denote by $\fclassStar{m}$ the class of functions $\fStar{}\colon [0,\infty)\to\RR$
	which are of the form
	\begin{align}
	\label{eq:weak-type-extremals-2-star}
	\fStar(t) &=
	\bigl(-\frac{m}{2+m} +\dOneStar{} t^{-1-m/2}\bigr) \indbr{t\in(\dStar{},\cStar{}]}\\
		&\quad+
		\bigl(\frac{m}{2+m} + \dZeroStar{} t^{-1-m/2}\bigr) \indbr{t\in(\bStar{},\aStar{}]}
		\nonumber
	\end{align}
	for some  $0< \dStar{} < \cStar{} \leq \bStar{} < \aStar{}$ (where $\dZeroStar{} = \dZeroStar(\aStar,m)$ and $\dOneStar{} = \dOneStar(\aStar,\bStar,\cStar,m)$ are as above).
\end{definition}

\begin{lemma}
	\label{lem:weak-type-extremals-2-star}
	Suppose that $m\in\NNone$, $0< \dStar{} < \cStar{} \leq \bStar{} < \aStar{}$,
	and let $\fStar\in\fclassStar{m}$ be the corresponding function of the form~\eqref{eq:weak-type-extremals-2-star}.
	Then $\Lambda_m^* \fStar(t)= - 1$ for $t\in(\dStar,\cStar)$ and $\Lambda_m^* \fStar(t)=1$ for $t\in(\bStar,\aStar)$.
\end{lemma}

\begin{proof}
	Fix $m\in\NN$ and $0< \dStar < \cStar \leq \bStar  < \aStar$.
	For $t>0$ denote
	\begin{align*}
	\fOneStar(t) &=  -\frac{m}{2+m} +\dOneStar{} t^{-1-m/2},\\
	\fZeroStar(t) &= \frac{m}{2+m} + \dZeroStar t^{-1-m/2},
	\end{align*}
	where $\dOneStar$, $\dZeroStar$ are as in Definition~\ref{def:weak-type-extremals-2-star}. 
	By Lemma~\ref{lem:weak-type-eigenfunctions-star} 
	we have $-\Lambda_m^* \fOneStar(t) = 1 = \Lambda_m^* \fZeroStar(t)$ for all $t>0$.

	We first claim that $\int_{\aStar}^\infty \fStar(s) s^{-1-m/2} ds 
	=\int_{\aStar}^\infty \fZeroStar(s) s^{-1-m/2} ds$.
	Indeed, we clearly we have $\int_{\aStar}^\infty \fStar(s) s^{-1-m/2} ds
	= 0$.
	By Lemma~\ref{lem:weak-type-eigenfunctions-star} and due to our choice of $\dZeroStar$ we also have
	\begin{align*}
	\int_{\aStar}^\infty \fZeroStar(s) s^{-1-m/2} ds
	&= \frac{\aStar^{-m/2}}{1+m} \Bigl(\Lambda_m^* \fZeroStar(\aStar) + \fZeroStar(\aStar) \Bigr)\\
	&= \frac{\aStar^{-m/2}}{1+m} \Bigl( 1 + \frac{m}{2+m} + \dZeroStar \aStar^{-1-m/2} \Bigr)\\
	&= \frac{2\aStar^{-m/2} }{2+m}+ \dZeroStar \frac{\aStar^{-1-m}}{1+m}  = 0.
	\end{align*}
	Since $\fStar=\fZeroStar$ on $(\bStar,\aStar)$, 
	this first claim immediately implies that
	\[
	\int_t^{\aStar} \fStar(s) s^{-1-m/2} ds =\int_t^{\aStar} \fZeroStar(s) s^{-1-m/2} ds
	\]
	 for $t\in [\bStar,\aStar]$
	and $\Lambda_m^* \fStar (t) = \Lambda_m^* \fZeroStar (t) = 1$ for $t\in(\bStar,\aStar)$.

	We also claim  that $\int_{\cStar}^\infty \fStar (s) s^{-1-m/2} ds= \int_{\cStar}^\infty \fOneStar (s) s^{-1-m/2} ds$.
	Indeed, using the first claim in the second equality, we see that
	\begin{align*}
	\int_{\cStar}^\infty \fStar (s) s^{-1-m/2} ds
	=
	\int_{\bStar}^\infty \fStar(s) s^{-1-m/2} ds
	&=
	\int_{\bStar}^\infty \fZeroStar(s) s^{-1-m/2} ds\\
	&=
	\frac{\bStar^{-m/2}}{1+m} \Bigl(\Lambda_m^* \fZeroStar (\bStar) + \fZeroStar(\bStar) \Bigr)\\
	&=\frac{\bStar^{-m/2}}{1+m} \Bigl( \frac{2(1+m)}{2+m} + \dZeroStar \bStar^{-1-m/2} \Bigr),
	\end{align*}
	while
	\begin{align*}
	\int_{\cStar}^\infty \fOneStar (s) s^{-1-m/2} ds
	&=
	\frac{\cStar^{-m/2}}{1+m} \Bigl(\Lambda_m^* \fOneStar(\cStar) + \fOneStar(\cStar) \Bigr)\\
	&=\frac{\cStar^{-m/2}}{1+m} \Bigl(-\frac{2(1+m)}{2+m} + \dOneStar \cStar^{-1-m/2} \Bigr).
	\end{align*}
	Thus, the claim follows by our choice of $\dOneStar$.
	Consequently, we have $ \Lambda_m^* \fStar(t) = \Lambda_m^* \fOneStar(t) = - 1$ for $t\in(\dStar,\cStar)$
	(since $\fStar=\fOneStar$ on $(\dStar,\cStar)$). 
	This ends the proof of the lemma. 
\end{proof}

\subsection{Adjoint: guesstimates}

\label{sec:guesstimates-star}

Similarly as in Section~\ref{sec:guesstimates} we consider functions of the form \eqref{eq:weak-type-extremals-2-star} with $\aStar = 1$, $\cStar = \bStar$, and some additional constraints.

\begin{definition}
	\label{def:f-class-spec-star}
	Let $\fclassspecStar{m}\subset \fclassStar{m}$ be the class of functions
	which are of the form
	\[
	\fStar(t) = \bigl(-\frac{m}{2+m} + \dOneStar t^{-1 -m/2}\bigr) \indbr{t\in (\dStar,\bStar]}
	+  \bigl( \frac{m}{2+m} -\frac{2(1+m)}{2+m} t^{-1-m/2}\bigr)\indbr{t\in(\bStar,1]},
	\]
	for some $0< \dStar < \bStar < 1$ which satisfy
	\begin{equation}
	\label{eq:constraint-for-special-b-d-star} 
	2 > -\frac{m}{2+m} + \dOneStar (\bStar,m) \dStar^{-1-m/2}
	>  0 >  -\frac{m}{2+m} + \dOneStar(\bStar,m) \bStar ^{-1-m/2},
	\end{equation}
	where 
	\[
	\dOneStar = \dOneStar (\bStar,m)  \coloneqq \frac{2(1+m)}{2+m} \bigl(2 \bStar^{1+m/2} - 1 \bigr) >0.
	\]
\end{definition}

For $m\in\NN$ denote by $\specialbdStar{m}$
the set of $(\bStar,\dStar)\in\RR^2$ which satisfy all the constraints 
of Definition ~\ref{def:f-class-spec-star} (in particular, $0< \dStar < \bStar < 1$).
The constraints 
$-m/(2+m) + \dOneStar(\bStar,m) \bStar^{-1-m/2} <0$ and $\dOneStar(\bStar,m) >0$
imply that $\bStar\in (\bMinStar, \bMaxStar)$, where
\begin{equation}
\label{eq:def-bMin-bMax-star}
\bMinStar = \bMinStar(m)\coloneqq 2^{-2/(2+m)},
\quad \ 
\bMaxStar = \bMaxStar(m) \coloneqq \Bigl(\frac{2+2m}{4+3m}\Bigr)^{2/(2+m)}.
\end{equation}
Note that 
$-m/(2+m) + \dOneStar(\bMaxStar,m) \bMaxStar^{-1-m/2} = 0$.

Additionally, the constraint \eqref{eq:constraint-for-special-b-d-star} yields
that for fixed $\bStar\in (\bMinStar, \bMaxStar)$ we must have
$\dStar\in (\dMinStar, \dMaxStar)$,
where
\begin{align}
\label{eq:def-dMin-star} 
\dMinStar ^{-1-m/2} &= \dMinStar(\bStar,m)^{-1-m/2} 
\coloneqq \frac{4+3m}{ 2(1+m) (2\bStar^{1+m/2}  -1)},\\
\label{eq:def-dMax-star} 
\dMaxStar^{-1-m/2} &= \dMaxStar(\bStar,m)^{-1-m/2} \coloneqq \frac{m}{ 2(1+m) (2\bStar^{1+m/2}  -1)}.
\end{align}
Clearly, the functions
$\bStar\mapsto \dMinStar(\bStar,m)$,
$\bStar\mapsto \dMaxStar(\bStar,m)$
are increasing (in the above specified interval)
and the definitions also make sense for $\bStar = \bMaxStar$.

Finally, for $m\in\NN$ and $(\bStar,\dStar)\in\specialbdStar{m}$ 
we denote
\begin{multline}
\label{eq:a=1-b=c-would-like-to-optimize-technical-star}
\XXXspecialbdStar(\bStar,\dStar,m)\\
\coloneqq 
\frac{1-\dStar}{-\frac{2+m}{m}  + \frac{m}{2+m}\dStar + \frac{2(2+m)}{m}\bStar +
	\frac{4(1+m)}{m(2+m)}(2\bStar^{1+m/2}  -1)\dStar^{-m/2} 
	- 2\tZeroStar(\bStar,m)},
\end{multline}
where  $\tZeroStar  = \tZeroStar (\bStar,m)$ is  
the unique number in $[\dStar,\bStar]$  which satisfies
$-\frac{m}{2+m} + \dOneStar(\bStar,m) \tZeroStar ^{-1-m/2} = 0$,
i.e.,
\begin{align}
\label{eq:def-t_0-explicit-formula-star}
\tZeroStar =  \tZeroStar(\bStar,m) 
&= \Bigl(\frac{2(1+m)}{m}\Bigr)^{2/(2+m)} \bigl(2 \bStar^{1+m/2} -1\bigr)^{2/(2+m)}\\
& = \dMaxStar(\bStar,m). \nonumber
\end{align}
The definition of $\tZeroStar$ makes sense also for $\bStar=\bMaxStar$.
Note also that it follows from the proof of Lemma~\ref{lem:lower-bound-fixed-m-special-b-c-d-alternate-expression-with-inequality-star}
below that the denominator on the right-hand side of~\eqref{eq:a=1-b=c-would-like-to-optimize-technical-star} is positive and that we can consider $(\bStar,\dStar)$ on the boundary of $\specialbdStar{m}$ as well.

We have the following lower bound.

\begin{lemma}
    \label{lem:lower-bound-fixed-m-special-b-c-d-alternate-expression-with-inequality-star}
	For $m\in\NNone$,
	\begin{equation*}
	\weakconstrestricted{\Lambda_m^*}{\fclassspecStar{m}} 
	\geq 
	\sup\bigl\{
	\XXXspecialbdStar(\bStar,\dStar,m) : (\bStar,\dStar)\in\specialbdStar{m} \bigr\}.
	\end{equation*}
\end{lemma}

\begin{proof}
	Consider $0< \dStar < \bStar < 1$ satisfying the constraints of  Definition~\ref{def:f-class-spec-star}
	and the corresponding function  $\fStar\in\fclassspecStar{m}$. 
	By Lemma~\ref{lem:weak-type-extremals-2-star},
	\begin{equation*}
	\bigl|\{t\in[0,\infty) : |\Lambda_m^* \fStar(t)|\geq 1\} \bigr| 
	\geq  1 - \dStar.
	\end{equation*}
	Let us calculate 
		$\|\fStar\|_{L^1([0, \infty))} = \int_{\dStar}^{\cStar} |\fStar(t)| dt + \int_{\bStar}^1 |\fStar(t)| dt$.
		Observe first that $\lim_{t\to 1^-} \fStar(t) = -1 $,
		and consequently the function $\fStar$ is negative and increases on $(\bStar,1]$.
		Hence,
		\begin{align*}
		\int_{\bStar}^1 |\fStar(t)| dt &= \int_{\bStar}^1  - \frac{m}{2+m} +\frac{2(1+m)}{2+m} t^{-1-m/2} dt  \\
		&= - \frac{m}{2+m}(1-\bStar) - \frac{4(1+m)}{m(2+m)} (1-\bStar^{-m/2})\\
		&= -\frac{2+m}{m} + \frac{m}{2+m}\bStar + \frac{4(1+m)}{m(2+m)} \bStar^{-m/2}.
		\end{align*}
	Using the definition of $\tZeroStar  = \tZeroStar (\bStar,m)$ (twice, in the first and last equality) we can write 	
	\begin{align*}
	\MoveEqLeft[2]
	\int^{\bStar}_{\dStar} |\fStar(t)| dt
	= \int_{\dStar}^{\tZeroStar} -\frac{m}{2+m} + \dOneStar t^{-1-m/2}  dt 
	+  \int_{\tZeroStar}^{\bStar} \frac{m}{2+m} - \dOneStar t^{-1-m/2}  dt \\
	&= \frac{m}{2+m}(\dStar-2\tZeroStar+ \bStar) 
	+ \frac{2\dOneStar}{m} (\dStar^{-m/2} - 2 \tZeroStar^{-m/2} + \bStar^{-m/2}) \\
	&=  \frac{m}{2+m}(\dStar+ \bStar) + \frac{2\dOneStar}{m} (\dStar^{-m/2} + \bStar^{-m/2})
	+2\tZeroStar \cdot \bigl( -\frac{m}{2+m} -  \frac{2\dOneStar}{m} \tZeroStar^{-1-m/2} \bigr)  \\
	&=  \frac{m}{2+m}(\dStar+ \bStar) + \frac{2\dOneStar}{m} (\dStar^{-m/2} + \bStar^{-m/2})
	- 2\tZeroStar.
	\end{align*}
	Recalling the definition of $\dOneStar = \dOneStar(\bStar, m)$ yields
	\begin{multline*}
	\frac{2\dOneStar}{m} (\dStar^{-m/2} + \bStar^{-m/2}) 
	= \frac{4(1+m)}{m(2+m)}(2\bStar^{1+m/2}  -1)(\dStar^{-m/2} + \bStar^{-m/2})\\
	=\frac{4(1+m)}{m(2+m)}(2\bStar^{1+m/2}  -1)\dStar^{-m/2} 
	+\frac{8(1+m)}{m(2+m)} \bStar  - \frac{4(1+m)}{m(2+m)} \bStar^{-m/2}.
	\end{multline*}
	Gathering everything and simplifying the expression for $\int_{\dStar}^{\bStar} |\fStar(t)| dt+ \int_{\bStar}^1 |\fStar(t)| dt$ yields
	\[
		\weakconstrestricted{\Lambda_m^*}{\fclassspecStar{m}} 
		\geq
		\frac{	\bigl|\{t\in[0,\infty) : |\Lambda_m^* \fStar(t)|\geq 1\} \bigr|}{\|\fStar\|_{L^1([0, \infty))} } 
		\geq		\XXXspecialbdStar(\bStar,\dStar,m).
	\]
	The assertion follows by taking the supremum over $(\bStar,\dStar)\in \specialbdStar{m} $.
\end{proof}

\begin{corollary}
	\label{cor:lower-bound-asymptotic-special-b-c-d-alternate-expression-star}
	We have 
	\begin{equation*}
	\liminf_{m\to \infty} \weakconstrestricted{\Lambda_m}{\fclassspecStar{m}}
	\geq
	\sup\Bigl\{ 
	\frac{x + y}{2e^x - x -4  - y + 2(2-e^x) (e^y + 1) -2\ln(2(2-e^x))}
	\Bigr\},
	\end{equation*}
	where the supremum is taken over
	$x\in [\ln(3/2), \ln(2))$
	and $e^{-y} \leq 2(2-e^x) \leq  3e^{-y}$.
\end{corollary}

\begin{proof}
	Fix $x\in [\ln(3/2), \ln(2))$, $y>0$ such that $e^{-y} < 2(2-e^x) <  3e^{-y}$.
	It suffices to show that
	\begin{equation}
	\label{eq:ladida-star}
	\liminf_{m\to \infty} \weakconstrestricted{\Lambda_m}{\fclassspecStar{m}}
	\geq 	\frac{x + y }{2e^x - x -2+\int_0^y |-1 + 2(2-e^x) e^s | ds}
	\end{equation}
    (the assertion is then obtained as in the proof of Corollary~\ref{cor:lower-bound-asymptotic-special-b-c-d-alternate-expression}).
	To this end set
	\begin{equation*}
	\dStar = \dStar(m) = e^{-2(\xStar+\yStar)/(2+m)},\quad
	\bStar = \bStar(m) = e^{-2\xStar/(2+m)}.
	\end{equation*}
	and let $\dOneStar = \dOneStar(\bStar,m)$ be as in Definition~\ref{def:f-class-spec-star}.
	We have $\dOneStar(\bStar,m) > 0$ (since $x<\ln(2)$)
	and
	$\dOneStar(\bStar,m)\to 2(2e^{-x} -1)$ for $m\to\infty$.
	Our constraints for $x$, $y$ imply that
	\begin{equation*}
	-1 + 2(2e^{-x} - 1) e^x < 0 <  -1 + 2(2e^{-x} - 1) e^{x+y} < 2,
	\end{equation*}	
	so the constraint~\eqref{eq:constraint-for-special-b-d-star} is satisfied for $m$ large enough.
	Thus we can consider a~function $\fStar\in\fclassspecStar{m}$ of the form given in Definition~\ref{def:f-class-spec-star}
	for $0 < \dStar < \bStar  <1$ as above.

	From Lemma~\ref{lem:weak-type-extremals-2-star} we know that 
	\begin{equation*}
	m\bigl|\{ t\geq 0 : |\Lambda_m^* \fStar (t)| \geq 1 \} \bigr| 
	\geq m(1-\dStar(m)) \xrightarrow[m\to\infty]{} 2x + 2y. 
	\end{equation*}
	Moreover,
	\begin{align*}
	m \int_0^\infty |\fStar(t)| dt 
	&=m\int^1_{\exp(-2x/(2+m))} \Bigl| \frac{m}{2+m} -  \frac{2(1+m)}{2+m}  t^{-1-m/2}\Bigr| dt\\
	&\qquad  
	+ m\int^{\exp(-2x/(2+m))}_{\exp(-2(x+y)/(2+m)} \Bigl| -\frac{m}{2+m} + \dOneStar t^{-1-m/2}\Bigr| dt\nonumber\\
	&= 2 \int_{0}^{x} \Bigl| \frac{m}{2+m} - \frac{2(1+m)}{2+m} e^{s} \Bigr| e^{-2s/(2+m)} ds \\
	&\qquad	+  2\int_{x}^{x+y} \Bigl| -\frac{m}{2+m} + \dOneStar e^{s} \Bigr| e^{-2s/(2+m)} ds,
	\end{align*}
	where in the last step we substituted $t = e^{-2s/(2+m)}$.
	For fixed $x, y$ we pass to the limit with $m\to\infty$.
	The integrals tend to
	\begin{align*}
	&2\int_{0}^{x} \Bigl| 1 - 2 e^{s}\Bigr| ds    
	=  2\int_{0}^{x} 2 e^{s} -1 ds  
	= 2(2e^{x} -x - 2),\\
	&2 \int_{x }^{x+y} \bigl| -1 +  2(2 e^{-x } -1)  e^{s}\bigr| ds
	= 2\int_{0 }^{y} \bigl| -1 +  2(2-e^x) e^{s}\bigr| ds,
	\end{align*}
	respectively.
	This finishes the proof of inequality~\eqref{eq:ladida-star} and the corollary.
	\end{proof}

\begin{proof}[Proof of Theorem~\ref{thm:main-lower-bounds-star}]
	For, e.g., $m=1$ we can numerically maximize the function $\XXXspecialbdStar(\bStar,\dStar,m)$
	which appears in Lemma~\ref{lem:lower-bound-fixed-m-special-b-c-d-alternate-expression-with-inequality-star}.
	The resulting near optimal values $\bStar = 0.649$, $\dStar = 0.150$ imply that 
	\[
	\weakconstrestricted{\Lambda_1^*}{\fclassspecStar{1}} \geq 1.383.
	\]
	In order to prove the asymptotic lower bound we invoke Corollary~\ref{cor:lower-bound-asymptotic-special-b-c-d-alternate-expression-star}
	and take the same values of $x$ and $y$ as in the proof of Theorem~\ref{thm:main-lower-bounds} (note that the lower bounds in Corollaries~\ref{cor:lower-bound-asymptotic-special-b-c-d-alternate-expression} and~\ref{cor:lower-bound-asymptotic-special-b-c-d-alternate-expression-star} coincide exactly).
\end{proof}

\subsection{Non-technical sketch of the of the proof of Remark~\ref{rem:W-vs-W-star}}

While the reasoning presented above in the case of the operator $\Lambda_m^*$ definitely is parallel to the one from Section~\ref{SEC:Lambda},
it does not seem that one can deduce Theorem~\ref{thm:main-lower-bounds-star} immediately from Theorem~\ref{thm:main-lower-bounds} by some easy substitution or change of variables (cf. the proofs of Corollaries~\ref{cor:lower-bound-asymptotic-special-b-c-d-alternate-expression} and~\ref{cor:lower-bound-asymptotic-special-b-c-d-alternate-expression-star},
where eventually the functions under the suprema on the right-hand sides coincide exactly,
but this happens only after the limit passage $m\to\infty$).

\begin{figure}[t]
	\begin{tikzpicture}[scale=1, xscale=4,yscale=1.2]
	\draw[->] (0, 0) -- (2.3, 0) node[below] {$b$};
	\draw[->] (0, 0) -- (0, 3.9) node[left] {$d$};
	\node[below ] at ({8/6},0) {$\bMin\coloneqq \bigl(\frac{2+3m}{2+2m}\bigr)^{2/m}\qquad $};
	\node[below] at ({2},0) {$\bMax\coloneqq 2^{2/m}\vphantom{\tfrac{3}{2}}\qquad\vphantom{\bigl(\frac{2+3m}{2+2m}\bigr)^{2/m}}$};
	\node[left] at (0,8/6) {$\bMin = \dMin(\bMin)  $};
	\node[left] at (0,{(4*sqrt(5))/3^(3/2)}) {$\dOpt(\bMin)$}; 
	\node[left] at (0,8/3) {$\dMax(\bMin)$};
	\fill [black!5!white, domain=8/6:1.7, variable=\x]
	(8/6,8/6)
	-- plot ({\x}, {(2+2)/(2*(1+2)*(2*\x^(-2/2) -1))})
	-- ({34/23},{(2*2 + 2+2)/(2*(1+2)*(2*(34/23)^(-2/2)-1))})
	-- (34/23, 8/3)
	-- (8/6,8/3)
	-- cycle;
	\fill [black!5!white, domain=8/6:34/23, variable=\x]
	(8/6,8/3)
	-- plot ({\x}, {(2*2+2+2)/(2*(1+2)*(2*\x^(-2/2) -1))})
	-- (34/23, 8/3)
	-- cycle;
	\draw[domain=8/6:1.7, smooth, variable=\x, MyLineCases] plot ({\x}, {(2+2)/(2*(1+2)*(2*\x^(-2/2) -1))});
	\draw[domain=8/6:8/3, smooth, variable=\y, MyLineCases]  plot ({8/6}, {\y});
	\draw[domain=8/6:34/23, smooth, variable=\x, MyLineCases] plot ({\x}, {(2*2+2+2)/(2*(1+2)*(2*\x^(-2/2) -1))});
	\draw[ domain=8/6:1.609865, smooth, variable=\x, MyLineOpt]  plot ({\x}, {sqrt(16/(3*(2/\x-1)^2)-2*\x^2)/sqrt(6)});
	\draw[ domain=0:2, smooth, variable=\x, black, MyLineHelp]  plot ({\x}, {\x});
	\draw[domain=8/3:34/9, smooth, variable=\y, MyLineHelp]  plot ({8/6}, {\y});
	\draw[domain=0:8/6, smooth, variable=\y, MyLineHelp]  plot ({8/6}, {\y});
	\draw[domain=0:34/9, smooth, variable=\y, MyLineHelp]  plot ({2}, {\y});
	\draw[domain=0:8/6, smooth, variable=\x, MyLineHelp]  plot ({\x}, {8/6});
	\draw[domain=0:8/6, smooth, variable=\x, MyLineHelp]  plot ({\x}, {(4*sqrt(5))/3^(3/2)});
	\draw[domain=0:8/6, smooth, variable=\x, MyLineHelp]  plot ({\x}, {8/3});
	\draw[MyArrowsL={0.5}, domain=8/6:24/17, variable=\x, MyLinePush]  plot ({\x}, {1.6});
	\draw[MyArrowsRL={0.25}{0.8}, domain=8/6:66/43, variable=\x, MyLinePush]  plot ({\x}, {2.2});
	\draw[MyArrowsRL={0.35}{0.85}, domain=18/13:18/11, variable=\x, MyLinePush]  plot ({\x}, {3});
	\end{tikzpicture}
	\caption{The light gray region
		is the set~$\specialbd{m}$
		consisting of parameters $b$, $d$
		satisfying the constraints of the Definition~\ref{def:f-class-spec}.
		Its boundary consists of the dashed lines described by:
		$b = \bMin(m)$, 
		$d= \dMin(b,m)$,
		and $d=\dMax(b,m)$.
		The thick black line is the curve $d = \dOpt(b,m)$.
		Note that the axes have been scaled for readability.}
	\label{fig:m-fixed-optimal-curve}
\end{figure}
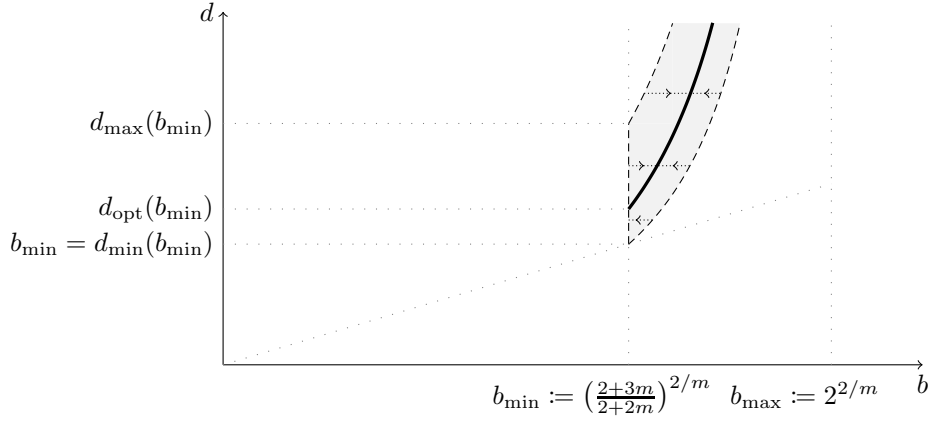

Nonetheless, Remark~\ref{rem:W-vs-W-star} asserts that there is
a relation
between the restricted weak-type constant of $\Lambda_m$ and its adjoint operator.
\LongArxivOrShortJournalVersion{%
We only provide a high-level sketch of the proof at this point;
the proof with all technical details and calculations 
is presented Appendix~\ref{APP:equality}.}{%
We only provide a high-level sketch of the proof;
the proof with all technical details and calculations
can be found in Appendix~A 
of the arXiv version of this paper.}

\begin{proof}[Sketch of the proof of Remark~\ref{rem:W-vs-W-star}]
	Let us assume that $m\geq 2$ (for $m=1$ the outline of the proof is the same, but some additional intricate technical details appear). First of all, one can prove that actually
	\begin{equation*}
	\weakconstrestricted{\Lambda_m}{\fclassspec{m}} 
	=  
	\sup\bigl\{
	\XXXspecialbd(b,d,m) : (b,d)\in\specialbd{m} \bigr\},
	\end{equation*}
	i.e., we have equality in the statement of Lemma~\ref{lem:lower-bound-fixed-m-special-b-c-d-alternate-expression-with-inequality}.
	This step requires checking carefully that for a function $f\in\fclassspec{m}$ we have $|\Lambda_m f(t)| = 1$ if \emph{and only if} $t\in (1,d)$.
	
	The parameter $b$ appears only in the denominator of the function $\XXXspecialbd$.
	Minimizing over it, we can find the optimal value of $b$ for fixed $d$; this optimal value is expressed in implicit form, i.e., one can prove that there exists a function $b\mapsto \dOpt(b,m)$ such that
	\begin{equation*}
		\weakconstrestricted{\Lambda_m}{\fclassspec{m}} 
		=  \sup\bigl\{
		\XXXspecialbd(b,\dOpt(b,m),m) : \bMin(m) \leq b < \bMax(m) \bigr\}
	\end{equation*}
	(cf. Figure~\ref{fig:m-fixed-optimal-curve}).
	Similarly, one can prove that
	\begin{equation*}
		\weakconstrestricted{\Lambda_m^*}{\fclassspecStar{m}} 	
	= 
		\sup\bigl\{
		\XXXspecialbdStar(\bStar,\dOptStar(\bStar,m),m) : \bMinStar(m) < \bStar \leq \bMaxStar(m) \bigr\}
	\end{equation*} 
	for some function $\bStar\mapsto \dOptStar(\bStar,m)$.
	
The main part of the proof is realizing that for an appropriate choice of $\bStar=\bStar(b,m)$
we have $\XXXspecialbd(b,\dOpt(b,m),m) = \XXXspecialbdStar(\bStar,\dOptStar(\bStar,m),m)$.
More precisely,  let us denote 
\[
\bStarWvsWStar \coloneqq \bStarWvsWStar (b,m) =  \frac{t_0(b,m)}{\dOpt(b,m)}.
\]
One can prove that this number satisfies the constraint $\bStarWvsWStar\in (\bMinStar, \bMaxStar)$ (and can take any value in this interval)
 and moreover the following  equalities hold (cf.\ Figure~\ref{fig:star-relations}):
\begin{align*}
\tZeroStar(\bStarWvsWStar,m) 
&= \frac{b}{\dOpt(b,m)},\\
\dOptStar(\bStarWvsWStar,m) 
&= \frac{1}{\dOpt(b,m)},\\
\XXXspecialbdStar(\bStarWvsWStar, \dOptStar(\bStarWvsWStar,m), m) 
&= 	\XXXspecialbd(b,\dOpt(b,m),m).
\end{align*}
Proving these facts requires some rather involved calucations with implicit functions.

Combining all these claims finishes the proof.
\end{proof}

\begin{figure}[t]
	\begin{tikzpicture}[scale=1, xscale=3.5, yscale=0.5]
	\draw[->] (0, 0) -- (3, 0); 
	\draw[->] (0, -2) -- (3, -2); 
	\draw (2.692582403567252,0) -- (1,-2);
	\draw (2,0) -- (2/2.692582403567252,-2);
	\draw [dashed] (1.5,0) -- (1.5/2.692582403567252,-2);
	\draw [dashed] (1,0) -- (1/2.692582403567252,-2);
	\node[circle,minimum size=3pt,inner sep=0,fill=black,label=above:{$1\vphantom{\dOpt(b,m)t_0(b,m)}$}] at (1,0) {};
	\node[circle,minimum size=3pt,inner sep=0,fill=black,label=above:{$b\vphantom{\dOpt(b,m)t_0(b,m)}$}] at (1.5,0) {};
	\node[circle,minimum size=3pt,inner sep=0,fill=black,label=above:{$t_0(b,m)\vphantom{\dOpt(b,m)}$}] at (2,0) {};		
	\node[circle,minimum size=3pt,inner sep=0,fill=black,label=above:{$\dOpt(b,m)\vphantom{\dOpt(b,m)t_0(b,m)}$}] at (2.692582403567252,0) {};
	\node[circle,minimum size=3pt,inner sep=0,draw=black,fill=white,label=below:{}] at (1/2.692582403567252,-2)  {}; 
	\node[circle,minimum size=3pt,inner sep=0,draw=black,fill=white,label=below:{}] at (1.5/2.692582403567252,-2)  {}; 
	\node[circle,minimum size=3pt,inner sep=0,fill=black,label=below:{$\frac{t_0(b,m)}{\dOpt(b,m)}$}] at (2/2.692582403567252,-2) {};		
	\node[circle,minimum size=3pt,inner sep=0,fill=black,label=below:{$1\vphantom{\frac{t_0(b,m)}{\dOpt(b,m)}}$}] at (1,-2) {};
	
	\end{tikzpicture}
	\caption{For a specific and appropriate choice of $d=d(b,m)$, $\bStar=\bStar(b,m)$, and $\dStar = \dStar(\bStar(b,m),m)$
		we have $\XXXspecialbd(b,d,m) = \XXXspecialbdStar(\bStar,\dStar,m)$.
	}
	\label{fig:star-relations}
\end{figure}

\section{Remarks and conjectures}

\label{sec:conclusions}

Theorem~\ref{thm:main-lower-bounds}
provides 
a counterexample to conjecture~\eqref{eq:conjecture-Gill} posed in~\cite{MR2677626}
and
shows that even on radial function subspaces
the behavior of the Beurling--Ahlfors transform~$\BA$ 
can be more complicated than one would initially expect.
Theorem~\ref{thm:main-lower-bounds-star} and Remark~\ref{rem:W-vs-W-star}
 suggest that there could be a deeper relation 
between the operators $\Lambda_m$ and $\Lambda_m^*$.
The author believes that the following conjecture holds.

\begin{conjecture}
	\label{conj:mine}
For $m\in\NNone$,
\[
\weakconst{\Lambda_m} 
\overset{??}{=} \weakconstrestricted{\Lambda_m}{\fclass{m}}
\overset{?}{=} \weakconstrestricted{\Lambda_m}{\fclassspec{m}}
= \weakconstrestricted{\Lambda_m^*}{\fclassspecStar{m}}	 
\overset{?}{=} \weakconstrestricted{\Lambda_m^*}{\fclassStar{m}} 
\overset{??}{=} \weakconst{\Lambda_m^*}.
\]
Moreover, this quantity is decreasing in $m$.
\end{conjecture}

\LongArxivOrShortJournalVersion{
The middle equality is true by Remark~\ref{rem:W-vs-W-star}.
Note that Proposition~\ref{prop:asymptotic-push} 
hints that the equality $\weakconstrestricted{\Lambda_m}{\fclass{m}}
= \weakconstrestricted{\Lambda_m}{\fclassspec{m}}$
holds asymptotically when $m\to\infty$,
but the form in which Proposition~\ref{prop:asymptotic-push}  (and related results, namely
Proposition~\ref{prop:lower-bound-asymptotic}, Corollary~\ref{cor:lower-bound-asymptotic-special-b-c-d-alternate-expression}, and  Proposition~\ref{prop:lower-bound-fixed-m-special-b-c-d})
 is stated does not yield a rigorous proof of this fact.
 Nonetheless, proving the equalities denoted with a single question mark
 should be a purely technical, if somewhat daunting, exercise in calculus;
 the author does not insist that to this end it is best to follow the approach presented in Appendix~\ref{APP:push}.}{
The middle equality is true by Remark~\ref{rem:W-vs-W-star}.
Proposition~C.3 in the arXiv version of this article suggests that
the equality $\weakconstrestricted{\Lambda_m}{\fclass{m}}
= \weakconstrestricted{\Lambda_m}{\fclassspec{m}}$ holds asymptotically when $m\to \infty$;
proving the equalities denoted with a single question mark
 should be a purely technical, if somewhat daunting, exercise in calculus (cf.\  Appendix~D in the arXiv version of this article).}

Proving the equalities denoted with double question marks
is a task of different kind, since it involves proving a weak-type inequality
for \emph{all} functions $f\in L^1([0,\infty))$
(instead of working with functions of a special form).
Nonetheless, the author believes that the classes $\fclass{m}$
could be the classes of extremal functions 
needed to give a Bellman function proof of the weak-type inequality for the operators $\Lambda_m$ 
(with optimal constant), cf.\ the method described in~\cite{MR3306045} 
and the way that extremal functions are connected with the Bellman function in \cite{MR3558516,MR4088501}.
Note that the operators $\Lambda_m$, $m\geq 1$,
can be viewed as weighted versions of $\Lambda_0$ and
it is possible that the problem of finding their weak-type constants 
will require more effort than in the unweighted case $m=0$
(especially since for $m=0$ nice probabilistic understanding of the operators $\Lambda_0$, $\Lambda_0^*$ is available, cf.\ \cite{MR3018958,MR3189475,MR3558516}).


Finally, let us mention
that an abstract and non-technical proof of either of the missing equalities, especially $\weakconst{\Lambda_m} 
=\weakconst{\Lambda_m^*},$ would be interesting.
Finding a less technical proof of the equality
$\weakconstrestricted{\Lambda_m}{\fclassspec{m}}
= \weakconstrestricted{\Lambda_m^*}{\fclassspecStar{m}}$
might be a good starting point.

%
%
%
%

\appendix

%
%
%
%
%
%
%
%

\section{Detailed proof of Remark~\ref{rem:W-vs-W-star}}

\label{APP:equality}

In this section we provide all the technical details needed for the proof that $\weakconstrestricted{\Lambda_m}{\fclassspec{m}}
= \weakconstrestricted{\Lambda_m^*}{\fclassspecStar{m}}$ for all $m\in\NN$.
Some of the results serve at the same time as additional motivation behined Conjecture~\ref{conj:mine} and Definitions~\ref{def:f-class-spec} and~\ref{def:f-class-spec-star}.

Throughout this section $m$ is fixed 
and we supress the dependence on it in many places in the notation:
we shall often write, e.g.,
$\bMin$ and $t_0(\bMin)$ instead of
$\bMin(m)$ and $t_0(\bMin(m), m)$, respectively.

\subsection{Additional technical calculations}

This subsection is devoted to proving an enhanced version of Lemma~\ref{lem:lower-bound-fixed-m-special-b-c-d-alternate-expression-with-inequality} --
it is 
stated in Proposition~\ref{prop:lower-bound-fixed-m-special-b-c-d} below.
The proof is split into several lemmas (and uses the notation introduced in Subsection~\ref{sec:guesstimates}).
Some of the ideas and steps are illustrated in Figure~\ref{fig:m-fixed-optimal-curve} on page~\pageref{fig:m-fixed-optimal-curve}.

We start with the following more general and more exact version of Lemma~\ref{lem:lower-bound-fixed-m-special-b-c-d-alternate-expression-with-inequality}.
Recall the notation $s\lor t = \max\{s, t\}$ and $s \land t = \min\{s, t\}$ for $s,t\in\RR$.

\begin{proposition}
	\label{prop:lower-bound-fixed-m}
	For $m\in\NNone$,
	\begin{equation*}
	\weakconstrestricted{\Lambda_m}{\fclass{m}} 
	=
	\sup\Bigl\{
	\frac{\bGain-1 + \dGain-c}{ \frac{m}{2+m} - \frac{2+m}{m}b + \frac{4(1+m)}{m(2+m)} b^{1+m/2} + \int_c^d \bigl|-\frac{2+m}{m} + \dOne t^{m/2} \bigr| dt}
	\Bigr\},
	\end{equation*}
	where the supremum is taken over  $1<b\leq c <d$ and  
	\begin{align*}
	\dOne &= \dOne(b,c,m) \coloneqq \frac{2(1+m)}{m c^{m/2}} \Bigl(1 + \Bigl(\frac{b}{c}\Bigr)^{1+m/2}\bigl(1 - b^{m/2}\bigr)\Bigr),\\
	\bGain &=\bGain(b,c,m)
	\coloneqq \Bigl[ b\lor \Bigl[ b\Bigl(-1 -\frac{2+m}{m} + \frac{2(1+m)}{m} b ^{m/2} \Bigr)^{2/(2+m)}\Bigr]\Bigr] \land c,\\
	\dGain &=\dGain(b,c,d,m)
	\coloneqq d\lor \Bigl[ d \bigl| - 1 - \frac{2+m}{m} + D d ^{m/2} \bigr|^{2/(2+m)}\Bigr].
	\end{align*}
\end{proposition}

\begin{proof}
	Consider a function  $f$ of the form \eqref{eq:weak-type-extremals-2}. 
	Due to scaling invariance,
	we may and do assume that $a=1$.
	Thus, for  any $1 < b \leq c <d$ we consider the function
	\[
	f(t) = \bigl(\frac{2+m}{m} -\frac{2(1+m)}{m} t^{m/2}\bigr) \indbr{t\in (1,b]}
	+  \bigl( -\frac{2+m}{m} + \dOne t^{m/2}\bigr)\indbr{t\in(c,d]},
	\]
	where $\dOne = \dOne(b,c,m) $ is as in the statement of the proposition.
	Let $\fZero(t)=  \frac{2+m}{m} - \frac{2(1+m)}{m} t^{m/2}$, $\fOne(t)=   -\frac{2+m}{m} + \dOne t^{m/2}$,
	be the functions from the proof of Lemma~\ref{lem:weak-type-extremals-2}.

	We start with identifying the set where $|\Lambda_m f (t)| \geq 1$.
	Clearly
	$ \Lambda_m f(t) =0$ for $t\in(0,1)$ 
	and $|\Lambda_m f(t)| = 1$ for $t\in (1,b) \cup (c,d)$
	by Lemma~\ref{lem:weak-type-extremals-2}.
	However, since $\lim_{t\to 1^+} f(t) = -1 $,
	the function $f$ is negative and decreases on $(1,b]$.
	Thus,
	although $f(t) = 0$ for $t\in (b,c)$,
	it may happen that $\Lambda_m f(t) \leq -1$ for some $t\in(b,c)$
	due to the (negative) mass acumulated in $\int_0^b f(s) s^{m/2} ds$.
	Indeed, for $t\in (b,c)$ we have
	\begin{align*}
	\Lambda_m f(t)
	&= \frac{1+m}{t^{1+m/2}} \int_0^b f(s) s^{m/2}ds
	= \frac{1+m}{t^{1+m/2}} \int_0^b \fZero(s) s^{m/2}ds\\
	&= \frac{b^{1+m/2}}{t^{1+m/2}} \bigl( \Lambda_m \fZero(b) + \fZero(b)\bigr)
	= \frac{b^{1+m/2}}{t^{1+m/2}} \bigl( 1 + \frac{2+m}{m}  - \frac{2(1+m)}{m} b^{m/2}\bigr), 
	\end{align*}
	which is negative and increases in $t$.
	Hence, the condition $|\Lambda_m f(t)| \geq 1$ and $t\in(b,c)$ is equivalent to $t\in(b, \bGain)$, where $\bGain$ is as in the statement of the proposition.
	
	Similarly, by the very definition of the coefficient $\dOne$ and the calculations from the proof of Lemma~\ref{lem:weak-type-extremals-2}, for $t\in(d,\infty)$,
	\begin{align*}
	\Lambda_m f(t)
	&= \frac{1+m}{t^{1+m/2}} \int_0^d f(s) s^{m/2}ds
	= \frac{1+m}{t^{1+m/2}} \int_0^d \fOne(s) s^{m/2}ds\\
	&= \frac{d^{1+m/2}}{t^{1+m/2}} \bigl( \Lambda_m \fOne(d) + \fOne(d)\bigr)
	= \frac{d^{1+m/2}}{t^{1+m/2}} \bigl( -1 - \frac{2+m}{m}  + \dOne d^{m/2}\bigr).
	\end{align*}
	The absolute value of this expression decreases in $t$,
	so the condition $|\Lambda_m f(t)| \geq 1$ and $t\in(d,\infty)$
	is equivalent to $t\in(d, \dGain)$.

	We calculate $\|f\|_{L^1([0, \infty))}$ as in the proof of Lemma~\ref{lem:lower-bound-fixed-m-special-b-c-d-alternate-expression-with-inequality} above.
	This ends the proof.
\end{proof}

With this result in hand we can upgrade the statement of Lemma~\ref{lem:lower-bound-fixed-m-special-b-c-d-alternate-expression-with-inequality}
from an estimate to an equality.

\begin{lemma}
	\label{lem:lower-bound-fixed-m-special-b-c-d-alternate-expression}
	For $m\in\NNone$,
	\begin{equation*}
	\weakconstrestricted{\Lambda_m}{\fclassspec{m}} 
	=  
	\sup\bigl\{
	\XXXspecialbd(b,d,m) : (b,d)\in\specialbd{m} \bigr\}.
	\end{equation*}
\end{lemma}

\begin{proof}
	We proceed as in the proof of Lemma~\ref{lem:lower-bound-fixed-m-special-b-c-d-alternate-expression-with-inequality},
	with the only difference being that Proposition~\ref{prop:lower-bound-fixed-m} (and its proof)
	and Definition~\ref{def:f-class-spec} actually  that for $f\in\fclassspec{m}$,
	\begin{equation*}
	\bigl|\{t\in[0,\infty) : |\Lambda_m f(t)|\geq 1\} \bigr| 
	= d-1.
	\qedhere
	\end{equation*}
\end{proof}
\begin{lemma}
	\label{lem:dOpt-definition}
	For $m\in\NN$ and $b\in [\bMin, \bMax)$ we have
	\begin{align*}
	\MoveEqLeft[2]
	2(1+m)(2b^{-m/2} -1) \dMin(b,m)^{1+m/2} \\
	&\leq -m(2b^{-m/2} - 1) b^{1+m/2}  
	+ 2(2+m) t_0(b,m),\\
	&\leq 2(1+m)(2b^{-m/2} -1) \dMax(b,m)^{1+m/2}.
	\end{align*}	
	In particular, for every $b\in [\bMin, \bMax)$
	there exists  $\dOpt= \dOpt(b,m)\in [\dMin, \dMax]$ 
	such that
	\begin{align}
	\label{eq:fixed-m-supremum-W-optimal-curve}
	2(1+m)(2b^{-m/2} -1) & \dOpt(b,m)^{1+m/2} \\
	&= -m(2b^{-m/2} - 1) b^{1+m/2}  
	+ 2(2+m) t_0(b,m). \nonumber
	\end{align}	
\end{lemma}

\begin{proof}
	Consider first $\dMin=\dMin(b,m)$.
	Recall from~\eqref{eq:def-t_0-explicit-formula} that $t_0 = t_0(b,m) = \dMin$ 
	and
	\begin{equation*}
	2(1+m)(2 b^{-m/2} -1)\dMin^{1+m/2} 
	=  (2+m) \dMin =  (2+m) t_0.
	\end{equation*} 
	Consequently,
	\begin{equation*}
	-m(2 b^{-m/2} -1)b^{1+m/2} \geq -m(2 b^{-m/2} -1) t_0^{1+m/2} = -\frac{m(2+m)}{2(1+m)} t_0.
	\end{equation*} 
	Thus
	\begin{multline*}
	-m(2 b^{-m/2} -1)b^{1+m/2} 
	+2(2+m) t_0(b,m) - 2(1+m)(2 b^{-m/2} -1)\dMin^{1+m/2}\\
	\geq \Bigl(-\frac{m(2+m)}{2(1+m)} + 2(2+m) - (2+m)\Bigr)\cdot t_0 = \frac{(2+m)^2}{2(1+m)} \cdot t_0 \geq 0,
	\end{multline*}
	which proves the first part of the assertion.
	
	It follows from the definition~\eqref{eq:def-dMax} that $\dMax = \dMax(b,m)$ satisfies
	\begin{align*}
	2(1+m)(2 b^{-m/2} -1)\dMax^{1+m/2}
	&= (2+3m) \dMax
	\end{align*} 
	and in order to prove the second part of the assertion  
	it suffices to check that
	\[
	2(2+m) t_0 \leq (2+3m) \dMax.
	\]
	To this end, observe that
	\begin{align*}
	2(1+m)(2 b^{-m/2} -1)\Bigl( \frac{2+3m}{4+2m} \dMax\Bigr)^{m/2}  
	&= (2+3m)\Bigl( \frac{2+3m}{4+2m} \Bigr)^{m/2}\\
	&\geq
	\begin{cases}
	2+3m & \text{if } m\geq 2,\\
	5 (5/6)^{1/2} & \text{if } m=1,
	\end{cases}
	\\
	&\geq 
	2+m \quad \text{(for all $m\geq 1$)}.
	\end{align*}
	This inequality means exactly that $t_0 \leq \frac{2+3m}{4+2m} \dMax$,
	which finishes the proof.
\end{proof}

We now consider the behavior of the function $\XXXspecialbd$
on the part of the boundary of the  set $\specialbd{m}$ 
which is obtained by considering $b \to \bMin(m) ^+$.

\begin{lemma}
	\label{lem:lower-bound-fixed-m-special-b-c-d-boundary-bMin}
	For $m\in\NNone$,
	\begin{align}
	\nonumber
	\dMin(\bMin(m), m) &= \Bigl( \frac{2+3m}{2+2m}\Bigr)^{2/m} = \bMin(m),\\	
	\nonumber
	\dMax(\bMin(m),m) &= \Bigl(\frac{(2+3m)^2}{(2+m)(2+2m)}\Bigr)^{2/m}
	= \bMin(m) \Bigl(\frac{2+3m}{2+m}\Bigr)^{2/m},\\
	\label{eq:dOpt-for-bMin}
	\dOpt(\bMin(m),m) &= \bMin(m) \cdot \Bigl( \frac{4+3m}{2+2m}\Bigr)^{2/(2+m)}.
	\end{align}
\end{lemma}

\begin{proof}	
	We have $\bMin^{m/2} = (2+3m)/(2+2m)$, so
	\[
	2\bMin^{-m/2}  - 1 = 2 \cdot \frac{2+2m}{2+3m} - 1 = 
	\frac{2+m}{2+3m}.
	\]
	The formulas for $\dMin(\bMin)$ and $\dMax(\bMin)$ follow immediately (cf.\ \eqref{eq:def-dMin}, \eqref{eq:def-dMax}); note also that
	$t_0(\bMin) \dMin(\bMin)= \bMin$.

	Using the preceding observations 
	we can rewrite~\eqref{eq:fixed-m-supremum-W-optimal-curve}
	as
	\begin{align*}  
	2(1+m)\frac{2+m}{2+3m} \dOpt(\bMin)^{1+m/2} 
	&= -m \frac{2+m}{2+3m} \bMin^{1+m/2}   
	+ 2(2+m) \bMin\\
	&= -m \frac{2+m}{2+3m} \bMin^{1+m/2}   
	+ 2(2+m) \frac{2+2m}{2+3m} \bMin^{1+m/2}, 
	\end{align*}
	which yields the claim~\eqref{eq:dOpt-for-bMin} and finishes the proof of whole lemma.
\end{proof}

\begin{remark}
	In particular,
	\begin{equation}
	\label{eq:claim-dOpt-for-bMin}
	\dMin(\bMin(m),m)  < \dOpt(\bMin(m),m)  < \dMax(\bMin(m),m).
	\end{equation}
	Indeed, this is equivalent to
	\[
	1 < \Bigl( \frac{4+3m}{2+2m}\Bigr)^{2/(2+m)} < \Bigl( \frac{2+3m}{2+m}\Bigr)^{2/m}.
	\]
	The left-hand side inequality holds true for all $m\in \NN$.
	The right-hand side inequality is satisfied for $m=1$ and for $m\geq 2$ we can write
	\[
	\Bigl( \frac{4+3m}{2+2m}\Bigr)^{2/(2+m)}
	\leq 2^{2/(2+m)} < 2^{2/m}  \leq  \Bigl( \frac{2+3m}{2+m}\Bigr)^{2/m}.
	\]
\end{remark}

\begin{lemma}
	\label{lem:tZero-divied-by-dOpt-monotone}
	For $m\in\NN$ and 
	$b\in [\bMin, \bMax)$ we have
	\[
	2^{-2/(2+m)}  < \frac{t_0(b)}{\dOpt(b)} \leq \frac{t_0(\bMin)}{\dOpt(\bMin)} 
	=
	\Bigl(\frac{2+2m}{4+3m}\Bigr)^{2/(2+m)}.
	\]
	Moreover, the function $b\mapsto \dOpt(b)$ is increasing, the function $b\mapsto t_0(b)/\dOpt(b)$ is decreasing and  
	\[
	\lim_{b\to\bMax}\frac{t_0(b)}{\dOpt(b)}
	= 2^{-2/(2+m)}.
	\]
\end{lemma}

\begin{proof}
	The expression for  $t_0(\bMin)/\dOpt(\bMin)$ 
	follows from Lemma~\ref{lem:lower-bound-fixed-m-special-b-c-d-boundary-bMin}.
	The definition of $t_0$ (see~\eqref{eq:def-t_0-explicit-formula}) 
	can be rewritten as
	\[
	2(1+m) (2b^{-m/2}  -1 ) = (2+m)  t_0(b)^{-m/2}.
	\]
	Hence, if $b\to \bMax = 2^{2/m}$, then $t_0(b) \to \infty$ and $b/t_0(b)\to 0$.
	Moreover, by multiplying both sides by $b^{m/2}$ we see that the function $b\mapsto b/t_0(b)$ is decreasing.
	Substituting the above expression for $2(1+m)(2b^{-m/2} - 1)$
	into~\eqref{eq:fixed-m-supremum-W-optimal-curve} we arrive at:
	\begin{equation*}
	\frac{ \dOpt(b)^{1+m/2}}{ t_0(b)^{m/2}}
	= -\frac{mb^{1+m/2}}{2(1+m) t_0(b)^{m/2}}   
	+ 2 t_0(b,m). 
	\end{equation*}
	If we divide both sides by $t_0(b)$ and let $b\to \bMax$, we can conclude that 
	\[
	\lim_{b\to\bMax}\frac{t_0(b)}{\dOpt(b)}
	= 2^{-2/(2+m)}.
	\]
	and moreover the convergence is monotone (decreasing).
	In particular, since $t_0(b)$ is increasing, $\dOpt(b)$ also has to be increasing.
\end{proof}

\begin{lemma}
	\label{lem:aux-dOpt-for-bMin}
	For $m\in\NN$,
	\[
	\frac{2+m}{m}\Bigl(\frac{4+3m}{2+2m}\Bigr)^{m/(2+m)} - \frac{2+m}{m} \geq \frac{1}{2} \geq
	\frac{m}{2+m}\Bigl(\frac{2+3m}{2+2m}\Bigr)^{(2+m)/m} - \frac{m}{2+m}.
	\]
\end{lemma}

\begin{proof}
	The left-hand side inequality can be rewritten as 
	\[
	\Bigl(\frac{4+3m}{2+2m}\Bigr)^{m/(2+m)}
	\geq \frac{1}{2}\cdot \frac{m}{2+m} + 1 = \frac{4+3m}{4+2m},
	\]
	which is equivalent to
	\[
	4+2m \geq (4+3m)^{2/(2+m)} (2+2m)^{m/(2+m)},
	\]
	which is true by the concavity of the logarithmic function since
	\[
	4+2m = \frac{2}{2+m}\cdot(4+3m) + \frac{m}{2+m} \cdot  (2+2m).
	\]
	
	Similarly,
	the right-hand side inequality can be rewritten as 
	\[
	\frac{2+3m}{2m} = \frac{1}{2}\cdot \frac{2+m}{m} + 1
	\geq 	\Bigl(\frac{2+3m}{2+2m}\Bigr)^{(2+m)/m}
	\]
	or equivalently as
	\[
	2+2m \geq (2m)^{m/(2+m)} (2+3m)^{2/(2+m)},
	\]
	which again follows by  concavity since
	\[
	2+2m = \frac{m}{2+m}\cdot(2m) + \frac{2}{2+m} \cdot  (2+3m).
	\]
	This finishes the proof of the lemma.
\end{proof}

\begin{lemma}
	\label{lem:lower-bound-fixed-m-special-b-c-d-boundary-bMin-part-2}
	For $m\in\NNone$
	and
	$d \in \bigl[\bMin(m), \dOpt(\bMin(m),m)\bigr]$,
	\begin{equation}
	\label{eq:claim-W-boundary}
	\XXXspecialbd(\bMin(m), d, m) \leq \XXXspecialbd\bigl(\bMin(m), \dOpt(\bMin(m),m), m\bigr).
	\end{equation}
\end{lemma}

\begin{proof} 
	From Lemma~\ref{lem:lower-bound-fixed-m-special-b-c-d-boundary-bMin} and \eqref{eq:def-t_0-explicit-formula} we 
	we know that
	$t_0(\bMin) = \dMin(\bMin) = \bMin$. 
	Recall also that
	$\bMin^{-m/2} = \frac{2+2m}{2+3m}$ 
	and
	\[
	\frac{4(1+m)}{m(2+m)}(2\bMin^{-m/2}  -1) = 	\frac{4(1+m)}{m(2+3m)}. 
	\]
	It follows that
	\[
	\XXXspecialbd(\bMin,d,m)= \frac{d-1}{
		\frac{4(1+m)}{m(2+3m)}d^{1+m/2} -\frac{2+m}{m}d + \frac{m}{2+m} + \frac{4}{2+m} \bMin(m)}.
	\]
	In order to prove the inequality~\eqref{eq:claim-W-boundary} 
	we claim that
	\begin{equation}
	\label{eq:claim-dW-dd-boundary}
	\frac{\partial }{\partial d} \XXXspecialbd(\bMin, d , m) \geq 0  \qquad\text{ for } \quad d \in [\bMin, \dOpt(\bMin)].
	\end{equation} 	 
	Indeed, $\frac{\partial }{\partial d} \XXXspecialbd(\bMin, d , m)$
	has the same sign as
	\begin{align*}
	&\frac{4(1+m)}{m(2+3m)}d^{1+m/2} -\frac{2+m}{m}d + \frac{m}{2+m} + \frac{4}{2+m} \bMin\\
	&\qquad \qquad \quad -(d-1)\cdot\Bigl(\frac{4(1+m)}{m(2+3m)}(1+m/2)d^{m/2} -\frac{2+m}{m}\Bigr)
	\\
	&\quad= \frac{4(1+m)}{m(2+3m)}\Bigl[ - \frac{m}{2} d^{1+m/2} 
	+ (1+m/2)d^{m/2} \Bigr]
	+ \frac{m}{2+m} + \frac{4}{2+m} \bMin
	- \frac{2+m}{m}.
	\end{align*}
	Since the expression in the square brackets decreases in $d$ (for $d>1$),
	in order to prove \eqref{eq:claim-dW-dd-boundary} we need only to check that it holds for $d = \dOpt(\bMin)$.
	Recalling first~\eqref{eq:dOpt-for-bMin}
	and using the formula for $\bMin$ a couple of times we see that
	\begin{align*}
	\MoveEqLeft[4]
	-\frac{4(1+m)}{m(2+3m)}\cdot\frac{m}{2}\dOpt(\bMin)^{1+m/2} 
	+ \frac{4}{2+m} \bMin \\
	&=-\frac{2(1+m)}{2+3m}\cdot
	\frac{4+3m}{2(1+m)}\bMin^{1+m/2} 
	+ \frac{4}{2+m} \cdot\frac{2(1+m)}{2+3m}\bMin^{1+m/2}\\
	&=\Bigl( - \frac{4+3m}{2+3m} + \frac{8(1+m)}{(2+m)(2+3m)}\Bigr)\bMin^{1+m/2} \\
	&= -\frac{m}{2+m}\bMin^{1+m/2}
	= -\frac{m}{2+m}\Bigl(\frac{2+3m}{2+2m}\Bigr)^{(2+m)/m} .
	\end{align*}
	On the other hand, it follows from~\eqref{eq:dOpt-for-bMin}
	and the formula for $\bMin$,
	that
	\[
	\dOpt(\bMin)^{m/2} = \frac{2+3m}{2(1+m)} \Bigl( \frac{4+3m}{2+2m}\Bigr)^{m/(2+m)}
	\]
	and thus
	\begin{equation*}
	\frac{4(1+m)}{m(2+3m)}(1+m/2)\dOpt(\bMin)^{m/2}
	=\frac{2+m}{m}\Bigl(\frac{4+3m}{2+2m}\Bigr)^{m/(2+m)} .
	\end{equation*}
	All in all, in order to finish the proof of~\eqref{eq:claim-dW-dd-boundary}
	it suffices to check whether
	\[
	-\frac{m}{2+m}\Bigl(\frac{2+3m}{2+2m}\Bigr)^{(2+m)/m} + \frac{2+m}{m}\Bigl(\frac{4+3m}{2+2m}\Bigr)^{m/(2+m)}
	+ \frac{m}{2+m} - \frac{2+m}{m} \geq 0.
	\]
	This inequality is true by Lemma~\ref{lem:aux-dOpt-for-bMin}.
	This finishes the proof of the lemma.
\end{proof}
The main result of this subsection reads as follows.

\begin{proposition}
	\label{prop:lower-bound-fixed-m-special-b-c-d}
	For $m\in\NNone$,
	\begin{align*}
	\weakconstrestricted{\Lambda_m}{\fclassspec{m}} 
	&=  
	\sup\bigl\{
	\XXXspecialbd(b,d,m) : (b,d)\in\specialbd{m} \bigr\}\\
	&= 	\sup\bigl\{
	\XXXspecialbd(b,\dOpt(b,m),m) : \bMin(m) \leq b < \bMax(m) \bigr\}.	
	\end{align*}
\end{proposition}

\begin{proof}
	Because of Lemma~\ref{lem:lower-bound-fixed-m-special-b-c-d-alternate-expression}
	we only need to prove the second equality.
	In order to maximize the function $\XXXspecialbd(b,d,m)$
	over $(b,d) \in \specialbd{m}$
	we analyze its behaviour when $d$ is fixed
	(dotted lines in Figure~\ref{fig:m-fixed-optimal-curve}).
	Since the numerator in the formula defining $\XXXspecialbd$ does not depend on $b$,
	we would like to  minimize the function 
	\[
	b \mapsto \XXXspecialbdDenominator(b) \coloneqq \XXXspecialbdDenominator(b,d,m) = - \frac{2m}{2+m}b +
	\frac{4(1+m)}{m(2+m)}(2b^{-m/2}  -1)d^{1+m/2} 
	+2t_0(b,m)
	\]
	(we treat $d$ and $m$ as fixed parameters and suppress them in notation).
	We proceed in a standard way, by checking when $\frac{\partial }{\partial b} \XXXspecialbdDenominator(b) = 0$ and what happens on the boundary.
	It follows from the formula for $t_0(b,m)$ (see~\eqref{eq:def-t_0-explicit-formula})
	that
	\[
	\frac{\partial}{\partial b} t_0(b,m) = \frac{2t_0(b,m) b^{-1-m/2}}{2 b^{-m/2} -1 }
	\]
	and thus
	\begin{align*}
	\frac{\partial}{\partial b} \XXXspecialbdDenominator(b)
	&= - \frac{2m}{2+m}
	- \frac{4(1+m)}{2+m}b^{-1-m/2}  d^{1+m/2} 
	+\frac{4t_0(b,m) b^{-1-m/2}}{2 b^{-m/2} -1 }\\
	&=   \frac{2b^{-1-m/2} 
	}{(2+m)(2 b^{-m/2} -1)}
	\cdot
	\Bigl( 
	-m(2 b^{-m/2} -1)b^{1+m/2} \\
	&\qquad\quad +2(2+m) t_0(b,m) - 2(1+m)(2 b^{-m/2} -1)d^{1+m/2}	\Bigr).
	\end{align*}		
	It follows directly from Lemma~\ref{lem:dOpt-definition} that
	$\frac{\partial }{\partial b} \XXXspecialbdDenominator(b,d,m) =0 $ if and only if  $d=\dOpt(b,m)$,
	$\frac{\partial }{\partial b} \XXXspecialbdDenominator(b,d,m) \geq 0$
	if $d = \dMin(b,m)$,
	and $\frac{\partial }{\partial b} \XXXspecialbdDenominator(b) \leq 0$
	if $d = \dMax(b,m)$.
	By Lemma~\ref{lem:tZero-divied-by-dOpt-monotone} 
	the function
	$b\mapsto \dOpt(b)$ is increasing 
	(and by Lemma~\ref{lem:dOpt-definition} its graph for $b\in(\bMin, \bMax)$ is contained in the set $\specialbd{m}$, cf.\ Figure~\ref{fig:m-fixed-optimal-curve}).
	Hence:
	\begin{itemize}
		\item $ \frac{\partial }{\partial b} \XXXspecialbdDenominator(b,d,m) \geq 0$ for $(b,d)\in\specialbd{m}$ 
		such that $d \leq \dOpt(b)$,
		\item $ \frac{\partial }{\partial b} \XXXspecialbdDenominator(b,d,m) \leq 0$ for $(b,d)\in\specialbd{m}$  such that $d \geq \dOpt(b)$.
	\end{itemize}
	
	Thus on each horizontal line $d= \text{const.} \in [\dOpt(\bMin), \infty)$
	the infimum of $\XXXspecialbdDenominator(b)$ is attained when
	$\frac{\partial }{\partial  b} \XXXspecialbdDenominator(b) = 0$ (cf.\ Figure~\ref{fig:m-fixed-optimal-curve}).
	On the other hand,
	on each horizontal line $d= \text{const.} \in [\bMin, \dOpt(\bMin)]$
	the infimum of $\XXXspecialbdDenominator(b)$ is attained  when $b$ is as small as possible (cf.\ Figure~\ref{fig:m-fixed-optimal-curve}).
	In order to finish the proof we invoke Lemma~\ref{lem:lower-bound-fixed-m-special-b-c-d-boundary-bMin-part-2}.   
\end{proof}

For reasons which will become apparent later, we  need one more technical result.
Denote $\bSh \coloneqq 7^{2/3}$
and
\begin{equation}
\label{eq:def-bMaxTilde}
\bMaxTilde = \bMaxTilde(m)
\coloneqq
\begin{cases}
\bMax(m) & \text{ if } m\geq 2,\\
\bSh & \text{ if } m = 1.
\end{cases}
\end{equation}
Observe that, for $m=1$,
$\bMin(1) = (5/4)^2 < 7^{2/3} < 4 = \bMax(1)$.

\begin{lemma}
	\label{lem:m=1-optimal-curve-far-end}
	For $m=1$,
	\begin{equation*}
	\weakconstrestricted{\Lambda_1}{\fclassspec{1}} 
	= 	\sup\bigl\{
	\XXXspecialbd(b,\dOpt(b,1),1) : \bMin(1) \leq b \leq \bMaxTilde(1) \bigr\}.	
	\end{equation*}
\end{lemma}

\begin{proof}
	From the proof of Theorem~\ref{thm:main-lower-bounds} we know that 	$\weakconstrestricted{\Lambda_1}{\fclassspec{1}} \geq 1.3$.
	Due to Proposition~\ref{prop:lower-bound-fixed-m-special-b-c-d} we only
	need to bound
	\[
	\XXXspecialbd(b,\dOpt(b,1),1)
	= \frac{\dOpt(b,1)-1}{ \frac{1}{3} - 3 \dOpt(b,1) - \frac{2}{3}b +
		\frac{8}{3}(2b^{-1/2}  -1) \dOpt(b,1)^{3/2} 
		+2t_0(b,1)}	
	\]
	for $b\in [\bSh, \bMax(1))$. 
	It follows from~\eqref{eq:fixed-m-supremum-W-optimal-curve} that
	\[
	\frac{8}{3}(2b^{-1/2}  -1) \dOpt(b)^{3/2} = -\frac{2}{3}(2b^{-1/2}  -1) b^{3/2} + 4t_0(b) 
	=
	-\frac{4}{3}b  +\frac{2}{3} b^{3/2} + 4t_0(b). 
	\]
	Thus
	\[
	\XXXspecialbd(b,\dOpt(b),1)
	= \frac{\dOpt(b)-1}{ \frac{1}{3} - 3 \dOpt(b) - 2b +
		\frac{2}{3}b^{3/2} + 6 t_0(b)}.	
	\]
	It is straightforward to see that 
	$\inf\{ 1/3 - 2b +(2/3) b^{3/2} : (5/4)^2 \leq b \leq 4\} = -7/3$.
	Moreover,  $t_0(b) \geq 4^{-1/3} \dOpt(b)$ by Lemma~\ref{lem:tZero-divied-by-dOpt-monotone}\
	(with $m=1$).
	Hence, for $b\in [\bSh, \bMax(1))$,
	\begin{align*}
	\XXXspecialbd(b,\dOpt(b),1)
	\leq \frac{\dOpt(b)-1}{ (6\cdot 4^{1/3}- 3 )\dOpt(b) - \frac{7}{3}}
	\leq
	\frac{\dOpt(\bSh)-1}{ (6\cdot 4^{1/3}- 3 )\dOpt(\bSh) - \frac{7}{3}}
	\leq 1.3,	
	\end{align*}
	where the second inequality  follows from monotonicity of $\dOpt(b)$ (see Lemma~\ref{lem:tZero-divied-by-dOpt-monotone})
	-- note that $7/3 > 6\cdot 4^{1/3}- 3$ and $\dOpt(\bSh) = \dOpt(7^{2/3},1) > 430$.
\end{proof}

\subsection{Adjoint: additional technical calculations}

This subsection is devoted to proving an enhanced version of Lemma~\ref{lem:lower-bound-fixed-m-special-b-c-d-alternate-expression-with-inequality-star}
(see Proposition~\ref{prop:lower-bound-fixed-m-special-b-c-d-star} below).
The proof is split into several lemmas (and uses the notation introduced in Subsection~\ref{sec:guesstimates-star}).

We first provide a more general and more exact version of Lemma~\ref{lem:lower-bound-fixed-m-special-b-c-d-alternate-expression-with-inequality-star}. %

\begin{proposition}
	\label{prop:lower-bound-fixed-m-star}
	For $m\in\NNone$,
	\begin{multline*}
	\weakconstrestricted{\Lambda_m^*}{\fclassStar{m}} \\
	=
	\sup\Bigl\{
	\frac{1 - \bGainStar + \cStar - \dGainStar}{ -\frac{2+m}{m} + \frac{m}{2+m}\bStar + \frac{4(1+m)}{m(2+m)} \bStar^{-m/2} + \int^{\cStar}_{\dStar} \bigl|-\frac{m}{2+m} + \dOneStar t^{-1-m/2} \bigr| dt}
	\Bigr\},
	\end{multline*}
	where the supremum is taken over   $0 <\dStar <  \cStar \leq  \bStar < 1$ and  
	\begin{align*}
	\dOneStar &= \dOneStar(\bStar,\cStar, m) = \frac{2(1+m)\cStar^{1+m/2}}{2+m} \Bigl(1 + \Bigl(\frac{\cStar}{\bStar}\Bigr)^{m/2}\bigl(1 - \bStar^{-1-m/2}\bigr)\Bigr),\\
	\bGainStar &=\bGainStar(\bStar,\cStar,m)
	\coloneqq  \cStar\lor\Bigl[ \Bigl[\bStar \bigl( - 1 - \frac{m}{2+m}  + \frac{2(1+m)}{2+m} \bStar^{-1-m/2}\bigr) ^{-2/m}\Bigr] \land \bStar\Bigr],\\
	\dGainStar &=\dGainStar(\bStar,\cStar,\dStar,m)
	\coloneqq \dStar \land \Bigl[ \dStar \bigl|  -1 - \frac{m}{2+m}  + \dOneStar \dStar^{-1-m/2}\bigr|^{-2/m}\Bigr].	 
	\end{align*}
\end{proposition}

\begin{proof}[Proof of Proposition~\ref{prop:lower-bound-fixed-m-star}]
	Consider a function  $f_*$ of the form \eqref{eq:weak-type-extremals-2-star}. 
	Due to scaling invariance,
	we may and do assume that $\aStar=1$.
	Thus, for  any $0 < \dStar < \cStar \leq \bStar  <1$ we consider the function
	\[
	\fStar(t) = \bigl( -\frac{m}{2+m} + \dOneStar t^{-1 -m/2}\bigr)\indbr{t\in (\dStar,\cStar]}
	+
	\bigl(\frac{m}{2+m} -\frac{2(1+m)}{2+m} t^{-1 -m/2}\bigr) \indbr{t\in (\bStar, 1]},
	\]
	where $	\dOneStar = \dOneStar(\bStar,\cStar,m)$ is as in the statement of the proposition.
	Let $\fZeroStar(t)=  \frac{m}{2+m} - \frac{2(1+m)}{2+m} t^{-1-m/2}$,
	$\fOneStar (t) =   -\frac{m}{2+m} + \dOneStar t^{-1 - m/2}$,
	be the functions from the proof of Lemma~\ref{lem:weak-type-extremals-2-star}.

	We start with identifying the set where $|\Lambda_m^* \fStar (t)| \geq 1$.
	Clearly
	$ \Lambda_m^* \fStar(t) =0$ for $t>1$ 
	and $|\Lambda_m^* \fStar(t)| = 1$ for $t\in (\dStar,\cStar) \cup (\bStar, 1)$
	by Lemma~\ref{lem:weak-type-extremals-2-star}.
	However, since $\lim_{t\to 1^-} \fStar(t) = -1 $,
	the function $\fStar$ is negative and increases on $(\bStar,1]$.
	Thus,
	although $\fStar(t) = 0$ for $t\in (\cStar,\bStar)$,
	it may happen that $\Lambda_m^* \fStar(t) \leq -1$ for some $t\in(\cStar,\bStar)$
	due to the (negative) mass acumulated in $\int_{\bStar}^1 \fStar(s) s^{-1-m/2} ds$.
	Indeed, for $t\in (\cStar,\bStar)$ we have
	\begin{align*}
	\Lambda_m^* \fStar(t)
	&= (1+m){t^{m/2}} \int_{\bStar}^1 \fStar(s) s^{-1-m/2}ds
	= (1+m){t^{m/2}} \int_{\bStar}^1 \fZeroStar(s) s^{-1-m/2}ds\\
	&= \frac{t^{m/2}}{\bStar^{m/2}} \bigl( \Lambda_m^* \fZeroStar(\bStar) + \fZeroStar(\bStar)\bigr)
	= \frac{t^{m/2}}{\bStar^{m/2}} \bigl( 1 + \frac{m}{2+m}  - \frac{2(1+m)}{2+m} \bStar^{-1-m/2}\bigr), 
	\end{align*}
	which is negative and decreases in $t$.
	Hence, the condition $|\Lambda_m^* \fStar(t)| \geq 1$ and $t\in(\cStar,\bStar)$ is equivalent to 
	$t\in(\bGainStar,\bStar)$, where $\bGainStar$ is as in the statement of the proposition.

	Similarly,  by the very definition of the coefficient $\dOneStar$ and the calculations from the proof of Lemma~\ref{lem:weak-type-extremals-2-star}, for $t\in(0,\dStar)$,
	\begin{align*}
	\Lambda_m^* \fStar(t)
	&= (1+m) t^{m/2} \int_{\dStar}^1 \fStar(s) s^{-1-m/2}ds
	= (1+m) t^{m/2} \int_{\dStar}^1 \fOneStar(s) s^{-1-m/2}ds\\
	&= \frac{t^{m/2}}{\dStar^{m/2}} \bigl( \Lambda_m^* \fOneStar(\dStar) + \fOneStar(\dStar)\bigr)
	= \frac{t^{m/2}}{\dStar^{m/2}} \bigl( -1 - \frac{m}{2+m}  + \dOneStar \dStar^{-1-m/2}\bigr).
	\end{align*}
	The absolute value of this expression increases in $t$,
	so the condition $|\Lambda_m^* \fStar(t)| \geq 1$ and $t\in(0,\dStar)$
	is equivalent to $t\in(\dGainStar, \dStar)$.

	We calculate 
	$\|\fStar\|_{L^1([0, \infty))}$ as in the proof of Lemma~\ref{lem:lower-bound-fixed-m-special-b-c-d-alternate-expression-with-inequality-star} above.
	This ends the proof.
\end{proof}

\begin{lemma}
	\label{lem:lower-bound-fixed-m-special-b-c-d-alternate-expression-star}
	For $m\in\NNone$,
	\begin{equation*}
	\weakconstrestricted{\Lambda_m^*}{\fclassspecStar{m}} 
	=  
	\sup\bigl\{
	\XXXspecialbdStar(\bStar,\dStar,m) : (\bStar,\dStar)\in\specialbdStar{m} \bigr\}.
	\end{equation*}
\end{lemma}

\begin{proof}
	We proceed as in the proof of Lemma~\ref{lem:lower-bound-fixed-m-special-b-c-d-alternate-expression-with-inequality-star}
	with the only difference being that 
	Proposition~\ref{prop:lower-bound-fixed-m-star} (and its proof)
	together with 	Definition~\ref{def:f-class-spec-star} actually yields
	that for $\fStar\in\fclassspecStar{m}$,
	\begin{equation*}
	\bigl|\{t\in[0,\infty) : |\Lambda_m^* \fStar(t)|\geq 1\} \bigr| = 1 - \dStar.
	\qedhere
	\end{equation*}
\end{proof}

Denote 
$\bShStar \coloneqq (27/(54 - 16(2-7^{1/3})^3))^{2/3}$ and
\begin{equation}
\label{eq:def-bMinTildeStar}
\bMinTildeStar = \bMinTildeStar(m)
\coloneqq
\begin{cases}
\bMinStar(m) & \text{ if } m\geq 2,\\
\bShStar & \text{ if } m = 1.
\end{cases}
\end{equation}
Observe that, for $m=1$,
\[
\bMinStar(1) = 2^{-2/3}
< 0.63 <
\bMinTildeStar = \bShStar \approx 0.63004
< 0.68 <
(4/7)^{2/3}
=
\bMaxStar(1).
\]
We have the following counterpart of Lemma~\ref{lem:dOpt-definition}.
Note that somewhat surprisingly a distinction between the cases $m=1$ and $m\geq 2$ appears
in the formulation.

\begin{lemma}
	\label{lem:dOpt-definition-star}
	If $m\geq 2$ and $\bStar\in (\bMinStar(m), \bMaxStar(m)]$,
	then
	\begin{align}
	\label{eq:double-inequality-for-defining-dOptStar}
	\MoveEqLeft[2]
	2(1+m)(2\bStar^{1+m/2} -1) \dMinStar(\bStar,m)^{-m/2} \\
	&\geq -(2+m)(2\bStar^{1+m/2} - 1) \bStar^{-m/2}  
	+ 2m \tZeroStar(\bStar,m)\nonumber\\
	&\geq 2(1+m)(2\bStar^{1+m/2} -1) \dMaxStar(\bStar,m)^{-m/2}.\nonumber
	\end{align}	
	In particular, for $m\geq 2$,  
	for every $\bStar\in (\bMinStar(m), \bMaxStar(m)]$ 
	there exists  
	$\dOptStar = \dOptStar(\bStar,m)\in [\dMinStar(m),\dMaxStar(m)]$ 
	such that
	\begin{align}
	\label{eq:fixed-m-supremum-W-optimal-curve-star}
	2(1+m)(2 \bStar^{1+m/2} -1)&\dOptStar(\bStar, m)^{-m/2}\\
	&=-(2+m)(2\bStar^{1+m/2} -1)\bStar^{-m/2} 
	+2m \tZeroStar(\bStar,m).\nonumber
	\end{align}
	
	If $m=1$ and $\bStar \in [\bShStar, \bMaxStar(1)]$, then the double inequality~\eqref{eq:double-inequality-for-defining-dOptStar} also holds.
	In particular, for every $\bStar\in [\bShStar, \bMaxStar(1)]$ 
	there exists $\dOptStar = \dOptStar(\bStar,1)\in [\dMinStar(1),\dMaxStar(1)]$ 
	such that~\eqref{eq:fixed-m-supremum-W-optimal-curve-star} holds.
	Moreover,
	\[
	\dOptStar(\bShStar,1) = \dMinStar(\bShStar, 1).
	\]
	Furthermore, for every $\bStar\in (\bMinStar(1),\bShStar)$
	there exists $\dOptStar = \dOptStar(\bStar,1)\in (0,\dMinStar(1))$ 
	such that~\eqref{eq:fixed-m-supremum-W-optimal-curve-star} holds.
\end{lemma}

\begin{proof}
	We first prove that for all $m\geq 1$ and all $\bStar\in (\bMinStar(m), \bMaxStar(m)]$ the right-hand side inequality of~\eqref{eq:double-inequality-for-defining-dOptStar} holds.
	Recall from~\eqref{eq:def-t_0-explicit-formula-star} 
	that
	$\tZeroStar = \tZeroStar(\bStar,m) = \dMaxStar$
	and
	\begin{equation*}
	2(1+m)(2 \bStar^{1+m/2} -1) \dMaxStar^{-m/2} 
	=  m\dMaxStar =  m\tZeroStar.
	\end{equation*} 
	Consequently,
	\begin{equation*}
	-(2+m)(2\bStar^{1+m/2} - 1) \bStar^{-m/2} \geq  -(2+m)(2\bStar^{1+m/2} - 1) \tZeroStar^{-m/2} = -\frac{m(2+m)}{2(1+m)} \tZeroStar.
	\end{equation*} 
	Thus
	\begin{multline*}
	-(2+m)(2\bStar^{1+m/2} - 1) \bStar^{-m/2}  
	+ 2m \tZeroStar
	- 2(1+m)(2\bStar^{1+m/2} -1) \dMaxStar^{-m/2}\\
	\geq \Bigl(-\frac{m(2+m)}{2(1+m)} + 2m - m \Bigr)\cdot t_0 = \frac{m^2}{2(1+m)} \cdot \tZeroStar \geq 0,
	\end{multline*}
	which proves the first part of the assertion.
	
	Let us now prove  that the left-hand side inequality of~\eqref{eq:double-inequality-for-defining-dOptStar} holds 
	for $m\geq 2$ and $\bStar\in (\bMinStar(m), \bMaxStar(m)]$.
	It follows from the definition~\eqref{eq:def-dMin-star} 
	that  $\dMinStar = \dMinStar(\bStar,m)$ satisfies 
	\begin{equation*}
	2(1+m)(2 \bStar^{1+m/2} -1)\dMinStar^{-m/2}
	= (4+3m) \dMinStar.
	\end{equation*} 
	In order to prove the second part of the assertion  
	it suffices to check that
	\[
	(4+3m) \dMinStar \geq 2m \tZeroStar.
	\]
	To this end, observe that 
	\begin{align*}
	2(1+m)(2 \bStar^{1+m/2} -1)\Bigl( \frac{4+3m}{2m} \dMinStar\Bigr)^{-1-m/2}  
	&= (4+3m)\Bigl( \frac{4+3m}{2m} \Bigr)^{-1-m/2}\\
	&= 2m \cdot\bigl(\frac{2m}{4+3m} \bigr)^{m/2}\\
	&\leq 
	\begin{cases}
	2m \cdot (2/3)^{m/2} & \text{if } m\geq 4,\\
	2m\cdot (6/13)^{3/2} & \text{if } m = 3,\\
	2m \cdot (4/10) & \text{if } m=2,
	\end{cases}
	\\
	&\leq 
	m \quad \text{(for  $m\geq 2$)}.
	\end{align*}
	This inequality means exactly that $\tZeroStar \leq \frac{4+3m}{2m} \dMinStar$ (for $m\geq 2$),
	which finishes the proof of our claim in the case $m\geq 2$.
	The in particular part of the assertion follows immediately.
	
	Next we check that for $m=1$ and $\bStar\in [\bShStar, \bMaxStar(1)]$,
	\begin{align*}
	\MoveEqLeft[2]
	2(1+m)(2\bStar^{1+m/2} -1) \dMinStar(\bStar,m)^{-m/2} \\
	&\overset{?}{\geq}  -(2+m)(2\bStar^{1+m/2} - 1) \bStar^{-m/2}  
	+ 2m \tZeroStar(\bStar,m).
	\end{align*}
	Plugging in $m=1$ and the definitions of $\dMinStar$ and $\tZeroStar$ we arrive at
	\begin{equation*}
	4^{2/3} 7^{1/3} (2\bStar^{3/2} -1)^{2/3}  \overset{?}{\geq} 
	-3(2\bStar^{3/2} - 1) \bStar^{-1/2}
	+ 2 \cdot4^{2/3}(2\bStar^{3/2} -1)^{2/3},
	\end{equation*}
	which is equivalent to
	\begin{equation*}
	3(2\bStar^{3/2} - 1)^{1/3} 
	\overset{?}{\geq}  4^{2/3}(2 - 7^{1/3}) \bStar^{1/2},
	\end{equation*}
	which after taking the cube of both sides and simplifying reduces to the inequality $\bStar\geq\bShStar$.
	It follows immediately that for $m=1$ and  for every $\bStar\in [\bShStar, \bMaxStar(1)]$ 
	there exists $\dOptStar = \dOptStar(\bStar,1)\in [\dMinStar(1),\dMaxStar(1)]$ 
	satisfying~\eqref{eq:fixed-m-supremum-W-optimal-curve-star};
	moreover, by the very definition, 
	\[
	\dOptStar(\bShStar,1) = \dMinStar(\bShStar, 1).
	\]
	
	Furthermore, the preceding calculation also shows that
	if $m=1$ and $\bStar\in (\bMinStar(1),\bShStar)$, then
	\begin{align*}
	0 &< 2(1+m)(2\bStar^{1+m/2} -1) \dMinStar(\bStar,m)^{-m/2} \\
	& <  -(2+m)(2\bStar^{1+m/2} - 1) \bStar^{-m/2}  
	+ 2m \tZeroStar(\bStar,m).
	\end{align*}
	Thus, for $m=1$ and $\bStar\in (\bMinStar(1),\bShStar)$
	there exists $\dOptStar = \dOptStar(\bStar,1)\in (0,\dMinStar(1))$ satisfying~\eqref{eq:fixed-m-supremum-W-optimal-curve-star}.
\end{proof}

Next we consider the behavior of the function $\XXXspecialbdStar$
on the part of the boundary of the  set $\specialbdStar{m}$ 
which is the interval obtained by considering $\bStar \to \bMaxStar(m) ^-$.

\begin{lemma}
	\label{lem:lower-bound-fixed-m-special-b-c-d-boundary-bMin-star}
	For $m\in\NNone$,
	\begin{align}
	\nonumber
	\dMinStar(\bMaxStar(m), m) &=  \Bigl(\frac{m(2+2m)}{(4+3m)^2}\Bigr)^{2/(2+m)}
	= \bMaxStar(m) \Bigl(\frac{m}{4+3m}\Bigr)^{2/(2+m)}, \\
	\nonumber
	\dMaxStar(\bMaxStar(m),m) &= \bMaxStar,\\
	\label{eq:dOpt-for-bMin-star}
	\dOptStar(\bMaxStar(m),m) &= \bMaxStar(m) \cdot \Bigl( \frac{2+2m}{2+3m}\Bigr)^{2/m}.
	\end{align}
\end{lemma}

\begin{proof}
	We have $\bMaxStar^{1+m/2} = (2+2m)/(4+3m)$, so
	\[
	2\bMaxStar^{1+m/2}  - 1 = 2 \frac{2+2m}{4+3m} - 1 = 
	\frac{m}{4+3m}.
	\]
	The formulas for $\dMinStar(\bMaxStar)$ and $\dMaxStar(\bMaxStar)$ follow immediately.
	Moreover, 	note that $\tZeroStar(\bMaxStar) = \dMaxStar(\bMaxStar) = \bMaxStar$.
	
	Using the preceding observations we can rewrite~\eqref{eq:fixed-m-supremum-W-optimal-curve-star}
	as
	\begin{align*}  
	2(1+m)\frac{m}{4+3m} \dOptStar(\bMaxStar)^{-m/2} 
	&= - (2+m) \frac{m}{4+3m} \bMaxStar^{-m/2}   
	+ 2m \bMaxStar\\
	&= - (2+m) \frac{m}{4+3m} \bMaxStar^{-m/2}   
	+ 2m  \frac{2+2m}{4+3m} \bMaxStar^{-m/2}, 
	\end{align*}
	which yields the claim~\eqref{eq:dOpt-for-bMin-star}.
\end{proof}

\begin{remark}	In particular,
	\begin{equation}
	\label{eq:claim-dOpt-for-bMin-star}
	\dMinStar(\bMaxStar(m),m)  < \dOptStar(\bMaxStar(m),m)  < \dMaxStar(\bMaxStar(m),m).
	\end{equation}	
	Indeed, this is equivalent to 
	\[
	\Bigl( \frac{m}{4+3m}\Bigr)^{2/(2+m)} < \Bigl( \frac{2+2m}{2+3m}\Bigr)^{2/m} < 1.
	\]
	The right-hand side inequality holds true for all $m\in \NN$.
	The left-hand side inequality is satisfied for $m=1$ and for $m\geq 2$ we can write
	\[
	\Bigl( \frac{m}{4+3m}\Bigr)^{2/(2+m)}
	\leq 	\Bigl( \frac{1}{3}\Bigr)^{2/(2+m)} < 	\Bigl( \frac{2}{3}\Bigr)^{2/m} \leq  \Bigl( \frac{2+2m}{2+3m}\Bigr)^{2/m}.
	\]
\end{remark}

We need the following observation
(note that although for $m=1$ and $\bStar\in (\bMinStar(1),\bShStar)$ we have
$\dOptStar(\bStar, 1) < \dMinStar(\bStar, 1)$, the
formula defining $\dOptStar(\bStar,m)$ 
does not change,
so there is no distinciton between the cases $m\geq 2$ and $m=1$ in the formulation).

\begin{lemma}
	\label{lem:tZero-divied-by-dOpt-monotone-star}
	For $m\geq 1$ and 
	$\bStar\in (\bMinStar, \bMaxStar]$ we have
	\[
	2^{2/m} > \frac{\tZeroStar(\bStar)}{\dOptStar(\bStar)} \geq \frac{\tZeroStar(\bMaxStar)}{\dOptStar(\bMaxStar)} 
	=
	\Bigl(\frac{2+3m}{2+2m}\Bigr)^{2/m}.
	\]
	Moreover, the function $\bStar\mapsto \dOptStar(\bStar)$ is increasing,  $\bStar\mapsto \tZeroStar(\bStar)/\dOptStar(\bStar)$ is decreasing and  
	\[
	\lim_{\bStar\to\bMinStar}\frac{\tZeroStar(\bStar)}{\dOptStar(\bStar)}
	= 2^{-2/(2+m)}.
	\]
\end{lemma}

\begin{proof}
	The expression for  $\tZeroStar(\bMaxStar)/\dOptStar(\bMaxStar)$ 
	follows from Lemma~\ref{lem:lower-bound-fixed-m-special-b-c-d-boundary-bMin-star}.
	The definition of $\tZeroStar$ (see~\eqref{eq:def-t_0-explicit-formula-star}) 
	can be rewritten as
	\[
	2(1+m) (2\bStar^{1+m/2}  -1 ) = m  \tZeroStar(\bStar)^{1+m/2}.
	\]
	Hence, if $\bStar\to \bMinStar = 2^{-2/(2+m)}$, then $\tZeroStar(\bStar) \to 0^+$ and $\tZeroStar(\bStar)/\bStar\to 0^+$.
	Moreover, by multiplying both sides by $\bStar^{-1-m/2}$ we see that the function $\bStar\mapsto \tZeroStar(\bStar)/\bStar$ is monotone (the expression decreases when $\bStar$ decreases).
	Substituting the above expression for $2(1+m)(2\bStar^{1+m/2} - 1)$
	into~\eqref{eq:fixed-m-supremum-W-optimal-curve-star} we arrive at:
	\begin{equation*}
	\frac{ \tZeroStar(\bStar)^{1+m/2}}{ \dOptStar(\bStar)^{m/2}}
	= -\frac{(2+m)\tZeroStar(\bStar)^{1+m/2}}{\bStar^{m/2}}   
	+ 2 \tZeroStar(\bStar,m). 
	\end{equation*}
	If we divide both sides by $\tZeroStar(\bStar)$ and let $\bStar\to \bMinStar$, we can conclude that 
	\[
	\lim_{\bStar\to\bMinStar}\frac{\tZeroStar(\bStar)}{\dOptStar(\bStar)}
	= 2^{2/m}.
	\]
	and moreover the convergence is monotone (the expression increases when $\bStar$ decreases).
	In particular, since $\tZeroStar(\bStar)$ decreases when $\bStar$ decreases, $\dOptStar(\bStar)$ also has to  decrease when $\bStar$ decreases. 
\end{proof}

\begin{lemma}
	\label{lem:lower-bound-fixed-m-special-b-c-d-boundary-bMin-part-2-star}
	For $m\in\NNone$ and $\dStar\in \bigl[ \dOptStar(\bMaxStar(m),m), \bMaxStar\bigr]$,
	\begin{equation}
	\label{eq:claim-W-boundary-star}
	\XXXspecialbdStar\bigl(\bMaxStar(m), \dOptStar(\bMaxStar(m),m), m\bigr)
	\geq 
	\XXXspecialbdStar(\bMaxStar(m), \dStar, m) .
	\end{equation}
\end{lemma}

\begin{proof}
	From Lemma~\ref{lem:lower-bound-fixed-m-special-b-c-d-boundary-bMin-star}
	and \eqref{eq:def-t_0-explicit-formula-star}
	we know that $\tZeroStar(\bMaxStar) = \dMaxStar(\bMaxStar) = \bMaxStar$.
	Recall also that $\bMaxStar^{1+m/2} = (2+2m)/(4+3m)$ and
	\[
	\frac{4(1+m)}{m(2+m)}\bigl(2\bMaxStar^{1+m/2}  - 1\bigr) = 
	\frac{4(1+m)}{(2+m)(4+3m)}.
	\]
	It follows that
	\[
	\XXXspecialbdStar(\bMaxStar,\dStar,m)= \frac{1-\dStar}{
		\frac{4(1+m)}{(2+m)(4+3m)}\dStar^{-m/2} +\frac{m}{2+m}\dStar - \frac{2+m}{m} + \frac{4}{m} \bMaxStar}.
	\]
	In order to prove the inequality~\eqref{eq:claim-W-boundary-star} 
	we claim that
	\begin{equation}
	\label{eq:claim-dW-dd-boundary-star}
	\frac{\partial }{\partial \dStar} \XXXspecialbdStar(\bMaxStar, \dStar , m) \leq 0  \qquad\text{ for } \quad \dStar \in [\dOptStar(\bMaxStar), \dMaxStar].
	\end{equation}
	Indeed, 
	$-\frac{\partial }{\partial \dStar} \XXXspecialbdStar(\bMaxStar, \dStar , m)$
	has the same sign as
	\begin{align*}
	\MoveEqLeft[4]
	\frac{4(1+m)}{(2+m)(4+3m)}\dStar^{-m/2} +\frac{m}{2+m}\dStar -\frac{2+m}{m} + \frac{4}{m} \bMaxStar\\
	&\qquad \qquad -(\dStar-1)\cdot\bigl(\frac{4(1+m)}{(2+m)(4+3m)}(-m/2)\dStar^{-1-m/2} +\frac{m}{2+m}\bigr)\\
	&= \frac{4(1+m)}{(2+m)(4+3m)}\Bigl[ (1+m/2)\dStar^{-m/2} 
	- (m/2)\dStar^{-1-m/2} \Bigr]\\
	&\qquad\qquad
	- \frac{2+m}{m} + \frac{4}{m} \bMaxStar+ \frac{m}{2+m}.
	\end{align*}
	Since the expression in the square brackets increases in $\dStar$ (for $\dStar <1$),
	in order to prove \eqref{eq:claim-dW-dd-boundary-star} we need only to check its validity for $\dStar = \dOptStar(\bMaxStar)$.	
	Recalling first~\eqref{eq:dOpt-for-bMin-star}
	and using the formula for $\bMaxStar$ a couple of times we see that
	\begin{align*}
	\MoveEqLeft[4]
	\frac{2(1+m)}{4+3m}\dOptStar(\bMaxStar)^{-m/2} 
	+ \frac{4}{m} \bMaxStar \\
	&=\frac{2(1+m)}{4+3m}\cdot
	\frac{2+3m}{2(1+m)}\bMaxStar^{-m/2} 
	+ \frac{4}{m} \cdot\frac{2(1+m)}{4+3m}\bMaxStar^{-m/2}\\
	&= \bigl(  \frac{2+3m}{4+3m} + \frac{8(1+m)}{m(4+3m)}\bigr)\bMaxStar^{-m/2}\\
	&= \frac{2+m}{m}\bMaxStar^{-m/2} 
	= \frac{2+m}{m}\Bigl(\frac{4+3m}{2+2m}\Bigr)^{m/(2+m)}.
	\end{align*}
	On the other hand, it follows from~\eqref{eq:dOpt-for-bMin-star}
	and the formula for $\bMaxStar$,
	that
	\[
	\dOptStar(\bMaxStar)^{-1-m/2} = \frac{4+3m}{2(1+m)} \Bigl( \frac{2+2m}{2+3m}\Bigr)^{(2+m)/m}
	\]
	and thus
	\begin{equation*}
	-\frac{4(1+m)}{(2+m)(4+3m)}(m/2)\dOptStar(\bMaxStar)^{-1-m/2}
	=-\frac{m}{2+m}\Bigl(\frac{2+3m}{2+2m}\Bigr)^{(2+m)/m} .
	\end{equation*}
	All in all, in order to finish the proof of~\eqref{eq:claim-dW-dd-boundary-star}
	it suffices to check whether
	\[
	\frac{2+m}{m}\Bigl(\frac{4+3m}{2+2m}\Bigr)^{\frac{m}{2+m}} - \frac{2+m}{m} \geq 
	\frac{m}{2+m}\Bigl(\frac{2+3m}{2+2m}\Bigr)^{\frac{2+m}{m}} - \frac{m}{2+m}.
	\]
	This inequality is true by Lemma~\ref{lem:aux-dOpt-for-bMin}.
	This finishes the proof of the lemma.
\end{proof}

\begin{lemma} 
	\label{lem:lower-bound-fixed-m-special-b-c-d-boundary-bSh-star}
	For $m=1$ and $\bStar\in (\bMinStar(1), \bShStar]$,
	\begin{equation}
	\label{eq:claim-W-boundary-bSh-star}
	\XXXspecialbdStar\bigl(\bStar, \dMinStar(\bStar,1), 1\bigr)
	\leq 
	\XXXspecialbdStar\bigl(\bShStar, \dMinStar(\bShStar,1), 1\bigr).
	\end{equation}
\end{lemma}

\begin{proof}
	Throughout the proof $m=1$.
	We have $\dMinStar(\bStar) = (\frac{4}{7}(2\bStar^{3/2} - 1))^{2/3}$ \
	and
	$\tZeroStar(\bStar) = (4(2\bStar^{3/2} - 1))^{2/3}$ (cf.\ \eqref{eq:def-dMin-star}, \eqref{eq:def-t_0-explicit-formula-star}).
	Thus
	\begin{align*}
	\XXXspecialbdStar\bigl(\bStar, \dMinStar(\bStar), 1\bigr)
	&=
	\frac{1-\dMinStar}{- 3  + \frac{1}{3}\dMinStar + 6 \bStar +
		\frac{8}{3}(2\bStar^{3/2}  -1)\dMinStar^{-1/2} 
		- 2\tZeroStar(\bStar)}\\
	&=\frac{1-\dMinStar}{6 \bStar - 3 + (\frac{1}{3} + \frac{8}{3}\frac{7}{4} -2 \cdot 7^{2/3})\dMinStar}\\
	&= \frac{1-\dMinStar}{6 \bStar - 3 + (5 -2 \cdot 7^{2/3})\dMinStar}.
	\end{align*}
	Observe that $\lim_{\bStar\to \bMinStar^+}\dMin(\bStar) = 0$
	and $\bStar = (\frac{7}{8}\dMinStar(\bStar)^{3/2} + \frac{1}{2})^{2/3}$.
	Denote $u = \dMinStar(\bStar)$. Our goal is to prove that
	the function 
	\[
	u\mapsto \frac{1-u}{6 (\frac{7}{8}u^{3/2} + \frac{1}{2})^{2/3} - 3 + (5 -2 \cdot 7^{2/3})u}
	\]
	increases on $(0, \dMinStar(\bShStar)]$.
	Its derivative with respect to $u$ has the same sign as
	\begin{align*}
	\MoveEqLeft[6]
	-6 (\tfrac{7}{8}u^{3/2} + \tfrac{1}{2})^{2/3} + 3 - (5 -2 \cdot 7^{2/3})u\\
	\qquad 
	-(1-u)&\Bigl(6 (\tfrac{7}{8}u^{3/2} + \tfrac{1}{2})^{-1/3}\tfrac{7}{8} u^{1/2} + 5 -2 \cdot 7^{2/3}\Bigr)\\
	&=
	-6 (\tfrac{7}{8}u^{3/2} + \tfrac{1}{2})^{2/3}  - 2 + 2 \cdot 7^{2/3}
	-6 (\tfrac{7}{8}u^{3/2} + \tfrac{1}{2})^{-1/3}\tfrac{7}{8} u^{1/2} \\
	& \qquad +6 (\tfrac{7}{8}u^{3/2} + \tfrac{1}{2})^{-1/3}\tfrac{7}{8} u^{3/2} \\
	&=4(\tfrac{7}{8}u^{3/2} + \tfrac{1}{2})^{-1/3}
	\cdot \Bigl[ - 12 + 4(7^{2/3} -1)(7u^{3/2} + 4)^{1/3} -21 u^{1/2}\Bigr].
	\end{align*}
	The expression in the square brackets is positive for $u\in(0,\dMinStar(\bShStar)]$ since it is greater or equal than
	\begin{equation*}
	- 12 + 4(7^{2/3} -1)4^{1/3} -21 u^{1/2}
	\geq 4 - 21 u^{1/2} \geq 4 -21 \dMinStar(\bShStar)^{1/2}
	\geq 2.9,
	\end{equation*}
	where the last inequality follows by plugging in the numerical value of $\dMinStar(\bShStar)$.
	This finishes the proof.
\end{proof}

Recall 
that $\bMinTildeStar(m)=\bMinStar(m)$ for $m\geq 2$
and $\bMinTildeStar(1)=\bShStar > \bMinStar(1)$.

\begin{proposition}
	\label{prop:lower-bound-fixed-m-special-b-c-d-star}
	For $m\in\NNone$, 
	\begin{align*}
	\weakconstrestricted{\Lambda_m^*}{\fclassspecStar{m}} 	
	&= 
	\sup\bigl\{
	\XXXspecialbdStar(\bStar,\dStar,m) : (\bStar,\dStar)\in\specialbdStar{m} \bigr\}\\
	&=
	\sup\bigl\{
	\XXXspecialbdStar(\bStar,\dOptStar(\bStar,m),m) : \bMinTildeStar(m) < \bStar \leq \bMaxStar(m) \bigr\}.
	\end{align*}
\end{proposition}

\begin{proof}
	Because of Lemma~\ref{lem:lower-bound-fixed-m-special-b-c-d-alternate-expression-star}
	we only need to prove the second part of the assertion.
	Recall from the definition~\eqref{eq:def-bMinTildeStar}
	that $\bMinTildeStar = \bMinStar$ for $m\geq 2$.
	In order to maximize the function $\XXXspecialbdStar(\bStar,\dStar,m)$
	over $\specialbdStar{m}$
	we analyze its behaviour when $\dStar$ is fixed.
	Since the numerator in the formula defining $\XXXspecialbdStar$ does not depend on $\bStar$,
	we would like to  minimize the function 
	\begin{align*}
	\bStar \mapsto \XXXspecialbdDenominatorStar(\bStar) &\coloneqq
	\XXXspecialbdDenominatorStar(\bStar,\dStar, m) \\
	& = \frac{2(2+m)}{m}\bStar +
	\frac{4(1+m)}{m(2+m)}(2\bStar^{1+m/2}  -1)\dStar^{-m/2} 
	- 2\tZeroStar(\bStar,m).
	\end{align*}
	(we treat $\dStar$ and $m$ as fixed parameters and suppress them in notation).
	We proceed in a standard way, by checking when $\frac{\partial }{\partial \bStar} \XXXspecialbdDenominatorStar(\bStar) = 0$ and what happens on the boundary.
	It follows from the formula defining $\tZeroStar$ (see~\eqref{eq:def-t_0-explicit-formula-star}) that
	\begin{equation*}
	\frac{\partial }{\partial \bStar}\tZeroStar(\bStar)
	= \tZeroStar(\bStar) \frac{2}{2+m} \frac{2(1+m/2) \bStar^{m/2}}{2 \bStar^{1+m/2} -1}
	=  \frac{2 \tZeroStar(\bStar)\bStar^{m/2}}{2 \bStar^{1+m/2} -1}
	\end{equation*}
	and thus
	\begin{align*}
	\frac{\partial }{\partial \bStar} \XXXspecialbdDenominatorStar(\bStar)
	&=  \frac{2(2+m)}{m}
	+ \frac{4(1+m)}{m}\bStar^{m/2}  \dStar^{-m/2} 
	-\frac{4\tZeroStar(\bStar,m) \bStar^{m/2}}{2 \bStar^{1+m/2} -1 }\\
	&=   \frac{2\bStar^{m/2} 
	}{m(2 \bStar^{1+m/2} -1)}
	\cdot
	\Bigl( 
	(2+m)(2\bStar^{1+m/2} -1)\bStar^{-m/2} \\
	&\qquad -2m \tZeroStar(\bStar,m)
	+ 2(1+m)(2 \bStar^{1+m/2} -1)\dStar^{-m/2}	\Bigr).
	\end{align*}
	It follows directly from Lemma~\ref{lem:dOpt-definition-star} that 
	$\frac{\partial }{\partial  \bStar} \XXXspecialbdDenominatorStar(\bStar,\dStar,m) = 0$
	if and only if $\dStar = \dOptStar(\bStar)$ (in particular, we must have $\bStar > \bMinTildeStar$),
	$ \frac{\partial }{\partial \bStar} \XXXspecialbdDenominator(\bStar,\dStar) \leq 0$ for $\dStar = \dMaxStar(\bStar,m)$,
	$ \frac{\partial }{\partial \bStar} \XXXspecialbdDenominatorStar(\bStar,\dStar, m) \geq 0$
	if $\dStar = \dMinStar(\bStar)$ and $\bStar >\bMinTildeStar$.
	By Lemma~\ref{lem:tZero-divied-by-dOpt-monotone-star} 
	the function
	$\bStar\mapsto \dOptStar(\bStar)$ is increasing 
	(and by Lemma~\ref{lem:dOpt-definition-star} its graph for $\bStar\in(\bMinTildeStar, \bMaxStar)$ is contained in the set $\specialbdStar{m}$).
	Hence:
	\begin{itemize}
		\item $ \frac{\partial }{\partial \bStar} \XXXspecialbdDenominatorStar(\bStar,\dStar, m) \geq 0$ for $(\bStar,\dStar)\in\specialbdStar{m}$ 
		such that $\bStar > \bMinTildeStar$ and $\dStar \leq \dOptStar(\bStar)$,
		\item $ \frac{\partial }{\partial \bStar} \XXXspecialbdDenominatorStar(\bStar,\dStar,m) \leq 0$ for $(\bStar,\dStar)\in\specialbdStar{m}$  such that $\bStar>\bMinTildeStar$ and $\dStar \geq \dOptStar(\bStar)$,
		\item $ \frac{\partial }{\partial \bStar} \XXXspecialbdDenominatorStar(\bStar,\dStar,m) \leq 0$ for $(\bStar,\dStar)\in\specialbdStar{m}$  such that $\bStar <\bMinTildeStar$.
	\end{itemize}
	Thus on each horizontal line $\dStar= \text{const.} \in (\dOpt(\bMinTildeStar),\dOptStar(\bMaxStar)]$
	the infimum of $\XXXspecialbdDenominatorStar(\bStar)$ is attained when
	$\frac{\partial }{\partial  \bStar} \XXXspecialbdDenominatorStar(\bStar) = 0$.
	On the other hand,
	on each horizontal line $\dStar= \text{const.} \in [\dOptStar(\bMaxStar), \bMaxStar )$
	the infimum of $\XXXspecialbdDenominatorStar(\bStar)$ is attained when $\bStar$ is as large as possible.
	If $m=1$ and $\bMinStar < \bMinTildeStar$, then 
	on each horizontal line $\dStar= \text{const.} \in (0, \dOpt(\bMinTildeStar)]$
	the infimum of $\XXXspecialbdDenominatorStar(\bStar)$ is attained when
	$\bStar$ is as large as possible.
	In order to finish the proof we invoke Lemma~\ref{lem:lower-bound-fixed-m-special-b-c-d-boundary-bMin-part-2-star}
	and, if $m=1$, Lemma~\ref{lem:lower-bound-fixed-m-special-b-c-d-boundary-bSh-star}. 
\end{proof}

\subsection{Proof of Remark~\ref{rem:W-vs-W-star}}

\begin{proof}[Proof of Remark~\ref{rem:W-vs-W-star}]	
	We shall prove that for an appropriate choice of $\bStar=\bStar(b,m)$
	we have $\XXXspecialbd(b,\dOpt(b),m) = \XXXspecialbdStar(\bStar,\dOptStar(\bStar),m)$.
	More precisely, from now on let us denote 
	\[
	\bStarWvsWStar \coloneqq \bStarWvsWStar (b,m) =  \frac{t_0(b,m)}{\dOpt(b,m)}.
	\]
	We claim that this number satisfies the constraint
	\begin{equation}
	\label{eq:W-vs-W-star-claim-0}
	2^{-2/(2+m)} = \bMinStar < \bStarWvsWStar  < \bMaxStar = \Bigl(\frac{2+2m}{4+3m}\Bigr)^{2/(2+m)}
	\end{equation}
	and moreover the following  equalities hold (cf.\ Figure~\ref{fig:star-relations-repeated-in-appendix}):
	\begin{align}
	\label{eq:W-vs-W-star-claim-1}
	\tZeroStar(\bStarWvsWStar) 
	&= \frac{b}{\dOpt(b)},\\
	\label{eq:W-vs-W-star-claim-2}
	\dOptStar(\bStarWvsWStar) 
	&= \frac{1}{\dOpt(b)},
	\end{align}
	and, provided that $\bMin(m) < b < \bMax(m)$ and $\bMinTildeStar < \bStarWvsWStar < \bMaxStar$,
	\begin{equation}
	\label{eq:W-vs-W-star-claim-3}	    
	\XXXspecialbdStar(\bStarWvsWStar, \dOptStar(\bStarWvsWStar), m) 
	= 	\XXXspecialbd(b,\dOpt(b),m).
	\end{equation}
	We proceed with the proofs of these four claims. 
	
	\emph{Proof of Claim~\eqref{eq:W-vs-W-star-claim-0}}.
	This follows directly from Lemma~\ref{lem:tZero-divied-by-dOpt-monotone}. 
	
	\emph{Proof of Claim~\eqref{eq:W-vs-W-star-claim-1}}.
	In order to prove that $ \tZeroStar(\bStarWvsWStar) = b/\dOpt(b)$
	we have to check that
	\[
	-\frac{m}{2+m} + \frac{2(1+m)}{2+m} \bigl(2 \Bigl( \frac{t_0(b)}{\dOpt(b)}\Bigr)^{1+m/2} - 1 \bigr) \Bigl( \frac{b}{\dOpt(b)}\Bigr) ^{-1-m/2} \overset{?}{=} 0
	\]
	(cf.~\eqref{eq:def-t_0-explicit-formula-star})
	or, equivalently,
	\begin{equation}
	\label{eq:equivalent-version-of-some-claim}
	-mb^{1+m/2} + 4(1+m)  t_0(b)^{1+m/2} \overset{?}{=}  2(1+m) \dOpt(b)^{1+m/2}.
	\end{equation}
	Substituting here
	\begin{equation}
	\label{eq:4-1-m-t_0-1-m-2}
	4(1+m)  t_0(b)^{1+m/2} = 2t_0(b) \cdot 2(1+m)t_0(b)^{m/2} = 2t_0(b) \cdot \frac{2+m}{2b^{-m/2} - 1}
	\end{equation}
	(recall \eqref{eq:def-t_0-explicit-formula})
	we obtain exactly the definition of $\dOpt(b)$
	(cf.~\eqref{eq:fixed-m-supremum-W-optimal-curve}).	
	This finishes the proof of the  claim.	
	
	\emph{Proof of Claim~\eqref{eq:W-vs-W-star-claim-2}}.
	In order to prove that $\dOptStar(\bStarWvsWStar)  =1/\dOpt(b)$
	we have to check that
	\begin{align*}
	\MoveEqLeft[4]
	2(1+m)
	\bigl(2 \Bigl(\frac{t_0(b)}{\dOpt(b)}\Bigr)^{1+m/2} -1\bigr)  \Bigl(\frac{1}{\dOpt(b)}\Bigr)^{-m/2} \\
	&\overset{?}{=}
	-(2+m)\bigl(2\Bigl(\frac{t_0(b)}{\dOpt(b)}\Bigr)^{1+m/2} -1\bigr)
	\Bigl(\frac{t_0(b)}{\dOpt(b)}\Bigr)^{-m/2}
	+2m \frac{b}{\dOpt(b)}
	\end{align*}
	(cf.~\eqref{eq:fixed-m-supremum-W-optimal-curve-star}; 
	note that we already substituted the value of $\tZeroStar(\bStarWvsWStar) $ 
	using Claim~\eqref{eq:W-vs-W-star-claim-1};
	note also that although for $m=1$ and $\bStar\in (\bMinStar(1),\bShStar)$ we have
	$\dOptStar(\bStar, 1) < \dMinStar(\bStar, 1)$, the
	formula defining $\dOptStar(\bStar,m)$ 
	does not change,
	so our reasoning is not affected).
	Equivalently,
	\begin{align*}
	\MoveEqLeft[4]
	4(1+m) t_0(b)^{1+m/2} - 2(1+m)\dOpt(b)^{1+m/2} \\
	&\overset{?}{=}
	-2(2+m) t_0(b) +  (2+m)\frac{\dOpt(b)^{1+m/2}}{t_0(b)^{m/2}}
	+2mb.
	\end{align*}
	The left-hand side is equal to $m b^{1+m/2}$ by \eqref{eq:equivalent-version-of-some-claim}
	(i.e., by~\eqref{eq:4-1-m-t_0-1-m-2} and~\eqref{eq:fixed-m-supremum-W-optimal-curve}).
	On the right-hand side we use~\eqref{eq:def-t_0-explicit-formula} to get
	\[
	(2+m)\frac{\dOpt(b)^{1+m/2}}{t_0(b)^{m/2}} = \frac{(2+m)\dOpt(b)^{1+m/2}}{\frac{2+m}{2(1+m)(2b^{-m/2} - 1)}} 
	= 2(1+m)(2b^{-m/2} - 1)\dOpt(b)^{1+m/2}.
	\]
	Hence reduces once again to the definition of $\dOpt(b)$
	(cf.~\eqref{eq:fixed-m-supremum-W-optimal-curve}).	
	This finishes the proof of the claim.

	\begin{figure}[t]
		\begin{tikzpicture}[scale=1, xscale=3.5, yscale=0.5]
		\draw[->] (0, 0) -- (3, 0); 
		\draw[->] (0, -2) -- (3, -2); 
		\draw (2.692582403567252,0) -- (1,-2);
		\draw (2,0) -- (2/2.692582403567252,-2);
		\draw [dashed] (1.5,0) -- (1.5/2.692582403567252,-2);
		\draw [dashed] (1,0) -- (1/2.692582403567252,-2);
		\node[circle,minimum size=3pt,inner sep=0,fill=black,label=above:{$1\vphantom{\dOpt(b,m)t_0(b,m)}$}] at (1,0) {};
		\node[circle,minimum size=3pt,inner sep=0,fill=black,label=above:{$b\vphantom{\dOpt(b,m)t_0(b,m)}$}] at (1.5,0) {};
		\node[circle,minimum size=3pt,inner sep=0,fill=black,label=above:{$t_0(b,m)\vphantom{\dOpt(b,m)}$}] at (2,0) {};		
		\node[circle,minimum size=3pt,inner sep=0,fill=black,label=above:{$\dOpt(b,m)\vphantom{\dOpt(b,m)t_0(b,m)}$}] at (2.692582403567252,0) {};
		\node[circle,minimum size=3pt,inner sep=0,draw=black,fill=white,label=below:{}] at (1/2.692582403567252,-2)  {}; 
		\node[circle,minimum size=3pt,inner sep=0,draw=black,fill=white,label=below:{}] at (1.5/2.692582403567252,-2)  {}; 
		\node[circle,minimum size=3pt,inner sep=0,fill=black,label=below:{$\frac{t_0(b,m)}{\dOpt(b,m)}$}] at (2/2.692582403567252,-2) {};		
		\node[circle,minimum size=3pt,inner sep=0,fill=black,label=below:{$1\vphantom{\frac{t_0(b,m)}{\dOpt(b,m)}}$}] at (1,-2) {};
		
		\end{tikzpicture}
		\caption{Claims~\eqref{eq:W-vs-W-star-claim-0}--\eqref{eq:W-vs-W-star-claim-3} 
			assert that for an appropriate choice of $d$, $\bStar$, and $\dStar$
			we have $\XXXspecialbd(b,d,m) = \XXXspecialbdStar(\bStar,\dStar,m)$.
		}
		\label{fig:star-relations-repeated-in-appendix}
	\end{figure}

	\emph{Proof of Claim~\eqref{eq:W-vs-W-star-claim-3}}.
	Let us start by explicitly stressing that  Lemmas~\ref{lem:dOpt-definition} 
	and
	\ref{lem:dOpt-definition-star}
	imply that $(b, \dOpt(b,m))\in\specialbd{m}$ if $\bMin(m) < b < \bMax$,
	$(\bStarWvsWStar, \dOptStar(\bStarWvsWStar,m))\in\specialbd{m}$ if $\bMinTildeStar(m) < \bStarWvsWStar < \bMaxStar$.
	Using claims \eqref{eq:W-vs-W-star-claim-1}, \eqref{eq:W-vs-W-star-claim-2} we observe that
	\begin{align*}
	\MoveEqLeft[3]\XXXspecialbdStar\bigl(
	\bStarWvsWStar,
	\dOptStar(\bStarWvsWStar),
	m\bigr)
	=\XXXspecialbdStar\Bigl(
	\frac{t_0(b)}{\dOpt(b)},\frac{1}{\dOpt(b)}, m\Bigr)\\
	& = \Bigl(1-\frac{1}{\dOpt(b)}\Bigr)\cdot
	\Bigl[-\frac{2+m}{m}  + \frac{m}{2+m}\frac{1}{\dOpt(b)} + \frac{2(2+m)}{m}\frac{t_0(b)}{\dOpt(b)} \\
	& \qquad\qquad +
	\frac{4(1+m)}{m(2+m)}(2\bigl( \frac{t_0(b)}{\dOpt(b)}\bigr)^{1+m/2}  -1)\bigl(\frac{1}{\dOpt(b)}\bigr)^{-m/2} 
	- 2\frac{b}{\dOpt(b)}\Bigr]^{-1}\\
	& = \Bigl(\dOpt(b) -1\Bigr)\cdot
	\Bigl[-\frac{2+m}{m}\dOpt(b)  + \frac{m}{2+m} 
	+ \frac{2(2+m)}{m}t_0(b) \\
	&\qquad\qquad+
	\frac{4(1+m)}{m(2+m)}(2\bigl( \frac{t_0(b)}{\dOpt(b)}\bigr)^{1+m/2}  -1)\dOpt(b)^{1+m/2} 
	- 2b \Bigr]^{-1}.
	\end{align*}
	Comparing with the formula for $\XXXspecialbd(b,\dOpt(b),m)$,
	we see that in order to prove the claim~\eqref{eq:W-vs-W-star-claim-3}
	we need to check that
	\begin{multline*}
	\frac{2(2+m)}{m}t_0(b) +
	\frac{4(1+m)}{m(2+m)}\bigl(2\Bigl( \frac{t_0(b)}{\dOpt(b)}\Bigr)^{1+m/2}  -1\bigr)\dOpt(b)^{1+m/2} 
	- 2b\\
	\overset{?}{=}
	- \frac{2m}{2+m}b +
	\frac{4(1+m)}{m(2+m)}(2b^{-m/2}  -1)\dOpt(b)^{1+m/2} 
	+2t_0(b),
	\end{multline*}
	which, after multiplying both side by $m(2+m)/2$ and rearranging,
	is equivalent to
	\begin{multline*}
	2(2+m)t_0(b) +	4(1+m) t_0(b)^{1+m/2} - 2(1+m)\dOpt(b)^{1+m/2} -2mb\\
	\overset{?}{=}	 2(1+m)(2b^{-m/2}  -1)\dOpt(b)^{1+m/2} 
	\end{multline*}
	This equality is true by the calculations from the proof of Claim~\eqref{eq:W-vs-W-star-claim-2}.
	This finishes the proof of tha last claim.
	
	With the claims proved, we can finish the proof of the proposition.
	Suppose first that $m\geq 2$, so that $\bMaxTilde = \bMax$ and $\bMinTildeStar = \bMin$.
	By Lemma~\ref{lem:tZero-divied-by-dOpt-monotone},
	$\bStarWvsWStar = \bStarWvsWStar(b)$ can take any value in $(\bMinStar, \bMaxStar)$. 
	Thus Claim~\eqref{eq:W-vs-W-star-claim-3}
	in conjunction with Propositions~\ref{prop:lower-bound-fixed-m-special-b-c-d} 
	and~\ref{prop:lower-bound-fixed-m-special-b-c-d-star}
	proves the assertion of the proposition.
	
	Assume now that $m=1$, i.e., we are in the case when 
	$\bMinStar(1) < \bMinTildeStar(1) = \bShStar{}$
	and 
	$\bSh{} = \bMaxTilde(1)   < \bMax(1)$.
	We claim that $\bStarWvsWStar(\bSh) = \bShStar$.
	Indeed, the definitions of $\bStarWvsWStar$, $t_0$, and $\tZeroStar$,
	the equality $\bSh = 7^{2/3}$, and Claim~\eqref{eq:W-vs-W-star-claim-1} yield
	\begin{align*}
	\bStarWvsWStar(\bSh)
	&= \frac{t_0(\bSh)}{\dOpt(\bSh)} = 
	\frac{(3/4)^2 (2\cdot 7^{-1/3} - 1)^{-2}}{\dOpt(\bSh)},\\
	4^{2/3}(2\bStarWvsWStar(\bSh)^{3/2} -1 )^{2/3}
	=\tZeroStar(\bStarWvsWStar(\bSh))
	&=\frac{\bSh}{\dOpt(\bSh)} = \frac{7^{2/3}}{\dOpt(\bSh)}.
	\end{align*}
	Solving both equations for $1/\dOpt(\bSh)$ and comparing yields
	\[
	(4/3)^2 (2\cdot 7^{-1/3} - 1)^{2}\bStarWvsWStar(\bSh)
	= (4/7)^{2/3}(2\bStarWvsWStar(\bSh)^{3/2} -1 )^{2/3}
	\]
	and after some easy transformations  we arrive at
	\[
	\bStarWvsWStar(\bSh) = \Bigl(\frac{27}{54 - 16(2-7^{1/3})^3}\Bigr)^{2/3},
	\]
	i.e., exactly the definition of $\bShStar$ (cf.\ the definition before~\eqref{eq:def-bMinTildeStar}).
	In order to finish the proof of the proposition in the case $m=1$,
	we apply Claim~\eqref{eq:W-vs-W-star-claim-3} with Propositions~\ref{prop:lower-bound-fixed-m-special-b-c-d} 
	and~\ref{prop:lower-bound-fixed-m-special-b-c-d-star} 
	(as above for $m\geq 2$)
	and use additionally Lemma~\ref{lem:m=1-optimal-curve-far-end}. 
\end{proof}

\section{A lower bound valid for all \texorpdfstring{$m$}{m}}

\label{APP:crunching}

In order to have a reasonably good estimate of $\weakconstrestricted{\Lambda_m}{\fclassspec{m}}$ for all $m\in\NN$ some technical 
 and numerical work is needed.
In this section we present the proof of the following estimate.

\begin{theorem}
	\label{thm:main-lower-bounds-1.34}
	For every $m\in\NN$, $\weakconstrestricted{\Lambda_m}{\fclassspec{m}} \geq 1.34$.
\end{theorem}

Note that Remark~\ref{rem:W-vs-W-star} immediately implies the following.

\begin{corollary}
	For every $m\in\NN$, $\weakconstrestricted{\Lambda_m^*}{\fclassspecStar{m}} \geq 1.34$.
\end{corollary}

Due to Lemma~\ref{lem:lower-bound-fixed-m-special-b-c-d-alternate-expression-with-inequality},
we only need to find some values of $b=b(m)$ and $d=d(m)$
which would confirm that $\XXXspecialbd(b(m), d(m),m)$
is fairly large for all $m\in \NN$. 
To this end, we will take
\[
b =b(m)= e^{1/m}, \qquad d = d(m) =e^{3/m}.
\]
Denote
\begin{align*}
\theta &\coloneqq   2 e^{-1/2}  - 1 \approx 0.213,\\
K &\coloneqq 4\theta e^{3/2} \approx 3.819, \\
L &\coloneqq -4\ln(2\theta) \approx 3.412, 
\end{align*}
and 
\begin{align*}
P(m,M) &\coloneqq (2K+2L -10) m^4
+ (10K+6L+2M -45) m^3\\
&\qquad+ (23 K + 4 L +  6 M -87) m^2
+ (24K + 4M - 96) m
+ (9K - 36).
\end{align*}
We need the following technical lemma.

\begin{lemma}
	\label{lem:technical-m-large-Taylor}
	The following holds for $m\geq 4$.
	\begin{enumerate}[a)]
		\item
		\label{item:Taylor-constraints}
		We have $(e^{1/m}, e^{3/m}) \in \specialbd{m}$, i.e., 
		the numbers $b=e^{1/m}$ and $d=e^{3/m}$ satisfy the constraints of Definition~\ref{def:f-class-spec}.
		\item
		\label{item:Taylor-uZero}
		Denote $u_0 = u_0(m) = m\ln(t_0(e^{1/m},m))$.
		Then $u_0(m)\in [1,3]$.
		Moreover, the quantities  $u_0(m)$ and $ u_0(m)/m$  decrease with $m$.
		\item
		\label{item:Taylor-coefficients_of_P}
		If $M\geq 0$ and $m\geq 1$, then $P(m, M) > 0 $.
		\item 
		\label{item:Taylor-estimate}
		If $\delta>0$, $M\in(0,\infty)$, and $m\in\NN$ are such that
		$
		u_0(m)/m \leq \delta$
		and 
		$u_0(m)^2 (1+ e^\delta \delta/3) \leq M$,
		then
		\[
		\XXXspecialbd(e^{1/m},e^{3/m},m) \geq \frac{6m(m+1)(m+3/2)(m+2)}{P(m,M)}.
		\]
	\end{enumerate}
\end{lemma}

\begin{proof}
	\ref{item:Taylor-constraints}  
	If $b=b(m) = e^{1/m}$ and $d=d(m) = e^{3/m}$, then clearly 
	$1<b <d$, $b^{m/2} = e^{1/2}$, and $d^{m/2}  = e^{3/2}$.
	The constraints of Definition~\ref{def:f-class-spec} are satisfied,
	since they read 
	\begin{equation*}
	-(2+m) + 2(1+m) ( 2 e^{-1/2}  - 1 ) e^{1/2}  <  0  < -(2+m) + 2(1+m) ( 2 e^{-1/2}  - 1 ) e^{3/2} < 2m,
	\end{equation*}
	which follows from 
	\[ 2 - e^{1/2} < 0.5 \leq \frac{2+m}{2(1+m)} \leq 0.75 < 
	2 e  -  e^{3/2}  < 1.25 \leq \frac{2+3m}{2(1+m)}.
	\]
	
	\ref{item:Taylor-uZero} 
	The first assertion follows since 
	$ e^{u_0(m) /m} =  t_0(e^{1/m},m) \in [e^{1/m},e^{3/m}]$.
	The definitions of $t_0 = t_0(b,m)$ and $\dOne = \dOne(b,m)$ (see equation~\eqref{eq:def-t_0-explicit-formula} and Definition~\ref{def:f-class-spec}) imply that
	\begin{equation}
	\label{eq:def-u_0-formula}
	e^{u_0 /2} = t_0(e^{1/m},m)^{m/2} = \frac{2+m}{m\dOne(e^{1/m},m)} = \frac{2+m}{2(1+m)\theta}
	= \frac{1}{2\theta}\Bigl( 1 + \frac{1}{1+m} \Bigr).
	\end{equation}
	Hence $u_0 = u_0(m) = m\ln(t_0(e^{1/m},m))$ decreases with $m$.
	Since this quantity is positive, the first part of the assertion follows as well.
	
	\ref{item:Taylor-coefficients_of_P}
	One can check that for $M\geq 0$ the coefficients in front of $m^4$, $m^3$, $m^2$ in the definition of $P(m,M)$ satisfy 
	\begin{align*}
	2K & +2L -10 > 0,\\
	10K & +6L+2M -45 \geq 10K + 6L -45 >0,\\
	23 K & + 4 L +  6 M -87 \geq 23 K + 4 L  - 87 > 0.
	\end{align*}
	Moreover, for $M\geq 0$ , 
	\begin{equation*}
	59 K + 12 L +  12 M - 238\geq 59 K + 12 L- 238 \geq 28 > |9K - 36| >0.
	\end{equation*}
	Thus, for $M\geq 0$ and $m\geq 1$,
	\begin{align*}
	P(m,M) &\geq
	(2K+2L -10) m
	+ (10K+6L+2M -45) m\\
	&\qquad+ (23 K + 4 L +  6 M -87) m
	+ (24K + 4M - 96) m
	+ (9K - 36)\\
	&= (59 K + 12 L +  12 M - 238)m + (9K - 36) > 0.
	\end{align*}
	
	\ref{item:Taylor-estimate}  
	Our goal is to estimate $\XXXspecialbd(e^{1/m}, e^{3m}, m)$.
	For our values of $b$, $d$ 
	we have $2b^{-m/2} - 1 = 2e^{1-/2} - 1  =\theta$,
	$d^{1+m/2} = e^{3/m + 3/2}$,
	and thus the denominator of $\XXXspecialbd(e^{1/m}, e^{3m}, m)$ (see~\eqref{eq:a=1-b=c-would-like-to-optimize-technical})
	reads
	\begin{align}
	\label{eq:denominator-W-b-e^1/m-d-e^3/m}
	&\frac{m}{2+m} -\frac{2+m}{m} e^{3/m} - \frac{2m}{2+m} e^{1/m} 
	+ 	\frac{4(1+m)}{m(2+m)}\theta e^{3/m + 3/2} + 2t_0(e^{1/m},m)
	\\
	&\quad = \frac{m}{2+m} -\Bigl(\frac{2+m}{m} - K \frac{1+m}{m(2+m)}\Bigr)e^{3/m} - \frac{2m}{2+m} e^{1/m} 
	+ 2t_0(e^{1/m},m).\nonumber
	\end{align}
	The biggest challenge is reasonably estimating the last summand in~\eqref{eq:denominator-W-b-e^1/m-d-e^3/m}, i.e.,  $2 t_0(e^{1/m},m) = 2 e^{u_0(m)/m}$.
	Equation~\eqref{eq:def-u_0-formula} implies that
	\begin{equation*}
	u_0 
	= 2 \ln\bigl(1+\frac{1}{1+m} \bigr) - 2\ln(2\theta)
	\leq \frac{2}{1+m} -2 \ln(2\theta).
	\end{equation*}
	Fix $\delta >0$.  
	If $x\in(0,\delta)$, then -- by Taylor's formula with the Lagrange form of the remainder -- for some $\xi\in (0,\delta)$,
	\[
	e^x = 1 + x + \frac{1}{2}x^2 + \frac{1}{6} e^{\xi} x^3 \leq 1 + x + \frac{1}{2}x^2 (1+ e^\delta \delta/3).
	\]
	Thus, if $u_0/m \leq \delta$,
	then
	\begin{align*}
	2t_0  =  2e^{u_0/m} &\leq 2 +  \frac{2u_0}{m} + \frac{u_0^2}{m^2}(1+ e^\delta \delta/3)\\
	&\leq 2 + \frac{4}{m(1+m)} - \frac{4\ln(2\theta)}{m} + \frac{u_0^2}{m^2}(1+ e^\delta \delta/3)\\
	&\leq  2 + \frac{4}{m(1+m)} + \frac{L}{m} + \frac{M}{m^2}, 
	\end{align*}
	where we recall that $L  = -4\ln(2\theta)$ and $M$ is by assumption a number such that
	$u_0^2 \cdot (1+ e^\delta \delta/3) \leq M$. 
	
	In total, the expression~\eqref{eq:denominator-W-b-e^1/m-d-e^3/m} (i.e., the denominator of $\XXXspecialbd(e^{1/m}, e^{3m}, m)$) can be estimated as follows.
	The coefficient in front of  $e^{3/m}$ is negative, since for $m\geq 4$,
	\[
	\frac{2+m}{m} \geq 1 \geq \frac{K}4 \geq \frac{K}m \geq K \frac{1+m}{m(2+m)}.
	\]
	Thus, using the Taylor expansions of $e^{3/m}$ and $e^{1/m}$,
	as well as the bound for $2t_0(e^{1/m},m)$ obtained above, we get
	
	\begin{align*}
	\MoveEqLeft[6]
	\frac{m}{2+m} -\Bigl(\frac{2+m}{m} - K \frac{1+m}{m(2+m)}\Bigr)e^{3/m}
	- \frac{2m}{2+m} e^{1/m} 
	+ 2t_0(e^{1/m},m)\\
	&\leq 
	\frac{m}{2+m} 
	-\Bigl(\frac{2+m}{m} - K \frac{1+m}{m(2+m)}\Bigr)\Bigl( 1 + \frac{3}{m} + \frac{3^2}{2 m^2}\Bigr)\\
	&\quad - \frac{2m}{2+m} \Bigl( 1 + \frac{1}{m} + \frac{1}{2 m^2}\Bigr) + 2 + \frac{4}{m(1+m)} + \frac{L}{m} + \frac{M}{m^2}.
	\end{align*}
	After multiplying by $2m^3(1+m)(2+m)$, expanding, and simplifying we arrive at
	$P(m,M)$
	(recall that $K$, $L$ are some numerical constants).
	
	In the case of the numerator of~$\XXXspecialbd(e^{1/m}, e^{3m}, m)$ (see~\eqref{eq:a=1-b=c-would-like-to-optimize-technical}), 
	after multiplying by  $2m^3(1+m)(2+m)$ we have
	\begin{align*}
	2m^3(1+m)(2+m)\cdot \bigl(e^{3/m} -1\bigr)
	&\geq 2m^3(1+m)(2+m) \Bigl( \frac{3}{m}  + \frac{3^2}{2m^2}\Bigr) \\
	&= 6m(m+1)(m+3/2)(m+2).
	\end{align*}
	Combining these two inequalities finishes the proof.
\end{proof}

\begin{proof}[Proof of Theorem~\ref{thm:main-lower-bounds-1.34}]
	For $m\in\{1,2,3,4\}$ we numerically maximized the function $\XXXspecialbd(b,d,m)$ above (see Table~\ref{table:small-m}).
	In order to  confirm that $\XXXspecialbd(b(m), d(m),m)$
	is fairly large for \emph{all} $m\in \NN$, $m\geq 5$, 
	we take $b =b(m)= e^{1/m}$ and $d = d(m) =e^{3/m}$.
	By Lemma~\ref{lem:technical-m-large-Taylor}~\ref{item:Taylor-constraints}
	these numbers satisfy the constraints of Definition~\ref{def:f-class-spec}. 
	We claim that $\XXXspecialbd(e^{1/m},e^{3/m},m) \geq 1.34$ for $m\geq 5$.

	Indeed, for $m\in\{5,6,\dots, 27\}$ 
	one can numerically check that we have
	\[
	\XXXspecialbd(e^{1/m},e^{3/m},m) \geq \XXXspecialbd(e^{1/27},e^{3/27},27) \geq 1.35 \geq 1.34.
	\]
	Numerical experiments suggest that the function $m\mapsto \XXXspecialbd(e^{1/m},e^{3/m},m)$ 
	is in fact decreasing in $m$ (and bounded below by $1.34$), 
	but we have no proof of this monotonicity, so we shall use a slightly different approach, based on the estimate from Lemma~\ref{lem:technical-m-large-Taylor}~\ref{item:Taylor-estimate}.
	
	Recall the definition of $u_0 = u_0(m)$ from Lemma~\ref{lem:technical-m-large-Taylor}
	and 
	observe that $u_0(25) \leq 1.79$ and $u_0(25)/25 \leq  0.072$
	(by the explicit formula~\eqref{eq:def-u_0-formula}).
	Take $\delta \coloneqq 0.072$ and $M \coloneqq 3.3$.
	By Lemma~\ref{lem:technical-m-large-Taylor}~\ref{item:Taylor-uZero}
	we have  $u_0(m) \leq 1.79$ and $u_0(m)/m \leq  \delta$ for $m\geq25$.
	It follows that for $m\geq 25$,
	\[
	u_0(m)^2 (1+e^{\delta} \delta/3) \leq 1.79^2(1+e^{0.072} 0.024) \leq 3.29 \leq M,
	\]
	i.e., the assumptions of Lemma~\ref{lem:technical-m-large-Taylor} ~\ref{item:Taylor-estimate}
	are satisfied for $m\geq 25$ (with $\delta$ and $M$ as above).
	Thus,
	\[
	\XXXspecialbd(e^{1/m},e^{3/m},m) \geq \frac{6m(m+1)(m+3/2)(m+2)}{P(m,3.3)}  
	\]
	for $m\geq 25$.
	To finish the proof it suffices to check that
	\[
	6m(m+1)(m+3/2)(m+2) - 1.34 P(m,3.3) \overset{?}{>} 0
	\]
	for $m\geq 25$.
	Denote
	\[
	Q(x) \coloneqq 6x(x+1)(x+3/2)(x+2) - 1.34 P(x,3.3)
	\]
	for $x\in\RR$. This is a fourth-degree  polynomial with explicit, numerical coefficients
	and the leading coefficient is positive since 
	\[
	6\geq 5.99 \geq 1.34(2K + 2L - 10).
	\]
	To finish the proof, one can for example check that $Q$ has four zeros in $(-20,25)$ or check that all the coefficients of the polynomial $y\mapsto Q(25+y)$ are positive; we omit the rest of the details.
\end{proof}

\section{Asymptotic considerations}

\label{APP:interlude}

%
%
%
%
%
%
%
%

Because of the difficulties with 
tackling the supremum from Proposition~\ref{prop:lower-bound-fixed-m}
directly,
one could start with considering the case $m\to\infty$,
since passing to the limit simplifies several aspects of the task.

\begin{proposition}
	\label{prop:lower-bound-asymptotic}
	We have
	\begin{equation}
	\label{eq:lower-bound-asymptotic}
	\liminf_{m\to \infty} \weakconstrestricted{\Lambda_m}{\fclass{m}} 
	\geq 
	\sup\Bigl\{ 
	\frac{\xGain(x,z) + \yGain(y,z)}{ 2e^{x} -x - 2  +  \int_{0}^{y} \bigl| -1 +  z  e^{s}\bigr| ds}
	\Bigr\},
	\end{equation}
	where the supremum is taken over all
	$ x>0$, $y >0$ , $2(2-e^x) \leq z < 2 $, and
	\begin{align*}
	\xGain = \xGain(x,z) &\coloneqq x + \bigl(0 \lor \ln(2e^x - 2)\bigr)\land \bigl(\ln(2e^x-2) - \ln(2-z)\bigr),\\
	\yGain = \yGain(y,z) &\coloneqq y +  0\lor \ln(|-2 + z e^{y}|).
	\end{align*}
\end{proposition}

\begin{proof}
	Throughout this proof $m\in\NN$ is a parameter tending to infinity 
	and some of the quantities below  depend on $m$,
	but in most places we suppress this dependence in the notation.
	Fix some $x, y >0$ and $z\in\RR$ such that $2(2-e^x) \leq z < 2$.
	Let $h\geq 0$ be defined by the relation $z = 2(1 + e^{-h}(1-e^x))$, i.e.,
	$ h = \ln(2 e^x - 2) - \ln(2-z)$.
	Set
	\begin{equation*}
	b = b(m) = e^{2x/m},\quad
	c = c(m)  =e^{2(x+h)/m},\quad
	d = d(m) = e^{2(x+h+y)/m}.
	\end{equation*}
	Consider a function of the form~\eqref{eq:weak-type-extremals-2}
	with $a=1$ and $1 < b\leq c < d$ as above
	(cf.~the proof of Proposition~\ref{prop:lower-bound-fixed-m}).
	We have 
	$ -\frac{2(1+m)}{m} \longrightarrow -2$ 
	and
	\begin{align*}
	\dOne &= \frac{2(1+m)}{m c^{m/2}} +  \frac{2(1+m)}{m c^{m/2}} \Bigl(\frac{b}{c}\Bigr)^{1+m/2}+ \dZero\Bigl(\frac{b}{c}\Bigr)^{1+m}\\
	& \xrightarrow[m\to\infty]{}
	2 e^{-x-h} + 2 e^{-x-2h} - 2 e^{-2h}
	=  z e^{-x-h}. 
	\end{align*}
	Let $\bGain = \bGain(m)$, $\dGain=\dGain(m)$ be the numbers defined in Proposition~\ref{prop:lower-bound-fixed-m} (for our choice of $b, c, d$),
	so that, by the proof of Proposition~\ref{prop:lower-bound-fixed-m}, 
	\begin{equation*}
	\bigl|\{ t\geq 0 : |\Lambda_m f (t)| \geq 1 \} \bigr| 
	= \bGain - 1 + \dGain - c. 
	\end{equation*}
	Then
	\begin{align*}
	\bGain 
	&=\Bigl[ e^{2x/m}\lor \Bigl[ e^{2x/m}\Bigl(-1 -\frac{2+m}{m} + \frac{2(1+m)}{m} e^x\Bigr)^{2/(2+m)}\Bigr]\Bigr] \land e^{2(x+h)/m} \\
	&=\exp\Bigl(\frac{2x}{m} + \Bigl[0\lor \Bigl[\frac{2}{2+m}\ln\bigl(-1 -\frac{2+m}{m} + \frac{2(1+m)}{m} e^x\bigr)\Bigr]\Bigr]\land \frac{2h}{m} \Bigr) 
	\end{align*}
	so
	\[
	m(\bGain - 1)
	\xrightarrow[m\to\infty]{}
	2x + \bigl[0\lor 2\ln\bigl(2e^x - 2\bigr)\bigr]\land 2h = 2\xGain(x,z).
	\]
	Similarly,
	\begin{align*}
	\dGain - c
	&= e^{2(x+h+y)/m}\lor \Bigl[ e^{2(x+h+y)/m}\Bigl| - 1 - \frac{2+m}{m} + \dOne e^{x+h+y} \Bigr|^{2/(2+m)}\Bigr]
	- e^{2(x+h)/m}\\
	&=\exp\Bigl(\frac{2(x+h+y)}{m} + 0\lor \Bigl[ \frac{2}{2+m}\ln\bigl|-1 -\frac{2+m}{m} + \dOne e^{x+h+y}\bigr|\Bigr] \Bigr) - e^{2(x+h)/m}
	\end{align*}
	so 
	\[
	m(\dGain - c) 
	\xrightarrow[m\to\infty]{} 
	2y + 0\lor 2\ln\bigl|-2 +  z e^{y}\big| = 2\yGain(y,z).
	\]

	We have 
	\begin{align*}
	m \int_0^\infty |f(t)| dt 
	&=m\int_{1}^{\exp(2x/m)} \Bigl| \frac{2+m}{m} + \dZero t^{m/2}\Bigr| dt\\
	&\quad  
	+ m\int_{\exp(2(x+h)/m)}^{\exp(2(x+h+y)/m)} \Bigl| -\frac{2+m}{m} + \dOne t^{m/2}\Bigr| dt\nonumber\\
	&= 2 \int_{0}^{x} \Bigl| \frac{2+m}{m} + \dZero e^{s} \Bigr| e^{2s/m} ds 
	+  2\int_{x+h}^{x+y+h} \Bigl| -\frac{2+m}{m} + \dOne e^{s} \Bigr| e^{2s/m} ds,
	\end{align*}
	where in the last step we substituted $t = e^{2s/m}$.
	For fixed $x, h, y$ we pass to the limit with $m\to\infty$.
	The integrals tend to
	\begin{align*}
	&2\int_{0}^{x} \Bigl| 1 - 2 e^{s}\Bigr| ds    
	=  2\int_{0}^{x} 2 e^{s} -1 ds  
	= 2(2e^{x} -x - 2),\\
	&2 \int_{x+ h }^{x+h+y} \bigl| -1 +  z e^{-x -h}  e^{s}\bigr| ds
	= 2\int_{0 }^{y} \bigl| -1 +  z e^{s}\bigr| ds
	\end{align*}
	respectively.
	It follows that for all $x,y>0$, $h\geq 0$ we have
	\begin{equation*}
	\liminf_{m\to \infty} \weakconstrestricted{\Lambda_m}{\fclass{m}}
	\geq 
	\frac{\xGain(x,z) + \yGain(y,z)}{ 2e^{x} -x - 2  +   \int_{0 }^{y} \bigl| -1 +  z e^{s}\bigr| ds }.
	\end{equation*}
	The assertion is obtained by taking the supremum.
\end{proof}

\begin{remark}
	\label{rem:asymptotic-star-for-the-record}
	For the record, we
	note that one can 
	also prove 
	that
	\begin{equation*}
	\liminf_{m\to \infty} \weakconstrestricted{\Lambda_m^*}{\fclassStar{m}} 
	\\
	\geq 
	\sup\Bigl\{ 
	\frac{\xGain(x,z) + \yGain(y,z)}{ 2e^{x} -x - 2  +  \int_{0}^{y} \bigl| -1 +  z  e^{s}\bigr| ds}
	\Bigr\},
	\end{equation*}
	where the supremum is taken over all
	$ x>0$, $y >0$ , $2(2-e^x) \leq z < 2 $, and
	$\xGain(x,z)$, $\yGain(x,z)$ are as in Proposition~\ref{prop:lower-bound-asymptotic}
	(note that the lower bounds exactly coincide).
	We omit the proof, since we do not use this result in the sequel.
	(And anyway the proof follows the same lines: we fix some $\xStar, \yStar >0$ and $\zStar\in\RR$ such that $2(2-e^{\xStar}) \leq \zStar < 2$,
	define $\hStar\geq 0$ by the relation $\zStar = 2(1 + e^{-\hStar}(1-e^{\xStar}))$,
	and	set
	\begin{align*}
	\dStar &= \dStar(m) = e^{-2(\xStar+\hStar+\yStar)/(2+m)},\\
	\cStar &= \cStar(m)  =e^{-2(\xStar  +\hStar)/(2+m)},\\
	\bStar &= \bStar(m) = e^{-2\xStar/(2+m)}.
	\end{align*}
	The resulting calculations are very similar to the ones in the proof of Proposition~\ref{prop:lower-bound-asymptotic}).
\end{remark}

The next result asserts that
in order to find the value of the supremum on the right-hand side of~\eqref{eq:lower-bound-asymptotic}
one can restrict to a special choice of the variables $x$, $y$, $z$:
\begin{itemize}
	\item $z = 2(2-e^x)$ (in particular, $\xGain(x,z) = x$ in Proposition~\ref{prop:lower-bound-asymptotic}),
	\item $2(2-e^x) \in (0,1]$ (in particular, $s\mapsto -1 + z e^s$ increases and $-1+z e^0 \leq 0$),
	\item $1\leq z e^y \leq 3$ (in particular, $\yGain(y,z) = y$  in Proposition~\ref{prop:lower-bound-asymptotic} and $-1+ze^y \geq 0$).
\end{itemize}
This serves as an motivation 
for the choices we made in Section~\ref{sec:guesstimates}.
It also explains  that the asymptotic estimate from Theorem~\ref{thm:main-lower-bounds} is nearly optimal:
up to numerical details no essentially better lower bound can be found using our approach.

\begin{proposition}
	\label{prop:asymptotic-push}
	The supremum on the right-hand side of~\eqref{eq:lower-bound-asymptotic} 
	is equal to
	\begin{equation*}
	\sup\Bigl\{ 
	\frac{x + y}{2e^x - x -4  - y + 2(2-e^x) (e^y + 1) -2\ln(2(2-e^x))}
	\Bigr\},
	\end{equation*}
	where the supremum is taken over
	$x\in [\ln(3/2), \ln(2))$
	(i.e.,
	$z = 2(2-e^x) \in(0,1]$)
	and $e^{-y} \leq 2(2-e^x) \leq  3e^{-y}$.
	Moreover,
	this supremum is attained on the curve described by the equation
	$e^x = 2(2-e^x) e^y$, $x\in [\ln(3/2), \ln(2))$.
\end{proposition}

Although elementary,
the proof is quite lengthy and in itself not that illuminating,
but for what it's worth, 
we present it in the next section.

\section{Detailed proof of Proposition~\ref{prop:asymptotic-push}}

\label{APP:push}

\subsection{Notation}

Throughout this section we work under the standing assumption
that
\begin{equation}
\label{eq:asymptotic-standing-assumption}
x>0, \quad y>0, \quad 2(2-e^x)\leq z \leq 2.
\end{equation}
In particular if $x\leq \ln(3/2)$, then $z\geq 1$;
if $z\leq 1$, then $x\geq \ln(3/2)$.
Recall the definitions of
$\xGain(x,z)$, $\yGain(y,z)$ from Proposition~\ref{prop:lower-bound-asymptotic}
and observe that 
\begin{align*}
\xGain(x,z) &= x + (0 \lor \ln(2e^x - 2))\land (\ln(2e^x-2) - \ln(2-z))\\
&= \begin{cases}
x                          & \text{ if }  x\leq \ln(3/2),\\
x + \ln(2e^x-2) - \ln(2-z) & \text{ if } x\geq \ln(3/2), z \leq 1,\\
x + \ln(2e^x-2)            & \text{ if } x\geq \ln(3/2), z \geq 1.
\end{cases}
\end{align*}
In particular, if  $x\geq \ln(3/2)$ and $z = 2(2-e^x)$, then $\xGain(x,z) = x$.
Furthermore,
\begin{equation*}
\yGain(y,z) = y +  0\lor \ln(|-2 + z e^{y}|) = \begin{cases}
y + \ln(2- ze^{y})      & \text{ if }  z \leq e^{-y},\\
y                       & \text{ if } e^{-y} \leq z \leq 3e^{-y},\\
y + \ln(ze^y - 2)       & \text{ if } z \geq 3e^{-y}.
\end{cases}
\end{equation*}
Denote from here on in
\begin{align*}
\XXXasymptoticNominator(x,y,z) &\coloneqq \xGain(x,z) + \yGain(y,z),\\
\XXXasymptoticDenominator(x,y,z) &\coloneqq 2e^x - x -2 + \int_{0}^{y} | -1 +  z  e^{s}| ds.
\end{align*}
Our goal is to maximize $\XXXasymptoticNominator(x,y,z)/\XXXasymptoticDenominator(x,y,z)$ over all $x$, $y$, $z$ satisfying~\eqref{eq:asymptotic-standing-assumption}.

\subsection{Auxiliary lemmas}

Below we gather some elementary technical lemmas
which will be used 
in the proof of Lemma~\ref{lem:asymptotic-push-case-0} where the special case $z = 2(2-e^x)$ is treated.
Note that all the denominators which appear in the statements of  below are positive
since they correspond to values of  $\XXXasymptoticDenominator(x,y,z)$
(for some $x$, $y$, $z$) and $\XXXasymptoticDenominator$ is positive.

\begin{lemma}
	\label{lem:asymptotic-push-case-0.2.a}
	We have
	\begin{align*}
	\sup\Bigl\{ \frac{y + \ln(2e^y - 2)}{ - y +2(e^y - 1)} : 
	\ln(3/2)\leq y\leq \ln(2) \Bigr\} 
	&\leq 1.1, \\ 
	\sup \Bigl\{ \frac{2 x + \ln(2) + \ln(4 (2 - e^x) e^x - 2)}{2 e^x - 2 x - 2 - \ln(2) + 2 (2 - e^x) (2 e^x - 1)} : 0 < x\leq \ln(3/2) \Bigr\}
	&\leq 1.1. 
	\end{align*}
\end{lemma}

\begin{proof}
	Set $\alpha \coloneqq 1/10$ and let
\[
g(y) \coloneqq (1+\alpha) (- y +2(e^y - 1)) - (y + \ln(2e^y - 2)).
\]
Our goal is to show that $g(y) \geq 0$ for $y\in [\ln(3/2),\ln(2)]$. 
Since
\begin{equation*}
g'(y) = (1 + \alpha)(-1+2e^y) -1 - \frac{e^y}{e^y-1}
=
\frac{(2e^y - 1)\bigl( (1+\alpha) e^y - (2+\alpha)\bigr)}{e^y-1},
\end{equation*}
$g$ attains its minimum on $[\ln(3/2),\ln(2)]$
at $y_0$ such that $e^{y_0} = (2+\alpha)/(1+\alpha) = 21/11 \in[3/2, 2]$.
We have $2e^{y_0}-2 = 2/(1+\alpha)$ 
and after minor computations we arrive at
\begin{align*}
g(y_0) = g\Bigl(\ln \Bigl(\frac{2+\alpha}{1+\alpha}\Bigr) \Bigr)
&= -(2+\alpha)\ln\Bigl(\frac{2+\alpha}{1+\alpha} \Bigr) + 2 - \ln\Bigl(\frac{2}{1+\alpha} \Bigr)\\
& = -\frac{21}{10} \ln\Bigl(\frac{21}{11}\Bigr) + 2 - \ln\Bigl(\frac{20}{11} \Bigr) > 0.\end{align*}
This finishes the proof of the first part.

The proof of the second part is similar.
Denote
\begin{align*}
h(x) &\coloneqq (1+\alpha)\bigl(2 e^x - 2 x - 2 - \ln(2) + 2 (2 - e^x) (2 e^x - 1)\bigr) \\
&\phantom{\coloneqq {}} - \bigl(2 x + \ln(2) + \ln(4 (2 - e^x) e^x - 2)\bigl).
\end{align*}
Our goal is to show that $h(x) \geq 0$ for $x\in (0, \ln(3/2)]$. 
Differentiating and simplifying yields
\begin{align*}
h'(x)  = \frac{2 (1 - 6 e^x + 4 e^{2 x})\cdot \bigl (2(1+\alpha) e^{2x} - 4 (1+\alpha) e^x + (2+\alpha) \bigr)}{2 e^x ( 2- e^x) -1} 
\end{align*}
The denominator is positive for $x\in (0, \ln(3/2)]$.
The nominator is equal to zero 
if 
\[
e^x = 1 + \sqrt{\frac{\alpha}{2(1+\alpha)}} = 1 +  1/ \sqrt{22}\approx 1.213
\]
or  $e^x = (3+\sqrt{5})/4 \approx 1.309$;
the function $h$ increases on $[0, \ln(1+1/\sqrt{22})]$,
 then decreases, 
 and increases again on $[\ln (3+\sqrt{5})- \ln(4), \ln(3/2)]$.
 It suffices to verify that
  $h(0) > 0$ and  $h(\ln((3+\sqrt{5})/4)) >0$ to finish the proof.
\end{proof}

\begin{lemma}
	\label{lem:asymptotic-push-case-0.2.b}
	We have
	\begin{align*}
	\sup\Bigl\{ \frac{x}{4 e^x - e^{2x} - x - 3} : 
	0 < x\leq \ln(3/2) \Bigr\} &= \frac{\ln(3/2)}{3/4 - \ln(3/2)} \leq 1.18.
	\end{align*}
\end{lemma}

\begin{proof}
We claim that the function under the supremum is increasing for $x\in(0,\ln(3/2)]$ 
and the supremum is attained at $x= \ln(3/2)$.
To this end it suffices to check that the function
$x\mapsto (4e^x - e^{2x} - 3)/x$ is decreasing on $(0,\ln(3/2)]$.
This is indeed true,
since the function $x\mapsto 4e^x - e^{2x} - 3$ is concave on $[0,\infty)$
(as direct differentiation shows) with value $0$ for $x=0$.
This ends the proof.
\end{proof}

\begin{lemma}
	\label{lem:asymptotic-push-case-0.3.a}
	We have
	\begin{equation*}
	\sup \Bigl\{ \frac{2 x - \ln(2-e^x)  +  \ln(2 e^x - 2)}{2 e^x - 2 x - 2\ln(2) - \ln(2-e^x)} : \ln(3/2) \leq x < \ln(2) \Bigr\} \leq 1.3.
	\end{equation*}
\end{lemma}

\begin{proof}
Set $\alpha = 3/10$ and denote
\begin{equation*}
g(x) \coloneqq (1+\alpha)\bigl(2 e^x - 2 x - 2\ln(2) - \ln(2-e^x)\bigr)  - \bigl(2 x - \ln(2-e^x)  +  \ln(2 e^x - 2)\bigl).
\end{equation*}
Our goal is to show that $g(x) \geq 0$ for $x\in [\ln(3/2), \ln(2))$. 
Differentiating and simplifying yields
\begin{align*}
g'(x)  = \frac{ (-4 +7 e^x - 2 e^{2 x})\cdot \bigl ((1+\alpha) e^{x} - (2+\alpha) \bigr)}{(2 - e^x) (-1 + e^x)} 
\end{align*}
The denominator and the factor $-4 +7 e^x - 2 e^{2 x}$ are positive for $x\in [\ln(3/2), \ln(2))$.
The second factor in the nominator 
changes sign for $x$ such that $e^x = (2+\alpha)/(1+\alpha) = 23/13 \in [3/2,2)$
(and $g$ attains its minimum on $[\ln(3/2), \ln(2))$ at that point).
Substituting and simplifying yields
\[
g\Bigl( \frac{23}{13} \Bigr) =
 \frac{46}{10}
  - \frac{46}{10} \ln\Bigl( \frac{23}{13} \Bigr) 
  - \frac{26}{10} \ln(2)
    - \frac{3}{10} \ln\Bigl( \frac{3}{13} \Bigr)
      - \ln\Bigl( \frac{20}{13} \Bigr) > 0,
      \]
which finishes the proof.
\end{proof}

\begin{lemma}
	\label{lem:asymptotic-push-case-2.1.1-no-y-x=ln3/2}
	We have
	\begin{equation*}
	\sup \Bigl\{ \frac{2 \ln(3/2) + \ln(2/z)}{4 -  2\ln(3/2) - \ln(2/z) -z} : 1\leq z\leq 2 \Bigr\} \leq 1.1.
	\end{equation*}
\end{lemma}

\begin{proof}
	Set $\alpha = 1/10$ and denote
	\begin{equation*}
	g(z) \coloneqq (1+\alpha)\bigl(4 -  2\ln(3/2) - \ln(2/z) -z\bigr)  - \bigl(2 \ln(3/2) + \ln(2/z)\bigl).
	\end{equation*}
	Our goal is to show that $g(z) \geq 0$ for $z\in [1,2]$. 
	Differentiating  yields
	\begin{align*}
	g'(z)  = (2+\alpha)/z - (1+\alpha),
	\end{align*}
	so $g$ is concave on $[1,2]$. Moreover, 
	\[
	g(1) = \frac{33}{10} + \frac{21\ln(2)}{10} - \frac{21\ln(3)}{5} > 0, \qquad
	g(2) = \frac{11}{5} - \frac{21}{5} \ln \Bigl(\frac{3}{2}\Bigr)> 0,
	\]
	so the proof is finished.
\end{proof}

\subsection{Special case: \texorpdfstring{$z = 2(2-e^x)$}{z=2(2-exp(x))}}

We first handle the special case when $z=2(2-e^x)$.

\begin{lemma}
	\label{lem:asymptotic-push-case-0}
	The value of the supremum
	\[
	\sup \Bigl\{ \frac{\XXXasymptoticNominator(x,y,2(2-e^x))}{\XXXasymptoticDenominator(x,y,2(2-e^x))} : x>0, y>0\Bigr\}
	\]
    does not change if we restrict the range of $(x,y)$ to $x\in[\ln(3/2), \ln(2))$ and $y>0$ such that $e^{-y} \leq 2(2-2e^x) \leq 3e^{-y}$.
    Moreover, the supremum is attained on the curve
     $e^x  +2 e^x e^{y} - 4e^{y} = 0$, where $x\in [\ln(3/2), \ln(2))$.
\end{lemma}

	Observe that $\xGain(x,2(2-e^x)) = x$ for every $x>0$.
	The proof of Lemma~\ref{lem:asymptotic-push-case-0} is divided into five cases:
	\begin{itemize}
		\item Case 1: $2(2-e^x) \leq e^{-y}$,
		\item Case 2.1: $2(2-e^x) \geq 1$ and $2(2-e^x) \geq 3e^{-y}$,
		\item Case 2.2: $2(2-e^x) \geq 1$ and $2(2-e^x) \leq 3e^{-y}$,
		\item Case 3.1: $e^{-y} \leq 2(2-e^x) \leq 1$ and $2(2-e^x) \geq 3e^{-y}$,
		\item Case 3.2: $e^{-y} \leq 2(2-e^x) \leq 1$ and $2(2-e^x) \leq 3e^{-y}$,
	\end{itemize}
	see~Figure~\ref{fig:asymptotic-case-0-push-the-cases}.
	The general idea is to minimize $\XXXasymptoticDenominator(x, y, 2(2-e^x)$
	along the lines where $\XXXasymptoticNominator(x, y, 2(2-e^x))$ is constant.
	This strategy and its outcome are depicted in Figure~\ref{fig:asymptotic-case-0-push-proof}, to which the Reader will be referred multiple times during the course of the proof.	
	For the sake of brevity, we will omit some standard details.

	\begin{figure}[b]
		\begin{tikzpicture}[scale = 3, xscale=3]
		\draw[->] (0, 0) -- (1.05, 0) node[below] {$x$};
		\draw[->] (0, 0) -- (0, 2.1) node[left] {$y$};
		\node[below] at ({ln(3/2)},0) {$\ln(\tfrac{3}{2})$};
		\node[below] at ({ln(2)},0) {$\ln(2)\vphantom{\tfrac{3}{2}}$};
		\node[left] at (0,{ln(3/2)}) {$\ln(\tfrac{3}{2})$};
		\node[left] at (0,{ln(3)}) {$\ln(3)$};
		\draw[domain=0:2, smooth, variable=\y, MyLineCases]  plot ({ln(3/2)}, {\y});
		\draw[domain=ln(3/2):0.65873, smooth, variable=\x, MyLineCases] plot ({\x}, {-ln(2*(2-e^\x))});
		\draw[domain=0:0.58612, smooth, variable=\x, MyLineCases] plot ({\x}, {-ln(2*(2-e^\x)/3)});
		\node[below] at (0.75,0.4) {Case 1};
		\node[below] at (0.2,0.4) {Case 2.2};
		\node[below] at (0.2,2) {Case 2.1};
		\node[below] at (0.49,2) {Case 3.1};
		\node[below] at (0.49,1.1) {Case 3.2};
		\end{tikzpicture}
		\caption{%
			Proof of Lemma~\ref{lem:asymptotic-push-case-0}.
			The dashed lines, $x = \ln(3/2)$, $3e^{-y} = 2(2-e^x)$, and $e^{-y} = 2(2-e^x)$,
			divide the first quadrant into five sets according to the formulas for
			$\XXXasymptoticNominator$ and $\XXXasymptoticDenominator$.
		}
		\label{fig:asymptotic-case-0-push-the-cases}
	\end{figure}
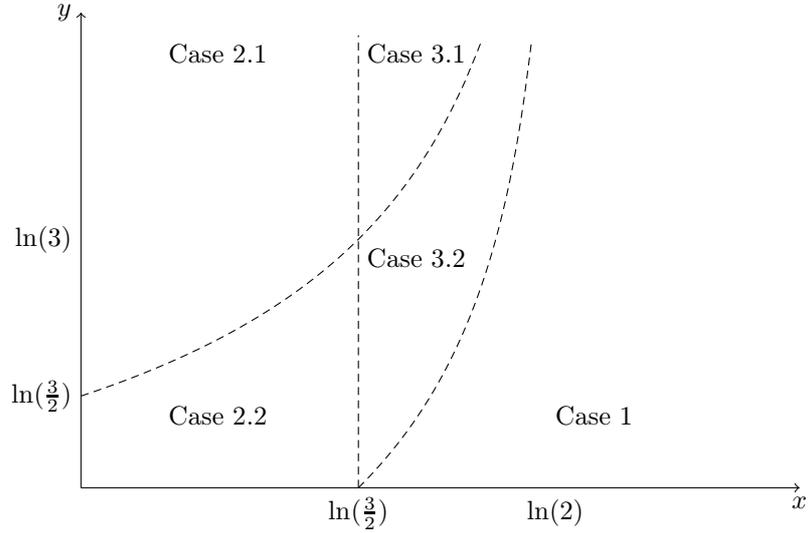
	
	\begin{figure}[b]
		\begin{tikzpicture}[scale=3, xscale=3]
		\draw[->] (0, 0) -- (1.05, 0) node[below] {$x$};
		\draw[->] (0, 0) -- (0, 2.1) node[left] {$y$};
		\node[below] at ({ln(3/2)},0) {$\ln(\tfrac{3}{2})$};
		\node[below] at ({ln(2)},0) {$\ln(2)\vphantom{\tfrac{3}{2}}$};
		\node[left] at (0,{ln(3/2)}) {$\ln(\tfrac{3}{2})$};
		\node[left] at (0,{ln(3)}) {$\ln(3)$};
		\draw[domain=0:2, smooth, variable=\y, MyLineCases]  plot ({ln(3/2)}, {\y});
		\draw[domain=ln(3/2):0.65873, smooth, variable=\x, MyLineCases] plot ({\x}, {-ln(2*(2-e^\x))});
		\draw[domain=0:0.58612, smooth, variable=\x, MyLineCases] plot ({\x}, {-ln(2*(2-e^\x)/3)});
		\draw[domain=ln(3/2):ln(2), smooth, variable=\y, MyLineOpt]  plot (0, {\y});
		\draw[domain=0:ln(3/2), smooth, variable=\x, MyLineOpt]  plot ({\x}, {\x + ln(2)});
		\draw[domain=ln(3):2, smooth, variable=\y, MyLineOpt]  plot ({ln(2*e^\y/(1+e^\y))}, {\y});
		\draw[domain=0:ln(3/2), smooth, variable=\x, MyLineOpt]  plot ({\x}, {\x + ln(2)});
		\draw[domain=0:ln(3/2), smooth, variable=\x, MyLineOpt]  plot ({\x}, {\x});
		\draw[ domain=ln(3/2):0.627, smooth, variable=\x, MyLineOpt]  plot ({\x}, {ln(e^\x/(4-2*e^\x))});
		\draw[domain=0:2, smooth, variable=\y, MyLineHelp]  plot ({ln(2)}, {\y});
		\draw[domain=0:ln(3/2), smooth, variable=\x, MyLineHelp]  plot ({\x}, {ln(3/2)});
		\draw[domain=0:ln(3/2), smooth, variable=\x, MyLineHelp]  plot ({\x}, {ln(3)});
		%
		%
		%
		\draw[MyArrowsR={0.5}, domain=0:0.1, variable=\x, MyLinePush]  plot ({\x}, {0.2-\x});
		\draw[MyArrowsL={0.5},domain=0.1:0.2, variable=\x, MyLinePush]  plot ({\x}, {0.2-\x});
		%
		%
		\draw[MyArrowsL={0.4}, domain=0:0.0925444, variable=\x, MyLinePush]
		plot ({\x}, {ln( (1 + sqrt(1 + 4*(2-e^\x)*e^(0.6 - \x)/2))/(2*(2-e^\x)) )});
		\draw[MyArrowsRR={0.33}{0.66}, domain=0.0925444:0.3, variable=\x, MyLinePush]  plot ({\x}, {0.6-\x});
		\draw[MyArrowsL={0.6}, domain=0.3:ln(3/2), variable=\x, MyLinePush]  plot ({\x}, {0.6-\x});
		\draw[MyArrowsL={0.5}, domain=ln(3/2):0.450667, variable=\x, MyLinePush]  plot ({\x}, {0.6-\x});
		\draw[MyArrowsL={0.5}, domain=0.451:0.455856, variable=\x, MyLinePush]
		plot ({\x}, {ln( (-1 + sqrt(1 - 4*(2-e^\x)*e^(0.6 - \x)/2))/(-2*(2-e^\x)) )});
		%
		%
		\draw[MyArrowsL={0.4}, domain=0:0.271214, variable=\x, MyLinePush]
		plot ({\x}, {ln( (1 + sqrt(1 + 4*(2-e^\x)*e^(1.05 - \x)/2))/(2*(2-e^\x)) )});
		\draw[MyArrowsR={0.5},domain=0.271214:ln(3/2), variable=\x, MyLinePush]  plot ({\x}, {1.05-\x});
		\draw[MyArrowsRL={0.2}{0.7},domain=ln(3/2):0.531906, variable=\x, MyLinePush]  plot ({\x}, {1.05-\x});
		\draw[MyArrowsL={0.5},domain=0.532:0.585392, variable=\x, MyLinePush]
		plot ({\x}, {ln( (-1 + sqrt(1 - 4*(2-e^\x)*e^(1.05 - \x)/2))/(-2*(2-e^\x)) )});
		%
		%
		\draw[MyArrowsRL={0.05}{0.55},domain=0:ln(3/2), variable=\x, MyLinePush]
		plot ({\x}, {ln( (1 + sqrt(1 + 4*(2-e^\x)*(9/2 )*e^(- \x)/2))/(2*(2-e^\x)) )});
		\draw[MyArrowsRL={0.4}{0.81}, domain=ln(3/2):ln(9/5), variable=\x, MyLinePush]  plot ({\x}, {ln(9/2)-\x});
		\draw[MyArrowsLL={0.33}{0.66},domain=0.588:ln(1/2 + sqrt(5/2)), variable=\x, MyLinePush]
		plot ({\x}, {ln( (-1 + sqrt(1 - 4*(2-e^\x)*(9/2)*e^(- \x)/2))/(-2*(2-e^\x)) )});
		%
		%
		%
		\draw[MyArrowsRR={0.33}{0.7},domain=0:ln(3/2), variable=\x, MyLinePush]
		plot ({\x}, {ln( (1 + sqrt(1 + 4*(2-e^\x)*(9/2 + 9/2 )*e^(- \x)/2))/(2*(2-e^\x)) )});
		\draw[MyArrowsRL={0.3}{0.8},domain=ln(3/2):2 *ln(2) + ln(3) - ln(7), variable=\x, MyLinePush]
		plot ({\x}, {ln( (1 + sqrt(1 + 4*(2-e^\x)*(9/2 + 9/2 )*e^(- \x)/2))/(2*(2-e^\x)) )});
		\draw[MyArrowsRL={0.3}{0.81},domain=ln(12/7):ln(36/19), variable=\x, MyLinePush]  plot ({\x}, {ln(9)-\x});
		\draw[MyArrowsLL={0.33}{0.66},domain=0.6392:ln(1/2 + sqrt(19)/2), variable=\x, MyLinePush]
		plot ({\x}, {ln( (-1 + sqrt(1 - 4*(2-e^\x)*(9)*e^(- \x)/2))/(-2*(2-e^\x)) )});
		\end{tikzpicture}
		\caption{%
			Proof of Lemma~\ref{lem:asymptotic-push-case-0} (cf.\ Figure~\ref{fig:asymptotic-case-0-push-the-cases}).
			On the dotted lines 
			$\XXXasymptoticNominator(x,y,2(2-e^x))$ is constant;
			the arrow heads indicate in which direction 
			$\XXXasymptoticDenominator(x,y,2(2-e^x))$ decreases
			and the thick lines indicate where the local minima  can lie.
			Note that the axes have been scaled for readability.	}
		\label{fig:asymptotic-case-0-push-proof}
	\end{figure}
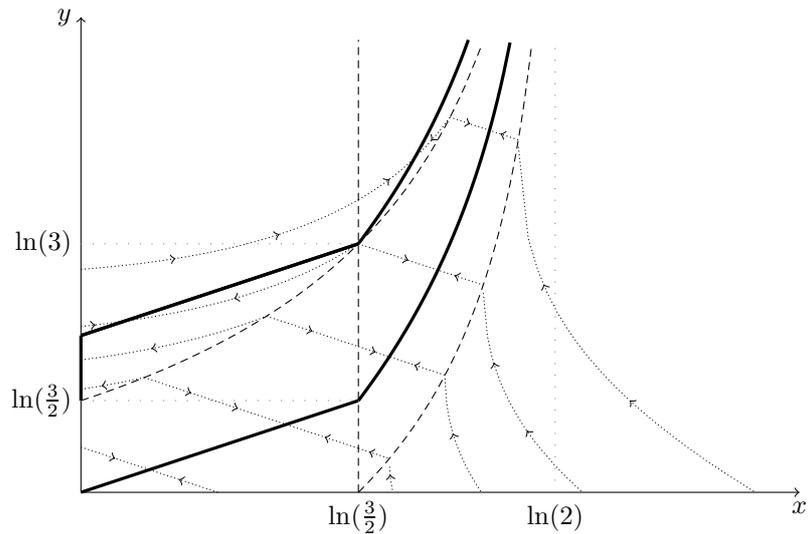

    \begin{proof}[Proof of Lemma~\ref{lem:asymptotic-push-case-0} in Case 1]
	Suppose that $2(2-e^x) \leq e^{-y}$.
	Then in particular $2(2-e^x) < 1$ (since $y>0$), i.e.,  $x> \ln(3/2)$,
	and
	\begin{align*}
	\XXXasymptoticNominator(x,y,2(2-e^x)) &= x + y + \ln(2-2(2-e^x)e^y),\\
	\XXXasymptoticDenominator(x,y,2(2-e^x)) &= 2e^x - x -2 + \int_{0}^{y} 1 -  2(2-e^x)  e^{s} ds\\
	& = 2e^x - x- 2 + y - 2(2-e^x)(e^y-1).
\end{align*}
Consider the implicit curve described by the equation $\XXXasymptoticNominator(x,y,2(2-e^x)) = c$
(for some constant $c$ and $x$, $y$ such as in the current case).
Since $x\mapsto \XXXasymptoticNominator(x,y,2(2-e^x)) $ is strictly increasing, on this implicit curve
	$x=x(y)$ is determined by the value of $y$;
	Actually, an explicit formula for $x(y)$ can be found
	and one easily sees that $x(y)$ is differentiable 
	(if we are in the interior of the set considered in Case~1).
	Differentiating $\XXXasymptoticNominator(x(y),y,2(2-e^{x(y)}))$
	 with respect to $y$ we
	get the condition
	\begin{align}
	\label{eq:derivative-implicit-curve-case-0.1}
	0 &= x'(y) + 1 + \frac{2e^{x(y)} x'(y) e^y - 2(2-e^x) e^y}{2-2(2-e^{x(y)})e^y}
	\\
	&= x'(y) \frac{1 - (2 - e^{x(y)}) e^y + e^{x(y)} e^y}{1 - (2 - e^{x(y)}) e^y} + \frac{1 - 2(2 - e^{x(y)}) e^y}{1 - (2 - e^{x(y)}) e^y}.\nonumber
	\end{align}
	In particular, $x'(y) \leq 0$ (since $2(2-e^{x}) e^y \leq 1$).
	Using \eqref{eq:derivative-implicit-curve-case-0.1} in the third inequality below we get
	\begin{align*}
	\MoveEqLeft[4]
	\frac{\partial}{\partial y} \XXXasymptoticDenominator(x(y), y, 2(2-e^{x(y)})) \\
	&= x'(y)(2e^{x(y)} - 1) + 1 
	+ 2e^{x(y)}x'(y) (e^y - 1) - 2(2-e^{x(y)}) e^y \\
	&= x'(y)(2e^{x(y)}e^y - 1) + 1  - 2(2-e^{x(y)}) e^y \\
	&= \frac{-\bigl(1 - 2(2 - e^{x(y)}) e^y\bigr)(2e^{x(y)}e^y - 1)}{1 - (2 - e^{x(y)}) e^y + e^{x(y)} e^y} + 1  - 2(2-e^{x(y)}) e^y \\	
	&= \frac{-\bigl(1 - 2(2 - e^{x(y)}) e^y\bigr)\bigl(2e^{x(y)}e^y - 1     - 1 + (2 - e^{x(y)}) e^y - e^{x(y)} e^y\bigr)}{1 - (2 - e^{x(y)}) e^y + e^{x(y)} e^y}  \\	
	&= \frac{-\bigl(1 - 2(2 - e^{x(y)}) e^y\bigr)(2 e^y - 2)}{1 - (2 - e^{x(y)}) e^y + e^{x(y)} e^y} ,	
	\end{align*}
	which is negative. 
	Hence, in order to minimize $\XXXasymptoticDenominator(x(y), y, 2(2-e^{x(y)}))$
	(while keeping $\XXXasymptoticNominator(x(y), y, 2(2-e^{x(y)}))$ constant)
	it is best to increase $y$ and decrease $x=x(y)$.
	We can do so without leaving the setting of Case~1
	 until we arrive at the situation where 
	$ 2(2-e^{x(y)}) = e^{-y}$
	(cf.\ Figure~\ref{fig:asymptotic-case-0-push-proof}).
	This reduces Case~1 to the Case~3 which we  consider below.
	\end{proof}

    \begin{proof}[Proof of Lemma~\ref{lem:asymptotic-push-case-0} in Case 2]
	Suppose that $2(2-e^x) \geq 1$.
	This is equivalent to $x\leq \ln(3/2)$ and hence
	\begin{align*}
	\XXXasymptoticDenominator(x,y,2(2-e^x)) &= 2e^x - x -2 + \int_{0}^{y} -1 +  2(2-e^x)  e^{s} ds\\
	& = 2e^x - x- 2 - y + 2(2-e^x)(e^y-1)\\
	&= 4e^x + 4 e^y -2e^{x+y} -x -y - 6. 
	\end{align*}
	We consider two subcases.
	
	\emph{Case~2.1}. Suppose that $2(2-e^x) \geq 1 $ and moreover $2(2-e^x) \geq 3e^{-y}$.
	Then
	\begin{equation*}
	\XXXasymptoticNominator(x,y,2(2-e^x)) = x + y + \ln(2(2-e^x) e^y  -2).
	\end{equation*}
	Consider the implicit curve described by the equation $\XXXasymptoticNominator(x,y,2(2-e^x)) = c$
	(for some constant $c$ and $x$, $y$ such as in the current subcase).
	Since $y\mapsto \XXXasymptoticNominator(x,y,2(2-e^x)) $ is strictly increasing, on this implicit curve
	$y = y(x)$ is determined by the value of $x$;
   an explicit formula for $y(x)$ can be found
		and one sees that $y(x)$ is differentiable 
		(if we are in the interior of the set considered in Case~2.1).
	Differentiating $\XXXasymptoticNominator(x,y(x),2(2-e^{x})) $ with respect to $x$
	yields the condition
	\begin{align}
	\label{eq:derivative-implicit-curve-case-0.2.a}
	0 &= 1 + y'(x) + \frac{-2e^{x} e^{y(x)} + 2(2-e^x) e^{y(x)} y'(x)}{2(2-e^{x})e^{y(x)} - 2}
	\\
	&= 
	\frac{(2-e^{x})e^{y(x)} - 1 -e^{x} e^{y(x)}}{(2-e^{x})e^{y(x)} - 1}
	+ y'(x)\frac{2(2-e^x) e^{y(x)} - 1}{(2-e^{x})e^{y(x)} - 1}
	.\nonumber
	\end{align}
	In particular, $y'(x) \geq 0$
	 (since $2(2-e^{x}) e^{y(x)} \geq 3$
	 and $(2-e^{x})e^{y(x)} - 1 -e^{x} e^{y(x)}
	 = 2(1-e^{x})e^{y(x)} - 1 \geq 0$).
	Using \eqref{eq:derivative-implicit-curve-case-0.2.a} in the third inequality below we get
	\begin{align*}
	\MoveEqLeft[4]
	\frac{\partial}{\partial x} \XXXasymptoticDenominator(x, y(x), 2(2-e^{x})) \\
	&= 2e^{x} - 1  - y'(x)
	- 2e^{x} (e^{y(x)} - 1) + 2(2-e^{x}) e^{y(x)} y'(x) \\
	&=2e^{x} - 1 - 2e^{x} (e^{y(x)} - 1) 
	+ y'(x)\bigl( - 1  + 2(2-e^{x}) e^{y(x)} \bigr) \\
	&=2e^{x} - 1 - 2e^{x} (e^{y(x)} - 1) 
	 - \bigl((2-e^{x})e^{y(x)} - 1 -e^{x} e^{y(x)}\bigr) \\
	 &=4e^x  - 2e^{y(x)},
	\end{align*}
	which is of the same sign as
	\[
	x + \ln(2) - y(x). 
	\]
	It follows that if $x$ increases,
	then $\XXXasymptoticDenominator(x,y(x),2(2-e^x))$ decreases while $x + \ln(2)  \leq  y(x)$ 
	and increases when $x + \ln(2) \geq y(x)$.
	Hence, in order to minimize the function  $\XXXasymptoticDenominator(x, y(x), 2(2-e^{x}))$
	 on the implicit curve where $\XXXasymptoticNominator(x,y(x), 2(2-e^x))$ is constant
	it suffices to consider the following special cases:
	\begin{itemize}
		\item $y = x + \ln(2)$, $x \in (0, \ln(3/2)]$,
		\item $x = 0$ and $y \in [\ln(3/2), \ln(2)]$,
		\item $y = \ln(3/2)$, $y\geq \ln(3)$
	\end{itemize}
(cf.~Figure~\ref{fig:asymptotic-case-0-push-proof};
 note that, depending on the value of $c$,
  the curve $\XXXasymptoticNominator(x,y(x), 2(2-e^x)) = c$
  may have zero, one, or two intersection points
   with the line $x + \ln(2) = y$).
In the two first situations 
the expression for $\XXXasymptoticNominator(x,y(x), 2(2-e^x))/\XXXasymptoticDenominator(x,y(x), 2(2-e^x))$ simplifies and the assertion follows from Lemma~\ref{lem:asymptotic-push-case-0.2.a}.
The third situation reduces to Case~3
 which we consider below.

	\emph{Case~2.2}. Suppose that $2(2-e^x) \geq 1 $ and $2(2-e^x) \leq 3e^{-y}$.
	Then
	\begin{equation*}
	\XXXasymptoticNominator(x,y,2(2-e^x)) = x + y.
	\end{equation*}
	On the lines $x + y = c$ (where $c>0$ is some constant)
	the function
	\[ (x,y)\mapsto \XXXasymptoticDenominator(x,y,2(2-e^x)) =  4e^x + 4 e^y -2e^{x+y} -x -y - 6
	\]
	treated as a function of $x\in[0,c]$ is clearly a convex function symmetric with respect to the midpoint of the interval $[0,c]$.
	Hence, in order to minimize $\XXXasymptoticDenominator(x,y, 2(2-e^x))$
	(while keeping $\XXXasymptoticNominator(x,y,2(2-e^x)$ constant)
	it suffices to consider  the following special cases:
	\begin{itemize}
		\item $x=y$, $x\in(0,\ln(3/2)]$,
		\item $x=\ln(3/2)$, $y\in[\ln(3/2),\ln(3)]$
	\end{itemize}
(cf.~Figure~\ref{fig:asymptotic-case-0-push-proof}).
	In the first situation the expression
	 for $U(x,y, 2(2-e^x))/V(x,y, 2(2-e^x))$ simplifies greatly
	  and Lemma~\ref{lem:asymptotic-push-case-0.2.b} yields the assertion. 
	The second situation reduces to Case~3, which we consider below.
	\end{proof}
	
	\begin{proof}[Proof of Lemma~\ref{lem:asymptotic-push-case-0} in Case 3]
	Suppose that $e^{-y} \leq 2(2-e^x) \leq 1$.
	This implies that $\ln(3/2) \leq x < \ln(2)$
	 and the function $s \mapsto -1 + 2(2-e^x) e^s$ changes sign at $s_0 = s_0(x)= -\ln(2(2-e^x)) \in [0,y]$. 
	Hence
	\begin{align*}
	\XXXasymptoticDenominator(x,y,2(2-e^x)) &= 2e^x - x -2 + 
	\int_{0}^{ s_0} 1 -  2(2-e^x)  e^{s} ds
	+ \int_{ s_0 }^{y} -1 +  2(2-e^x)  e^{s} ds\\
	& = 2e^x - x- 2  + 2s_0  - y  + 2(2-e^x)(e^y - 2e^{s_0} + 1)\\
 	& = 2e^x - x-  4  - 2\ln(2(2-e^x)) - y   + 2(2-e^x)(e^y + 1).
	\end{align*}
	 We consider two subcases. 
	
	\emph{Case~3.1}. Suppose that $e^{-y} \leq 2(2-e^x) \leq 1 $ and moreover $2(2-e^x) \geq 3e^{-y}$.
	Then
	\begin{equation*}
	\XXXasymptoticNominator(x,y,2(2-e^x)) = x + y + \ln(2(2-e^x) e^y  -2).
	\end{equation*}
	Consider the implicit curve described by the equation  $\XXXasymptoticNominator(x,y,2(2-e^x)) = c$
	(for some constant $c$ and $x$, $y$ such as in the current subcase).
	As in Case~2.1 above,
	on this implicit curve
	$y = y(x)$ is determined by the value of $x$ and we have the condition~\eqref{eq:derivative-implicit-curve-case-0.2.a}.
	Using it in the third inequality below we get 
	\begin{align*}
	\MoveEqLeft[4]
	\frac{\partial}{\partial x} \XXXasymptoticDenominator(x, y(x), 2(2-e^{x})) \\
	&= 2e^{x} - 1  + \frac{2e^x}{2-e^x}
	-y'(x)
	- 2e^{x} (e^{y(x)} + 1) + 2(2-e^{x}) e^{y(x)} y'(x) \\
	&=2e^{x} - 1  + \frac{2e^x}{2-e^x} - 2e^{x} (e^{y(x)} + 1) 
	+ y'(x)\bigl( - 1  + 2(2-e^{x}) e^{y(x)} \bigr) \\
	&=2e^{x} - 1 + \frac{2e^x}{2-e^x}- 2e^{x} (e^{y(x)} + 1) 
	- \bigl((2-e^{x})e^{y(x)} - 1 -e^{x} e^{y(x)}\bigr) \\
	&=\frac{2e^x}{2-e^x} - 2e^{y(x)}.
	\end{align*}
	It follows that if $x$ increases,
	then $\XXXasymptoticDenominator(x,y(x),2(2-e^x))$ decreases 
	while $e^x -(2-e^x)e^{y(x)} \leq 0$
	and increases when $e^x - (2-e^x) e^{y(x)}\geq 0$.
	Note that the curve $(2-e^x)e^y =e^x$ with $x\in[\ln(3/2), \ln(2))$
	lies in the region currently considered since $e^x \geq \ln(3/2)$.
	Thus Case~3.1 reduces to the situation:
	\begin{itemize}
		\item $e^x - 2e^{y} + e^x e^{y} = 0$, $x\in [\ln(3/2), \ln(2))$
		\end{itemize}
	(cf.\ Figure~\ref{fig:asymptotic-case-0-push-proof}).
    For such $x$, $y$ the expression for $\XXXasymptoticNominator(x,y, 2(2-e^x))/\XXXasymptoticDenominator(x,y, 2(2-e^x))$ simplifies 
    and Lemma~\ref{lem:asymptotic-push-case-0.3.a} yields the assertion. 
		
	\emph{Case~3.2}. Suppose that $e^{-y} \leq 2(2-e^x) \leq 1 $ and moreover $2(2-e^x) \leq 3e^{-y}$.
	Then
	\begin{equation*}
	\XXXasymptoticNominator(x,y,2(2-e^x)) = x + y
	\end{equation*}
	and clearly the curve described by the equation $\XXXasymptoticNominator(x,y,2(2-e^x)) = c$
	(for some constant $c$ and $x$, $y$ such as in the current subcase) determines $y=y(x)$ as a function of $x$ with $y'(x) = -1$.
	Therefore,
\begin{align*}
\MoveEqLeft[4]
\frac{\partial}{\partial x} \XXXasymptoticDenominator(x, y(x), 2(2-e^{x})) \\
&= 2e^{x} - 1  + \frac{2e^x}{2-e^x}
-(-1)
- 2e^{x} (e^{y(x)} + 1) + 2(2-e^{x}) e^{y(x)}(-1) \\
&=\frac{2e^x}{2-e^x} - 4e^{y(x)},
\end{align*}
which is of the same sign as 
\[
e^x - 2(2-e^x) e^{y(x)} = e^x + 2 e^{x + y(x)} - 4 e ^{y(x)}.
\]
The last quantity clearly increases when $x$ increases (and $y(x)$ decreases, while $x+y(x)$ is constant).
The curve $e^x - 2(2-e^x) e^{y} = 0$ with $x\in [\ln(3/2), \ln(2))$
lies in the currently considered region since 
$e^x \in [3/2, 2) \subset [1,3]$
and we are currently in the subcase where $1 \leq  2(2-e^x)e^y \leq 3$.
 Hence Case~3.1 reduces to one of the situations:
	 \begin{itemize}
	 	\item $e^x + 2 e^x e^{y} - 4 e ^{y} = 0$, $x\in [\ln(3/2), \ln(2))$,
	 	\item $x=\ln(3/2)$, $y\in (0,\ln(3/2)]$
	 \end{itemize}
	 (cf.\ Figure~\ref{fig:asymptotic-case-0-push-proof}).
	 The first situation corresponds directly to the assertion of the lemma,
	 while the second reduces to the already fully solved part of the Case~2.2
	 (the one which was reduced to Lemma~\ref{lem:asymptotic-push-case-0.2.b} and not to Case~3, cf.\ Figure~\ref{fig:asymptotic-case-0-push-proof}).
	 This finishes the proof in Case 3 and thus also the proof of the lemma.
\end{proof}

\subsection{Outline of the proof for general \texorpdfstring{$z$}{z}}

The proof of Proposition~\ref{prop:asymptotic-push} for general $z$ is split into several 
cases and steps
(depending on where the number $z$ lies with respect to the interval $[e^{-y}, 1]$):
\begin{itemize}
	\item Case 1: $z \leq e^{-y}$,
	\item Case 2.1.1: $z \geq 1$ and $z \geq 3e^{-y}$ and $x\geq \ln(3/2)$,
	\item Case 2.1.2: $z \geq 1$ and $z \geq 3e^{-y}$ and $x\leq \ln(3/2)$,
	\item Case 2.2: $z \geq 1$ and $z \leq 3e^{-y}$,
	\item Case 3.1: $e^{-y} \leq z \leq 1$ and $z \geq 3e^{-y}$,
	\item Case 3.2: $e^{-y} \leq z \leq 1$ and $z \leq 3e^{-y}$,
\end{itemize}
which we present in separate subsections below. 
The reasoning is similar to the one presented in the proof of Lemma~\ref{lem:asymptotic-push-case-0}, but slightly more tedious due to one extra variable.
We omit some standard details, but try to illustrate the steps with the help of figures where possible
(note that everywhere the axes have been scaled for readability).

\subsection{Case 1: \texorpdfstring{$z\leq e^{-y}$}{z leq exp(-y)}}
	Suppose that $z\leq e^{-y}$ (i.e., $ze^y \leq 1$).
	The standing assumptions~\eqref{eq:asymptotic-standing-assumption} imply   that
	$
	2(2-e^x) \leq z \leq e^{-y} \leq 1
	$
	and in particular $z\leq 1$, $x\geq \ln(3/2)$.
	Consequently, 
	\begin{align*}
	\XXXasymptoticNominator(x,y,z) &= x + \ln(2e^x-2) - \ln(2-z) + y + \ln(2-ze^y),\\
	\XXXasymptoticDenominator(x,y,z) &= 2e^x - x -2 + \int_{0}^{y} 1 -  z  e^{s} ds\\
	& = 2e^x - x- 2 + y - z(e^y-1).
	\end{align*}
\begin{proof}[Proof of Proposition~\ref{prop:asymptotic-push} in Case 1]
	Fix $z$ and consider the implicit curve described by the equation $\XXXasymptoticNominator(x,y,z) = c$ for some constant $c$. Since $x\mapsto x + \ln(2e^x - 2)$ is strictly increasing, on this implicit curve
	$x=x(y)$ is determined by the value of $y$.
	Differentiating $\XXXasymptoticNominator(x(y),y,z)$ with respect to $y$ we
	get the condition
	\begin{align}
	\label{eq:derivative-implicit-curve-case-1}
	0 = x'(y) + \frac{2e^{x(y)} x'(y)}{2e^{x(y)} - 2} + 1 - \frac{ze^y}{2-ze^y}
	= x'(y) \frac{(2 e^{x(y)}-1)}{e^{x(y)} - 1} + \frac{2(1-ze^y)}{2-ze^y}.
	\end{align}
	In particular, $x'(y)\leq 0$ (since $ze^y \leq 1$).
	Using \eqref{eq:derivative-implicit-curve-case-1} in the second inequality below we get
	\begin{align*}
	\frac{\partial}{\partial y} \XXXasymptoticDenominator(x(y), y, z) 
	&= x'(y)(2e^{x(y)} - 1) + 1 - z e^y \\
	&= \frac{-2(1-ze^y)(e^{x(y)} - 1)}{2-ze^y} + 1 - ze^{y}\\
	&= \frac{(1-ze^y)\bigl(2(2-e^{x(y)}) -ze^y\bigr)}{2-ze^y},
	\end{align*}
	which is negative,  since $ze^{y} \leq 1$ and $2(2-e^{x(y)}) \leq z \leq ze^y$.
	It follows that in Case~1 in order to minimize 
	$\XXXasymptoticDenominator(x(y),y,z) $ on the implicit curve where $\XXXasymptoticNominator(x(y),y,z) $ is constant,
	it is best to increase $y$ and decrease $x=x(y)$.
	For fixed $z$ we can do so (without leaving the setting of Case~1) until we arrive at one of the two possibilities:
	\begin{itemize}
		\item $2(2-e^x) \leq  z = e^{-y} \leq 1$,
		\item $2(2-e^x) = z \leq e^{-y} \leq 1$.
	\end{itemize}
	The former subcase will be treated in Case~3 below (where we assume that $z\in [e^{-y}, 1]$). 	
	In the latter situation we can apply~Lemma~\ref{lem:asymptotic-push-case-0}. 
	\end{proof}

	\subsection{Case 2: \texorpdfstring{$z\geq 1$}{z geq 1}}

	In this subsection we assume that $z\geq 1$. Then
	\begin{equation*}
	\XXXasymptoticDenominator(x,y,z) = 2e^x - x -2 + \int_{0}^{y} - 1 + z  e^{s} ds = 2e^x - x- 2 - y + z(e^y-1).
	\end{equation*}

%
%
\begin{figure}[t]
	\begin{tikzpicture}[scale=3,xscale=3]
	\draw[->] (0, 0) -- (1.25, 0) node[below] {$x$};
	\draw[->] (0, 0) -- (0,1.75) node[left] {$y$};
	\node[left] at ({ln(3/2)},1.7) {Case 2.1.2};
	\node[right] at ({ln(3/2)+0.1},1.7) {Case 2.1.1};
	\node[below] at (0,0) {$0\vphantom{\tfrac{3}{2}}$};
	\node[below] at ({ln(2-\zCaseTwo/2)},0) {$\ln(2-\tfrac{z}{2})\vphantom{\tfrac{3}{2}}$};
	\node[below] at ({ln(3/2)},0) {$\ln(\tfrac{3}{2})$};
	\node[left] at (0,{ln(2/(\zCaseTwo))}) {$\ln(\tfrac2z)$};
	\node[left] at (0,{ln(3/(\zCaseTwo))}) {$\ln(\tfrac3z)$};
	\node[left] at (0,{ln(4/(\zCaseTwo))}) {$\ln(\tfrac4z)$};
	\node[left] at (0,{ln(6/(\zCaseTwo))}) {$\ln(\tfrac6z)$};
	\draw[domain=0:1.7,variable=\y, MyLineCases] plot ({ln(2-\zCaseTwo/2)}, {\y});
	\draw[domain=ln(3/(\zCaseTwo)):1.7, variable=\y, MyLineCases] plot ({ln(3/2)}, {\y});
	\draw[domain=0:ln(3/(\zCaseTwo)), variable=\y, MyLineHelp] plot ({ln(3/2)}, {\y});
	\draw[domain=ln(2-\zCaseTwo/2):1.2, variable=\x, MyLineCases] plot (\x, {ln(3/(\zCaseTwo))});
	\draw[domain=0:ln(2-\zCaseTwo/2), variable=\x, MyLineHelp] plot (\x, {ln(3/(\zCaseTwo))});
	\draw[domain=0:ln(3/2), variable=\x, MyLineHelp] plot (\x, {ln(6/(\zCaseTwo))});
	\draw[domain=ln(3/2):ln(3)-0.15, variable=\x, MyLineOpt] plot ({\x}, {\x + ln(2) - ln(\zCaseTwo)});
	\draw[domain=0:ln(3/2), variable=\x, MyLineHelp] plot ({\x}, {\x + ln(2) - ln(\zCaseTwo)});
	\draw[domain=ln(2-\zCaseTwo/2):ln(3/2), variable=\x, MyLineOpt] plot ({\x}, {\x + ln(4) - ln(\zCaseTwo)});
	\draw[domain=0:ln(2-\zCaseTwo/2), variable=\x, MyLineHelp] plot ({\x}, {\x + ln(4) - ln(\zCaseTwo)});
	\draw[domain=ln(3/(\zCaseTwo)):ln(8/(\zCaseTwo)-2), variable=\y, MyLineOpt] plot ({ln(2-\zCaseTwo/2)}, {\y});
	\draw[MyArrowsR={0.5},
	domain=ln(2- (\zCaseTwo)/2):ln(3/2), variable=\x, MyLinePush]
	plot ({\x}, {ln((1+sqrt(1+\zCaseTwo*(36/(\zCaseTwo))*e^(-\x)))/(\zCaseTwo))});
	\draw[MyArrowsL={0.5},
	domain=ln(2- (\zCaseTwo)/2):ln(3/2), variable=\x, MyLinePush]
	plot ({\x}, {ln((1+sqrt(1+\zCaseTwo*((4-\zCaseTwo)*(4-\zCaseTwo)*(6-2*\zCaseTwo)/(\zCaseTwo))*e^(-\x)))/(\zCaseTwo))});
	\draw[MyArrowsL={0.5}, domain=ln(2- (\zCaseTwo)/2):ln(3/2), variable=\x, MyLinePush]
	plot ({\x}, {ln((1+sqrt(1+\zCaseTwo*(9/2/(\zCaseTwo))*e^(-\x)))/(\zCaseTwo))});
	\draw[MyArrowsL={0.5}, domain=ln(2- (\zCaseTwo)/2):ln(3/2), variable=\x, MyLinePush]
	plot ({\x}, {ln((1+sqrt(1+\zCaseTwo*(7/(\zCaseTwo))*e^(-\x)))/(\zCaseTwo))});
	\draw[MyArrowsL={0.5}, domain=ln(2- (\zCaseTwo)/2):ln(3/2), variable=\x, MyLinePush]
	plot ({\x)}, {ln((1+sqrt(1+\zCaseTwo*(12/(\zCaseTwo))*e^(-\x)))/(\zCaseTwo))});
	\draw[MyArrowsRL={0.25}{0.75}, domain=ln(2- (\zCaseTwo)/2):ln(3/2), variable=\x, MyLinePush]
	plot ({\x}, {ln((1+sqrt(1+\zCaseTwo*(26/(\zCaseTwo))*e^(-\x)))/(\zCaseTwo))});
	\draw[MyArrowsR={0.5}, domain=ln(2- (\zCaseTwo)/2):ln(3/2), variable=\x, MyLinePush]
	plot ({\x}, {ln((1+sqrt(1+\zCaseTwo*(46/(\zCaseTwo))*e^(-\x)))/(\zCaseTwo))});
	\draw[MyArrowsRL={0.25}{0.75}, domain=ln(3/2):ln((1+sqrt(1+2*\zCaseTwo*(36/(\zCaseTwo))/3))/2), variable=\x, MyLinePush]
	plot ({\x}, {ln((1+sqrt(1+\zCaseTwo*(36/(\zCaseTwo))*e^(-\x)/(2*e^\x -2)))/(\zCaseTwo))});
	\draw[MyArrowsRL={0.25}{0.75}, domain=ln(3/2):ln((1+sqrt(1+2*\zCaseTwo*((4-\zCaseTwo)*(4-\zCaseTwo)*(6-2*\zCaseTwo)/(\zCaseTwo))/3))/2), variable=\x, MyLinePush]
	plot ({\x}, {ln((1+sqrt(1+\zCaseTwo*((4-\zCaseTwo)*(4-\zCaseTwo)*(6-2*\zCaseTwo)/(\zCaseTwo))*e^(-\x)/(2*e^\x -2)))/(\zCaseTwo))});
	\draw[MyArrowsRL={0.25}{0.75}, domain=ln(3/2):ln((1+sqrt(1+2*\zCaseTwo*(7/(\zCaseTwo))/3))/2), variable=\x, MyLinePush]
	plot ({\x}, {ln((1+sqrt(1+\zCaseTwo*(7/(\zCaseTwo))*e^(-\x)/(2*e^\x -2)))/(\zCaseTwo))});
	\draw[MyArrowsRL={0.25}{0.75}, domain=ln(3/2):ln((1+sqrt(1+2*\zCaseTwo*(12/(\zCaseTwo))/3))/2), variable=\x, MyLinePush]
	plot ({\x}, {ln((1+sqrt(1+\zCaseTwo*(12/(\zCaseTwo))*e^(-\x)/(2*e^\x -2)))/(\zCaseTwo))});
	\draw[MyArrowsRL={0.25}{0.75}, domain=ln(3/2):ln((1+sqrt(1+2*\zCaseTwo*(26/(\zCaseTwo))/3))/2), variable=\x, MyLinePush]
	plot ({\x}, {ln((1+sqrt(1+\zCaseTwo*(26/(\zCaseTwo))*e^(-\x)/(2*e^\x -2)))/(\zCaseTwo))});
	\draw[MyArrowsRL={0.25}{0.75}, domain=ln(3/2):ln((1+sqrt(1+2*\zCaseTwo*(46/(\zCaseTwo))/3))/2), variable=\x, MyLinePush]
	plot ({\x}, {ln((1+sqrt(1+\zCaseTwo*(46/(\zCaseTwo))*e^(-\x)/(2*e^\x -2)))/(\zCaseTwo))});
	\end{tikzpicture}
	\caption{%
		Proof of Proposition~\ref{prop:asymptotic-push}, 
		Cases 2.1.1 and 2.1.2, where $z \geq 1$.
		On the dotted lines 
		$\XXXasymptoticNominator(x(y),y, z)$ is constant;
		the arrow heads indicate in which direction 
		$\XXXasymptoticDenominator(x(y),y,z)$ decreases
		and
		the thick lines indicate where the local minima can lie.}
	\label{fig:asymptotic-case-2}
\end{figure}
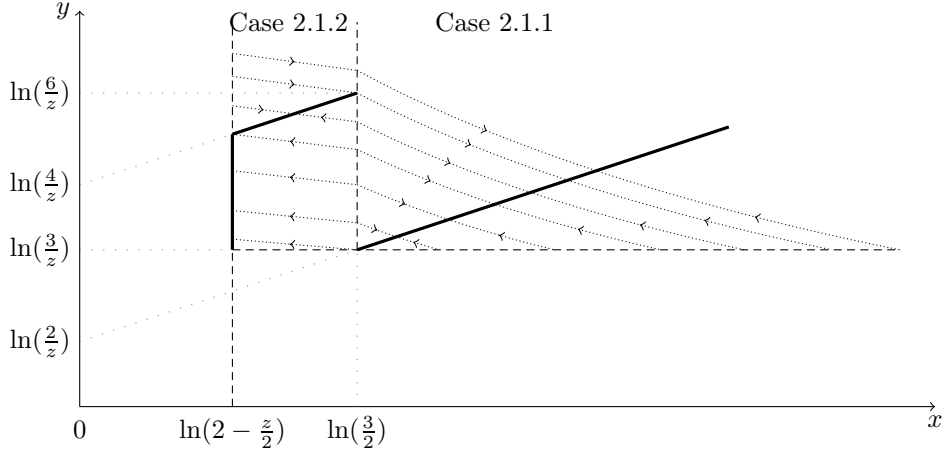

\begin{proof}[Proof of Proposition~\ref{prop:asymptotic-push} in Case 2]
Assume that $z\geq 1$.	
We consider three subcases.

\emph{Case 2.1.1}.
	Suppose that $z \geq 1$ and moreover $z \geq 3e^{-y}$, $x\geq \ln(3/2)$.
	Then in particular $2(2-e^x)\leq 1$
	and consequently
	\begin{equation*}
	\XXXasymptoticNominator(x,y,z) = x + \ln(2e^x-2) + y + \ln(ze^y - 2).
	\end{equation*}
	Fix $z$ and consider the implicit curve described by the equation $\XXXasymptoticNominator(x,y,z) = c$ for some constant $c$. Since $x\mapsto x + \ln(2e^x - 2)$ is strictly increasing, on this implicit curve
	$x=x(y)$ is determined by the value of $y$.
	Differentiating $\XXXasymptoticNominator(x(y), y,z)$ with respect to $y$ we
	get the condition
	\begin{align}
	\label{eq:derivative-implicit-curve-case-2.b}
	0 = x'(y) + \frac{2e^{x(y)} x'(y)}{2e^{x(y)} - 2} + 1 + \frac{ze^y}{ze^y-2}
	= x'(y) \frac{(2 e^{x(y)}-1)}{e^{x(y)} - 1} + \frac{2(ze^y - 1
		)}{ze^y-2}.
	\end{align}
	In particular, $x'(y)\leq 0$.
	Using \eqref{eq:derivative-implicit-curve-case-2.b} in the second inequality below we get
	\begin{align*}
	\frac{\partial}{\partial y} \XXXasymptoticDenominator(x(y), y, z) 
	&= x'(y)(2e^{x(y)} - 1) - 1 + z e^y \\
	&= \frac{-(ze^y-1)(2e^{x(y)} - 2)}{ze^y-2} - 1 + ze^{y}\\
	&= \frac{(ze^y -1)(ze^y - 2e^{x(y)})}{ze^y-2},
	\end{align*}
	which is of the same sign as $ze^y - 2e^{x(y)}$ (since $ze^{y} \geq 3)$.
	It follows that if $y$ increases (and  $x(y)$ decreases),
	then $\XXXasymptoticDenominator(x(y),y,z)$ increases if $y \geq x(y) + \ln(2/z)$, 
	and decreases when $y\leq x(y) + \ln(2/z)$. 
	Hence, in Case~2.1.1 and for fixed $z$, 
	in order to minimize $\XXXasymptoticDenominator$ on the implicit curve where $\XXXasymptoticNominator$ is constant,
	it suffices to consider only the situation:
	\begin{itemize}
		\item $ze^y = 2e^x$, $x\geq \ln(3/2)$, and $z\geq 1$
	\end{itemize}
	(cf.~Figure~\ref{fig:asymptotic-case-2}).
	If $ze^y = 2e^x$, i.e.,  $y = x + \ln(2/z)$, then
	\begin{align*}
	\XXXasymptoticNominator(x,x + \ln(2/z), z)   &= 2x + 2\ln(2e^x-2) + \ln(2/z),\\
	\XXXasymptoticDenominator(x,x+\ln(2/z), z) &= 4e^x - 2x - 2 - \ln(2/z)  - z.
	\end{align*}
	On the implicit curve $\XXXasymptoticNominator(x,x + \ln(2/z), z) = c $,
	$z = z(x)$ is clearly determined by the value of $x$.
	Differentiating  $\XXXasymptoticNominator(x,x + \ln(2/z(x)), z(x))$
	with respect to $x$ we arrive at the condition
	\begin{equation}
	\label{eq:derivative-implicit-curve-case-2.1.1-no-y}
	0 = 2 + \frac{2e^x}{e^x - 1} -\frac{z'(x)}{z(x)} 
	= \frac{4e^x - 2 }{e^x - 1} -\frac{z'(x)}{z(x)}.
	\end{equation}
	In particular, $z'(x) \geq 0$.
	Using~\eqref{eq:derivative-implicit-curve-case-2.1.1-no-y}, we obtain
	\begin{align*}
	\frac{\partial }{\partial x} \XXXasymptoticDenominator(x,x+\ln(2/z(x)), z(x)) 
	&= 4e^x - 2  + \frac{z'(x)}{z(x)} - z'(x)
	\\ 
	&=     4e^x - 2  + \frac{(4 e^x - 2)(1-z(x))}{e^x - 1}\\
	&= \frac{(4 e^x - 2)(e^x -z(x))}{e^x - 1}, 
	\end{align*}
	which is positive since $x >0$ and $z\leq 2$ (due to the standing assumption~\eqref{eq:asymptotic-standing-assumption}).
	Hence, in order to minimize $\XXXasymptoticDenominator(x,x+\ln(2/z(x)), z(x))$
	on the implicit curve where $\XXXasymptoticNominator(x,x + \ln(2/z), z)$ is constant,
	it is best to take $x$ (and thus also $z(x)$) as small as possible.
	We can decrease $x$ and decrease $z(x)$ until we arrive at the situation where either $z=1$ or $x=\ln(3/2)$ (see Figure~\ref{fig:asymptotic-case-2.1.1-no-y}).
	The former will be treated in Case~3 below (where we assume that $z\in [e^{-y}, 1]$). 	
	In the latter situation we can apply~Lemma~\ref{lem:asymptotic-push-case-2.1.1-no-y-x=ln3/2}.
	
	\begin{figure}[b]
		\begin{tikzpicture}[scale=2,xscale=2]
		\draw[->] (0, 0) -- (1,
		0) node[below] {$x$};
		\draw[->] (0, 0) -- (0,2.2) node[left] {$z$};
		\node[below] at ({ln(3/2)},0) {$\ln(\tfrac{3}{2})$};
		\node[left] at (0,{1}) {$1$};
		\node[left] at (0,{2}) {$2$};
		\draw[domain=ln(3/2):0.9, variable=\x, MyLineOpt] plot ({\x}, {1});
		\draw[domain=ln(3/2):0.9, variable=\x, MyLineCases] plot ({\x}, {2});
		\draw[domain=0:ln(3/2), variable=\x, MyLineHelp] plot ({\x}, {1});
		\draw[domain=0:ln(3/2), variable=\x, MyLineHelp] plot ({\x}, {2});
		\draw[domain=0:1, variable=\z, MyLineHelp] plot ({ln(3/2)}, {\z});
		%
		\draw[domain=0:1, variable=\z, MyLineHelp] plot ({ln(3/2)}, {\z});
		\draw[domain=1:2, variable=\z, MyLineOpt] plot ({ln(3/2)}, {\z});
 		\draw[MyArrowsLL={0.4}{0.8}, domain=ln(3/2):0.49764, variable=\x, MyLinePush]
		plot ({\x}, {e^(2*\x)/(2*(3/2)^2)*(2*e^(\x) - 2)^2*2});
		\draw[MyArrowsL={0.75}, domain=ln(3/2):0.442, variable=\x, MyLinePush]
		plot ({\x}, {e^(2*\x)/(3)*(2*e^(\x) - 2)^2*2});
		\draw[MyArrowsLL={0.3}{0.75}, domain=ln(1.73):0.655906, variable=\x, MyLinePush]
		plot ({\x}, {e^(2*\x)/((1.73)^2*(2*1.73 - 2)^2*2)*(2*e^(\x) - 2)^2*2});
		\draw[MyArrowsLL={0.2}{0.7}, domain=ln(2):0.81295, variable=\x, MyLinePush]
		plot ({\x}, {e^(2*\x)/(32)*(2*e^(\x) - 2)^2*2});
		\end{tikzpicture}
		\caption{%
			Proof of Proposition~\ref{prop:asymptotic-push}, Case~2.1.1, after the reduction to $ze^y = 2e^x$,
			i.e., $y = x + \ln(2/z)$. We have the constraints $x\geq \ln(3/2)$, $1\leq z\leq 2$.
			On the dotted lines 
			$\XXXasymptoticNominator(x,x + \ln(2/z), z)$ is constant;
			the arrow heads indicate in which direction 
			$\XXXasymptoticDenominator(x,x + \ln(2/z), z)$ decreases and
			the thick lines indicate where the local minima can lie.
		}
		\label{fig:asymptotic-case-2.1.1-no-y}
	\end{figure}

	\begin{figure}[b]
		\begin{tikzpicture}[scale=2,xscale=2]
		\draw[->] (0, 0) -- (0.6,0) node[below] {$x$};
		\draw[->] (0, 0) -- (0,2.2) node[left] {$z$};
		\node[below] at ({ln(3/2)},0) {$\ln(\tfrac{3}{2})$};
		\node[left] at (0,{1}) {$1$};
		\node[left] at (0,{2}) {$2$};
		\draw[domain=0:ln(3/2), variable=\x, MyLineOpt] plot ({\x}, {(2*(2-e^\x)});
		\draw[domain=0:ln(3/2), variable=\x, MyLineCases] plot ({\x}, 2);
		\draw[domain=0:1, variable=\z, MyLineHelp] plot ({ln(3/2)}, {\z});
		%
		\draw[domain=0:ln(3/2), variable=\x, MyLineHelp] plot ({\x}, {1});
		\draw[domain=0:1, variable=\z, MyLineHelp] plot ({ln(3/2)}, {\z});
		\draw[domain=1:2, variable=\z, MyLineCases] plot ({ln(3/2)}, {\z});
		\draw[MyArrowsL={0.55}, domain=0.117783:0.153153, variable=\x, MyLinePush]
		plot ({\x}, {e^(2*\x)/(7.23214)*(4*e^(\x) - 2)*4});
		\draw[MyArrowsLL={0.3}{0.65}, domain=0.223144:0.302501, variable=\x, MyLinePush]
		plot ({\x}, {e^(2*\x)/(12.5)*(4*e^(\x) - 2)*4});
		\draw[MyArrowsL={0.4}, domain=ln(2-1.25/2):ln(3/2), variable=\x, MyLinePush]
		plot ({\x}, {e^(2*\x)/(21.175)*(4*e^(\x) - 2)*4});
		\end{tikzpicture}
		\caption{%
			Proof of Proposition~\ref{prop:asymptotic-push}, Case~2.1.2, after the reduction to $ze^y = 4e^x$,
			i.e., $y = x + \ln(4/z)$.
			We have the constraints $x\leq \ln(3/2)$, $2\geq z\geq 2(2-e^x) \geq 1$.
			On the dotted lines 
			$\XXXasymptoticNominator(x,x + \ln(4/z), z)$ is constant;
			the arrow heads indicate in which direction 
			$\XXXasymptoticDenominator(x,x + \ln(4/z), z)$ decreases.
			The thick line is $z = 2(2-e^x)$.
		}
		\label{fig:asymptotic-case-2.1.2-no-y}
	\end{figure}

\emph{Case 2.1.2}.
	Suppose that $z \geq 1$ and moreover $z \geq 3e^{-y}$, but $x\leq \ln(3/2)$.
	Then in particular $2(2-e^x)\geq 1$ and we have to remember about the constraint $z\geq 2(2-e^x)$.
    We now have:
	\begin{equation*}
	\XXXasymptoticNominator(x,y,z) = x +  y + \ln(ze^y - 2).
	\end{equation*}
	Fix $z$ and consider the implicit curve described by the equation $\XXXasymptoticNominator(x,y,z) = c$ for some constant $c$. 
	Clearly, on this implicit curve
	$x=x(y)$ is determined by the value of $y$.
	Differentiating with respect to $y$ we
	get the condition
	\begin{align}
	\label{eq:derivative-implicit-curve-case-2.c}
	0 = x'(y)  + 1 + \frac{ze^y}{ze^y-2}
	= x'(y) + \frac{2(ze^y - 1
		)}{ze^y-2}.
	\end{align}
	In particular, $x'(y)\leq 0$.
	Using \eqref{eq:derivative-implicit-curve-case-2.c} in the second inequality below we get
	\begin{align*}
	\frac{\partial}{\partial y} \XXXasymptoticDenominator(x(y), y, z) 
	&= x'(y)(2e^{x(y)} - 1) - 1 + z e^y \\
	&= \frac{-2(ze^y -1)(2e^{x(y)} - 1)}{ze^y-2} - 1 + ze^{y}\\
	&= \frac{(ze^y -1)(ze^y - 4e^{x(y)})}{ze^y-2},
	\end{align*}
	which is of the same sign as $ze^y - 4e^{x(y)}$ (since $ze^{y} \geq 3)$.
	It follows that if $y$ increases (and  $x(y)$ decreases),
	then $\XXXasymptoticDenominator(x(y),y,z)$ decreases while $y\leq x(y) + \ln(4/z)$,
	and increases when $y \geq x(y) + \ln(4/z)$. 
	Hence, in Case~2.1.2 and for fixed $z$,  
	in order to minimize $\XXXasymptoticDenominator(x(y), y,z)$ on the implicit curve where $\XXXasymptoticNominator(x(y),y,z)$ is constant, it suffices to consider:
	\begin{itemize}
		\item $x=\ln(2-z/2)$, $y\in[\ln(3/z), \ln(8/z -2)]$,
		\item $ze^y = 4e^x$,
		\item $x = \ln(3/2)$ and $y\geq \ln(6/z)$
	\end{itemize}
		(cf.~Figure~\ref{fig:asymptotic-case-2}).
		The first situation is treated with Lemma~\ref{lem:asymptotic-push-case-0} and the third reduces to the Case~2.1.1 (cf.\ Figure~\ref{fig:asymptotic-case-2}), which in turn was either resolved by auxiliary lemmas or reduced to Case 3.
	It remains to consider the second situation.
	
	If $ze^y = 4e^x$, then $z$ has to be positive, $y = x + \ln(4/z)$,
	and
	\begin{align*}
	\XXXasymptoticNominator(x,x + \ln(4/z), z)   &= 2x + \ln(4e^x-2) + \ln(4/z),\\
	\XXXasymptoticDenominator(x,x+\ln(4/z), z) &= 6e^x - 2x - 2 - \ln(4/z)  - z.
	\end{align*}
	On the implicit curve $\XXXasymptoticNominator(x,x + \ln(4/z), z) = c $,
	$z = z(x)$ is clearly determined by the value of $x$.
	Differentiating  $\XXXasymptoticNominator(x,x + \ln(4/z(x)), z(x))$
	with respect to $x$ we arrive at the condition
	\begin{equation}
	\label{eq:derivative-implicit-curve-case-2.1.2-no-y}
	0 = 2 + \frac{2e^x}{2e^x - 1} -\frac{z'(x)}{z(x)} 
	= \frac{6e^x - 2 }{2e^x - 1} -\frac{z'(x)}{z(x)}.
	\end{equation}
	In particular, $z'(x) \geq 0$.
	Using~\eqref{eq:derivative-implicit-curve-case-2.1.2-no-y}, we obtain
	\begin{align*}
	\frac{\partial }{\partial x} \XXXasymptoticDenominator(x,x+\ln(4/z(x)), z(x)) 
	&= 6e^x - 2  + \frac{z'(x)}{z(x)} - z'(x)
	\\ 
	&=     6e^x - 2  + \frac{(6 e^x - 2)(1-z(x))}{2e^x - 1}\\
	&= \frac{(6 e^x - 2)(2 e^x -z(x))}{2e^x - 1}, 
	\end{align*}
	which is positive since $x >0$ and $z\leq 2$ (due to the standing assumption~\eqref{eq:asymptotic-standing-assumption}).
	Hence, in order to minimize $\XXXasymptoticDenominator(x,x+\ln(4/z(x)), z(x))$
	on the implicit curve where $\XXXasymptoticNominator(x,x + \ln(4/z), z)$ is constant,
	it is best to take $x$ (and thus also $z(x)$) as small as possible.
	We can decrease $x$ and decrease $z(x)$ until we arrive at $z= 2(2-e^x)$  (see Figure~\ref{fig:asymptotic-case-2.1.2-no-y}), where we can apply~Lemma~\ref{lem:asymptotic-push-case-0}.

\emph{Case 2.2}.
	Suppose that $z\geq 1$ and $z\leq 3e^{-y}$.
	Then $\yGain(y,z) = y$ (since $e^{-y} \leq z\leq 3e^{-y}$) 
	and $\xGain(x,z) = x$ if $x\leq \ln(3/2)$ 
	or $\xGain(x,z) = x + \ln(2e^x-2)$ if $x\geq \ln(3/2)$,
	so neither of those quantities depends on $z$. 
	On the other hand, 
	\begin{equation*}
	\XXXasymptoticDenominator(x,y,z) = 2e^x - x- 2 -y + z(e^y-1)
	\end{equation*}
	is increasing in $z$.
	Therefore, in order to maximize $\XXXasymptoticNominator(x,y,z)/\XXXasymptoticDenominator(x,y,z)$ (for fixed $x$, $y$)
	we want to take $z$ as small as possible, i.e.,
	either $z=1$ or $z=2(2-e^x)$
	(depending on which of these two conditions is more restrictive).	
	The former possibility will be treated in Case~3 below (where we assume that $z\in [e^{-y}, 1]$). 
	In the latter situation we can apply~Lemma~\ref{lem:asymptotic-push-case-0}.
	\end{proof}

\subsection{Case 3: \texorpdfstring{$e^{-y} \leq z \leq 1$}{exp(-y) leq z leq 1}}

In this subsection we assume that $e^{-y} \leq z \leq 1$.
Then $ -\ln(z) \in [0, y]$ 
and
\begin{multline*}
\int_0^y |1-z e^s| ds
= \int_0^{-\ln(z)} 1 - z e^s ds + \int_{-\ln(z)}^y  - 1 + z e^s ds\\
= -2\ln(z) - y + z(e^y - 2e^{-\ln(z)} + 1)
= -2\ln(z) - y + z (e^y + 1) -2.
\end{multline*}
Thus
\begin{equation*}
\XXXasymptoticDenominator(x,y,z)
 = 2e^x - x -4  
   - y + z (e^y + 1) -2\ln(z) .
\end{equation*}

%
%
\begin{figure}[t]
	\begin{tikzpicture}[scale=3,xscale=2.5, yscale=1.25]
	\draw[->] (0, 0) -- ({ln((1+sqrt(1+2*46))/2)-0.3+0.05}, 0) node[below] {$x$};
	\draw[->] (0, 0) -- (0,2.5+0.05) node[left] {$y$};
	\node[below] at (0,0) {$0\vphantom{\tfrac{3}{2}}$};
	\node[below] at ({ln(2-\zCaseThree/2)},0) {$\ln(2-\tfrac{z}{2})\vphantom{\tfrac{3}{2}}$};
	\node[below] at ({ln(3/2)},0) {$\ln(\tfrac{3}{2})$};
    \node[below] at ({ln(3)},0) {$\ln(3)\vphantom{\tfrac{3}{2}}$};
   	\node[left] at (0,{ln(1/(\zCaseThree))}) {$\ln(\tfrac1z)$};
   	\node[left] at (0,{ln(2/(\zCaseThree) -1/2)}) {$\ln(\tfrac2z - \tfrac12)$};
	\node[left] at (0,{ln(2/(\zCaseThree))}) {$\ln(\tfrac2z)$};
	\node[left] at (0,{ln(3/(\zCaseThree))}) {$\ln(\tfrac3z)$};
	\draw[domain=ln(1/(\zCaseThree)):2.5,variable=\y, MyLineCases] plot ({ln(2-\zCaseThree/2)}, {\y});
	\draw[domain=0:ln(1/(\zCaseThree)),variable=\y, MyLineHelp] plot ({ln(2-\zCaseThree/2)}, {\y});
	\draw[domain=0:ln(3/(\zCaseThree)), variable=\y, MyLineHelp] plot ({ln(3/2)}, {\y});
	\draw[domain=ln(2-\zCaseThree/2):{ln((1+sqrt(1+2*46))/2)-0.3}, variable=\x, MyLineCases] plot (\x, {ln(1/(\zCaseThree))});
	\draw[domain=0:ln(2-\zCaseThree/2), variable=\x, MyLineHelp] plot (\x, {ln(1/(\zCaseThree))});
	\draw[domain=0:ln(2-\zCaseThree/2), variable=\x, MyLineHelp] plot (\x, {ln(2/(\zCaseThree) -1/2)});
	\draw[domain=ln(2-\zCaseThree/2):{ln((1+sqrt(1+2*46))/2)-0.3}, variable=\x, MyLineCases] plot (\x, {ln(3/(\zCaseThree))});
	\draw[domain=0:ln(2-\zCaseThree/2), variable=\x, MyLineHelp] plot (\x, {ln(3/(\zCaseThree))});
	\draw[domain=ln(2-\zCaseThree/2):ln(3), variable=\x, MyLineOpt] plot ({\x}, {\x - ln(\zCaseThree)});
	\draw[domain=0:ln(2-\zCaseThree/2), variable=\x, MyLineHelp] plot ({\x}, {\x - ln(\zCaseThree)});
	\draw[domain=ln(2-\zCaseThree/2):ln(3)-0.1, variable=\x, MyLineOpt] plot ({\x}, {\x + ln(2) - ln(\zCaseThree)});
	\draw[domain=0:ln(2-\zCaseThree/2), variable=\x, MyLineHelp] plot ({\x}, {\x + ln(2) - ln(\zCaseThree)});
	\draw[domain=ln(1/(\zCaseThree)):ln(2/(\zCaseThree) -1/2), variable=\y, MyLineOpt] plot ({ln(2-(\zCaseThree)/2)}, {\y});
	\draw[domain=ln(3/(\zCaseThree)):ln(4/(\zCaseThree) -1), variable=\y, MyLineOpt] plot ({ln(2-(\zCaseThree)/2)}, {\y});			
	%
	\draw[MyArrowsL={0.5},
	domain=ln(2-(\zCaseThree)/2):ln((1+sqrt(1+2*3.5))/2), variable=\x, MyLinePush]
	plot ({\x}, { -\x - ln(2*e^\x - 2) + ln(2-\zCaseThree) + ln(3.5/(\zCaseThree)/(2-\zCaseThree))});
	\draw[MyArrowsL={0.5},
	domain=ln(2-(\zCaseThree)/2):ln((1+sqrt(1+2*(4-\zCaseThree)*(4-\zCaseThree)*(2-\zCaseThree)/4))/2), variable=\x, MyLinePush]
	plot ({\x}, { -\x - ln(2*e^\x - 2) + ln(2-\zCaseThree) + ln((4-\zCaseThree)*(4-\zCaseThree)*(2-\zCaseThree)/4/(\zCaseThree)/(2-\zCaseThree))});
	\draw[MyArrowsL={0.5},
	domain=ln(2-(\zCaseThree)/2):ln((1+sqrt(1+2*3.5))/2), variable=\x, MyLinePush]
	plot ({\x}, { -\x - ln(2*e^\x - 2) + ln(2-\zCaseThree) + ln(3.5/(\zCaseThree)/(2-\zCaseThree))});
	\draw[MyArrowsRLL={0.15}{0.475}{0.75},
	domain=ln(2-(\zCaseThree)/2):ln((1+sqrt(1+2*6))/2), variable=\x, MyLinePush]
	plot ({\x}, { -\x - ln(2*e^\x - 2) + ln(2-\zCaseThree) + ln(6/(\zCaseThree)/(2-\zCaseThree))});
    \draw[MyArrowsRLL={0.25}{0.55}{0.8}, domain=ln((1+sqrt(1+2*(3*(4-\zCaseThree)*(2-\zCaseThree)/2)/3))/2):ln((1+sqrt(1+2*3*(4-\zCaseThree)*(2-\zCaseThree)/2))/2), variable=\x, MyLinePush]
    plot ({\x}, { -\x - ln(2*e^\x - 2) + ln(2-\zCaseThree) + ln(3*(4-\zCaseThree)*(2-\zCaseThree)/2/(\zCaseThree)/(2-\zCaseThree))});
    \draw[MyArrowsRLL={0.2}{0.5}{0.8}, domain=ln((1+sqrt(1+2*10/3))/2):ln((1+sqrt(1+2*10))/2), variable=\x, MyLinePush]
    plot ({\x}, { -\x - ln(2*e^\x - 2) + ln(2-\zCaseThree) + ln(10/(\zCaseThree)/(2-\zCaseThree))});		
    \draw[MyArrowsRLL={0.15}{0.45}{0.8}, domain=ln((1+sqrt(1+2*((4-\zCaseThree)^2*(2-\zCaseThree)^2/2)/3))/2):ln((1+sqrt(1+2*(4-\zCaseThree)^2*(2-\zCaseThree)^2/2))/2), variable=\x, MyLinePush]
    plot ({\x}, { -\x - ln(2*e^\x - 2) + ln(2-\zCaseThree) + ln((4-\zCaseThree)^2*(2-\zCaseThree)^2/2/(\zCaseThree)/(2-\zCaseThree))});
    \draw[MyArrowsRLL={0.1}{0.4}{0.7}, domain=ln((1+sqrt(1+2*18/3))/2):ln((1+sqrt(1+2*18))/2), variable=\x, MyLinePush]
    plot ({\x}, { -\x - ln(2*e^\x - 2) + ln(2-\zCaseThree) + ln(18/(\zCaseThree)/(2-\zCaseThree))});
    \draw[MyArrowsRLL={0.05}{0.35}{0.6}, domain=ln((1+sqrt(1+2*26/3))/2):ln((1+sqrt(1+2*46))/2)-0.3, variable=\x, MyLinePush]
   plot ({\x}, { -\x - ln(2*e^\x - 2) + ln(2-\zCaseThree) + ln(26/(\zCaseThree)/(2-\zCaseThree))});
   \draw[MyArrowsLL={0.35}{0.65}, domain=ln((1+sqrt(1+2*36/3))/2):ln((1+sqrt(1+2*46))/2)-0.3, variable=\x, MyLinePush]
   plot ({\x}, { -\x - ln(2*e^\x - 2) + ln(2-\zCaseThree) + ln(36/(\zCaseThree)/(2-\zCaseThree))});
     \draw[MyArrowsL={0.65},  domain=ln((1+sqrt(1+2*46/3))/2):ln((1+sqrt(1+2*46))/2)-0.3, variable=\x, MyLinePush]
     plot ({\x}, { -\x - ln(2*e^\x - 2) + ln(2-\zCaseThree) + ln(46/(\zCaseThree)/(2-\zCaseThree))}); 
	\draw[MyArrowsL={0.5}, domain=ln(2-\zCaseThree/2):ln((1+sqrt(1+2*10/3))/2), variable=\x, MyLinePush]
	plot ({\x}, {ln((1+sqrt(1+\zCaseThree*(2-\zCaseThree)*(10/(\zCaseThree)/(2-\zCaseThree))*e^(-\x)/(2*e^\x -2)))/(\zCaseThree))});
	\draw[MyArrowsL={0.5},
	 domain=ln(2-\zCaseThree/2):ln((1+sqrt(1+2*\zCaseThree*((4-\zCaseThree)^2*(2-\zCaseThree)^2/2/(\zCaseThree))/3))/2), variable=\x, MyLinePush]
	plot ({\x}, {ln((1+sqrt(1+\zCaseThree*(2-\zCaseThree)*((4-\zCaseThree)^2*(2-\zCaseThree)/2/(\zCaseThree))*e^(-\x)/(2*e^\x -2)))/(\zCaseThree))});
	\draw[MyArrowsRL={0.1}{0.5}, domain=ln(2-\zCaseThree/2):ln((1+sqrt(1+2*18/3))/2), variable=\x, MyLinePush]
	plot ({\x}, {ln((1+sqrt(1+\zCaseThree*(2-\zCaseThree)*(18/(\zCaseThree)/(2-\zCaseThree))*e^(-\x)/(2*e^\x -2)))/(\zCaseThree))});
	\draw[MyArrowsRL={0.15}{0.5}, domain=ln(2-\zCaseThree/2):ln((1+sqrt(1+2*26/3))/2), variable=\x, MyLinePush]
	plot ({\x}, {ln((1+sqrt(1+\zCaseThree*(2-\zCaseThree)*(26/(\zCaseThree)/(2-\zCaseThree))*e^(-\x)/(2*e^\x -2)))/(\zCaseThree))});
	\draw[MyArrowsRL={0.175}{0.5},
	domain=ln(2-\zCaseThree/2):ln(3), variable=\x, MyLinePush]
	plot ({\x}, {ln((1+sqrt(1+\zCaseThree*(2-\zCaseThree)*(36/(\zCaseThree)/(2-\zCaseThree))*e^(-\x)/(2*e^\x -2)))/(\zCaseThree))});	
	\draw[MyArrowsRL={0.2}{0.5}, domain=ln(2-\zCaseThree/2):ln((1+sqrt(1+2*46/3))/2), variable=\x, MyLinePush]
	plot ({\x}, {ln((1+sqrt(1+\zCaseThree*(2-\zCaseThree)*(46/(\zCaseThree)/(2-\zCaseThree))*e^(-\x)/(2*e^\x -2)))/(\zCaseThree))});
	\node[left] at ({ln((1+sqrt(1+2*46))/2)-0.3},1.95) {Case 3.1};
	\node[left, fill=white] at ({ln((1+sqrt(1+2*46))/2)-0.3},1.7) {Case 3.2};
	\end{tikzpicture}
	\caption{%
		Proof of Proposition~\ref{prop:asymptotic-push}, 
		Case 3, where $3e^{-y} \leq z \leq 1$.
		On the dotted lines 
		$\XXXasymptoticNominator(x,y(x), z)$ is constant;
		the arrow heads indicate in which direction 
		$\XXXasymptoticDenominator(x,y(x),z)$ decreases and
		the thick lines indicate where the local minima can lie.
	}
	\label{fig:asymptotic-case-3}
\end{figure}

\begin{proof}[Proof of Proposition~\ref{prop:asymptotic-push} in Case 3]
We consider two subcases.

\emph{Case 3.1}
	Suppose that $e^{-y} \leq z \leq 1$ and moreover $z\geq 3e^{-y}$,
	i.e., $3e^{-y} \leq z \leq 1$.
	The standing assumption~\eqref{eq:asymptotic-standing-assumption} implies in particular that
	$x\geq \ln(3/2)$ and
	\begin{equation*}
	\XXXasymptoticNominator(x,y,z) = x + \ln(2e^x-2) - \ln(2-z) + y + \ln(ze^{y} - 2) .
	\end{equation*}
	Fix $z$ and consider the implicit curve described by the equation $\XXXasymptoticNominator(x,y,z) = c$ for some constant $c$. 
	Clearly, on this implicit curve
	$y=y(x)$ is determined by the value of $x$.
	Differentiating $\XXXasymptoticNominator(x,y(x),z)$ with respect to $x$ we
	get the condition
	\begin{align}
	\label{eq:derivative-implicit-curve-case-3.b}
	0 = 1 + \frac{2e^{x}}{2e^{x} - 2} + y'(x)\bigl(1 + \frac{ze^{y(x)}}{ze^{y(x)} - 2}\bigr) =   \frac{2e^{x} - 1}{e^{x} - 1} + y'(x) \frac{2(ze^{y(x)} -1)}{ze^{y(x)} - 2}.
	\end{align}
	In particular, $y'(x)\leq 0$.
	Using \eqref{eq:derivative-implicit-curve-case-3.b} in the second inequality below we get
	\begin{align*}
	\frac{\partial}{\partial x} \XXXasymptoticDenominator(x, y(x), z) 
	&= 2e^{x} - 1  + y'(x)(-1  + z e^{y(x)}) \\
	&= 2e^{x} - 1 - \frac{(2e^x -1)(ze^{y(x)} - 2)}{2e^x - 2}\\
	&= \frac{(2e^x - 1)(2e^x - z e^{y(x)})}{2e^x  -2},
	\end{align*}
	which is of the same sign as $2e^x - ze^{y(x)}$.
	It follows that if $x$ increases (and consequently $y(x)$ decreases),
	then $\XXXasymptoticDenominator(x,y(x),z)$ decreases while $x \leq y(x) + \ln(z/2)$,
	and then increases when $x \geq y(x) + \ln(z/2)$.
	Hence, in Case~3.1 and for fixed $z$,  
	in order to minimize $\XXXasymptoticDenominator(x, y(x),z)$ on the implicit curve where $\XXXasymptoticNominator(x,y(x),z)$ is constant, it suffices to consider:
	\begin{itemize}
		\item $x=\ln(2-z/2)$, $y\in[\ln(3/z), \ln(4/z -1)]$,
		\item $ze^y = 2e^x$, $x\geq \ln(3)$.
	\end{itemize}
(cf.~Figure~\ref{fig:asymptotic-case-3}).
			The first situation is treated with Lemma~\ref{lem:asymptotic-push-case-0}. 
			It remains to deal with the case when $ze^y = 2e^x$.
			Under this assumption we have  $y = x + \ln(2/z)$, and
				\begin{align*}
				\XXXasymptoticNominator(x,x + \ln(2/z), z)   &= 2x + 2\ln(2e^x-2) - \ln(2-z) + \ln(2) - \ln(z),\\
				\XXXasymptoticDenominator(x,x+\ln(2/z), z) &= 4e^x - 2x - 4 -\ln(2)  + z -\ln(z).
				\end{align*}
				On the implicit curve $\XXXasymptoticNominator(x,x + \ln(2/z), z) = c $,
				$x=x(z)$ is determined by the value of $z$
				(since $x\mapsto x +\ln(2e^x -2)$ is increasing).
				Differentiating  $\XXXasymptoticNominator(x(z),x(z) + \ln(2/z), z)$
				with respect to $z$ we arrive at the condition
				\begin{equation}
				\label{eq:derivative-implicit-curve-case-3.1-no-y}
				0 = \Bigl(2 + \frac{2e^x}{e^x - 1}\Bigr) x'(z) + \frac{1}{2-z} - \frac{1}{z} = \frac{4e^x - 2 }{e^x - 1}x'(z) +\frac{2(z-1)}{(2-z)z}.
				\end{equation}
				In particular, $x'(z) \geq 0$ (since $0<z\leq 1$).
				Using~\eqref{eq:derivative-implicit-curve-case-3.1-no-y}, we can write
				\begin{align*}
				\frac{\partial }{\partial z} \XXXasymptoticDenominator(x(z),x(z)+\ln(2/z), z) 
				&=(4e^{x(z)} -2) x'(z) + 1 - \frac{1}{z}\\
				&=   \frac{2(1-z)(e^{x(z)} -1)}{(2-z)z} + \frac{z-1}{z}\\
				&= \frac{(1-z)( z - 2(2-e^{x(z)}))}{(2-z)z}, 
				\end{align*}
				which is positive since $0<z\leq 1$ and  
				we work under the standing assumption~\eqref{eq:asymptotic-standing-assumption}.
				Hence, in order to minimize $\XXXasymptoticDenominator(x,x+\ln(2/z(x)), z(x))$
				on the implicit curve where $\XXXasymptoticNominator(x,x + \ln(2/z), z)$ is constant,
				it is best to take $z$ (and thus also $x(z)$) as small as possible.
				We can decrease $z$ and decrease $x(z)$ until we arrive at the situation where $2(2-e^x)  = z$ (see Figure~\ref{fig:asymptotic-case-3.1-no-y}).
				We can thus apply Lemma~\ref{lem:asymptotic-push-case-0}.
			    This ends the proof in Case~3.1.

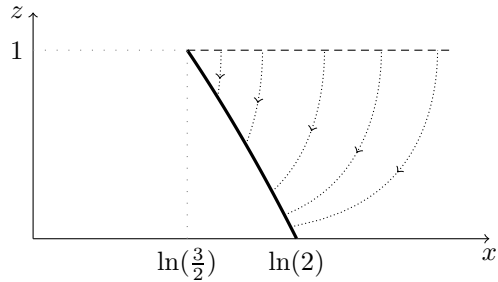
\begin{figure}[b]
	\begin{tikzpicture}[scale=2.5,xscale=2]
	\draw[->] (0, 0) -- (1.2,
	 0) node[below] {$x$};
	\draw[->] (0, 0) -- (0,1.2) node[left] {$z$};
	\node[below] at ({ln(3/2)},0) {$\ln(\tfrac{3}{2})$};
	\node[below] at ({ln(2)},0) {$\ln(2)\vphantom{\tfrac{3}{2}}$};
	\node[left] at (0,{1}) {$1$};
	\draw[domain=ln(3/2):ln(2), variable=\x, MyLineOpt] plot ({\x}, {(2*(2-e^\x)});
	\draw[domain=ln(3/2):1.1, variable=\x, MyLineCases] plot ({\x}, {1});
	\draw[domain=0:ln(3/2), variable=\x, MyLineHelp] plot ({\x}, {1});
	\draw[domain=0:1, variable=\z, MyLineHelp] plot ({ln(3/2)}, {\z});
	 \draw[MyArrowsL={0.5}, domain=2*(2- 13/8):1, variable=\z, MyLinePush]
	 plot ({ln(  (1 + sqrt(1+sqrt(   2*845/96*(2-\z)*\z  ))  )/2)}, {\z});
	 \draw[MyArrowsL={0.5}, domain=2*(2- 7/4):1, variable=\z, MyLinePush]
	 plot ({ln(  (1 + sqrt(1+sqrt(   2*147/8*(2-\z)*\z  ))  )/2)}, {\z});
	 \draw[MyArrowsL={0.5}, domain=2*(2- 15/8):1, variable=\z, MyLinePush]
	 plot ({ln(  (1 + sqrt(1+sqrt(   2*1575/32*(2-\z)*\z  ))  )/2)}, {\z});
	 \draw[MyArrowsL={0.5}, domain=0.12512:1, variable=\z, MyLinePush]
	 	 plot ({ln(  (1 + sqrt(1+sqrt(   2*112.5*(2-\z)*\z  ))  )/2)}, {\z});
    0.06208975195438
    \draw[MyArrowsL={0.5}, domain= 0.06208975:1, variable=\z, MyLinePush]
    plot ({ln(  (1 + sqrt(1+sqrt(   2*242*(2-\z)*\z  ))  )/2)}, {\z});
	\end{tikzpicture}
	\caption{%
		Proof of Proposition~\ref{prop:asymptotic-push}, Case~3.1, after the reduction to $ze^y = 2e^x$,
		i.e., $y = x + \ln(2/z)$. We have the constraints $x\geq \ln(3/2)$, $2(2-e^x) \leq z \leq 1$, $z>0$.
				On the dotted lines 
				$\XXXasymptoticNominator(x,x + \ln(2/z), z)$ is constant;
				the arrow heads indicate in which direction 
			$\XXXasymptoticDenominator(x,x + \ln(2/z), z)$ decreases.
			The thick line is $z = 2(2-e^x)$. 
	}
	\label{fig:asymptotic-case-3.1-no-y}
\end{figure}

\emph{Case 3.2}
	Suppose that $e^{-y} \leq z \leq 1$ and moreover $z\leq 3e^{-y}$.
	The standing assumptions imply in particular that
	$x\geq \ln(3/2)$ and
	\begin{equation*}
	\XXXasymptoticNominator(x,y,z) = x + \ln(2e^x-2) - \ln(2-z) + y.	
	\end{equation*}
	Fix $z$ and consider the implicit curve described by the equation $\XXXasymptoticNominator(x,y,z) = c$ for some constant $c$. 
	Clearly, on this implicit curve
	$y=y(x)$ is determined by the value of $x$.
	Differentiating with respect to $x$ we
	get the condition
	\begin{align}
	\label{eq:derivative-implicit-curve-case-3.a}
	0 = 1 + \frac{2e^{x}}{2e^{x} - 2} + y'(x) =   \frac{2e^{x} - 1}{e^{x} - 1} + y'(x).
	\end{align}
	In particular, $y'(x)\leq 0$.
	Using \eqref{eq:derivative-implicit-curve-case-3.a} in the second inequality below we get
	\begin{align*}
	\frac{\partial}{\partial x} \XXXasymptoticDenominator (x, y(x), z) 
	&= 2e^{x} - 1  + y'(x)(-1  + z e^{y(x)}) \\
	&= 2e^{x} - 1 + \frac{(2e^x -1)(1-ze^{y(x)})}{e^x - 1}\\
	&= \frac{(2e^x - 1)(e^x - z e^{y(x)})}{e^x  -1},
	\end{align*}
	which is of the same sign as $e^x - ze^{y(x)}$.
	It follows that if $x$ increases (and consequently $y(x)$ decreases),
	then $\XXXasymptoticDenominator(x,y(x),z)$ decreases while $x \leq y(x) + \ln(z)$,
	and then increases when $x \geq y(x) + \ln(z)$.
	Hence, in Case~3.2 and for fixed $z$,  
	in order to minimize $\XXXasymptoticDenominator(x, y(x),z)$ on the implicit curve where $\XXXasymptoticNominator(x,y(x),z)$ is constant, it suffices to consider:
	\begin{itemize}
		\item $x=\ln(2-z/2)$, $y\in[\ln(1/z), \ln(2/z -1/2)]$,
		\item $ze^y = e^x$, $x\in [\ln(2-z/2), \ln(3)]$,
		\item $y = \ln(3/z)$, $x  \geq \ln(3)$
	\end{itemize}
	(see Figure~\ref{fig:asymptotic-case-3}).
	The first situation is treated with Lemma~\ref{lem:asymptotic-push-case-0}.
	The third situation reduces to the already fully solved Case~3.1 (cf.\ Figure~\ref{fig:asymptotic-case-3}).
	It remains to deal with the case when $ze^y =  e^x$.
	Under this assumption we have  $y = x  -\ln(z)$, and
	\begin{align*}
	\XXXasymptoticNominator(x,x - \ln(z), z)   &= 2x + \ln(2e^x-2) - \ln(2-z)  - \ln(z),\\
	\XXXasymptoticDenominator(x,x -\ln(z), z) &= 3e^x - 2x - 4   + z -\ln(z).
	\end{align*}
	On the implicit curve $\XXXasymptoticNominator(x,x - \ln(z), z) = c $,
	$x=x(z)$ is clearly determined by the value of $z$
	(since $x\mapsto x +\ln(2e^x -2)$ is increasing).
	Differentiating  $\XXXasymptoticNominator(x(z),x(z) - \ln(z), z)$
	with respect to $z$ we arrive at the condition
	\begin{equation}
	\label{eq:derivative-implicit-curve-case-3.2-no-y}
	0 = \Bigl(2 + \frac{e^x}{e^x - 1}\Bigr) x'(z) + \frac{1}{2-z} - \frac{1}{z} = \frac{3e^x - 2 }{e^x - 1}x'(z) +\frac{2(z-1)}{(2-z)z}.
	\end{equation}
	In particular, $x'(z) \geq 0$ (since $0<z\leq 1$).
	Using~\eqref{eq:derivative-implicit-curve-case-3.2-no-y} we can write
	\begin{align*}
	\frac{\partial }{\partial z} \XXXasymptoticDenominator(x(z),x(z)-\ln(z), z) 
	&=(3e^{x(z)} -2) x'(z) + 1 - \frac{1}{z}\\
	&=   \frac{2(1-z)(e^{x(z)} -1)}{(2-z)z} + \frac{z-1}{z}\\
	&= \frac{(1-z)\bigl( z - 2(2-e^{x(z)})\bigr)}{(2-z)z}, 
	\end{align*}
	which is positive since $0<z\leq 1$ 
	and we work under the standing assumption~\eqref{eq:asymptotic-standing-assumption}.
	Hence, in order to minimize $\XXXasymptoticDenominator(x(z),x(z)-\ln(z), z)$
	on the implicit curve where $\XXXasymptoticNominator(x(z),x(z) - \ln(z), z)$ is constant,
	it is best to take $z$ (and thus also $x(z)$) as small as possible.
	We can decrease $z$ and decrease $x(z)$ until we arrive at the situation where $2(2-e^x)  = z$
	 (the situation is similar to the one depicted in Figure~\ref{fig:asymptotic-case-3.1-no-y}).
	We can thus apply Lemma~\ref{lem:asymptotic-push-case-0}.
	This ends the proof in Case~3.2 and the consideration of Case~3, and  thus the whole proof of Proposition~\ref{prop:asymptotic-push}.
\end{proof}


\bibliographystyle{amsplain}
\bibliography{literature}
\end{document}